\documentclass[11pt]{amsart}
\usepackage{amsmath}
\usepackage{amssymb}
\usepackage{amsfonts}
\usepackage{amsthm}
\usepackage{enumerate}
\usepackage[notcite,notref]{showkeys}
\numberwithin{equation}{section}

\newtheorem{theorem}{Theorem}[section]
\newtheorem{lemma}[theorem]{Lemma}
\newtheorem{proposition}[theorem]{Proposition}
\newtheorem{corollary}[theorem]{Corollary}

\theoremstyle{definition}
\newtheorem{definition}[theorem]{Definition}
\newtheorem{example}[theorem]{Example}

\theoremstyle{remark}
\newtheorem{remark}[theorem]{Remark}

\numberwithin{equation}{section}

\renewcommand{\Re}{\operatorname{Re}}

\newcommand{\cF}{\mathcal{F}}
\newcommand{\cS}{\mathcal{S}}
\newcommand{\suchthat}{:}
\newcommand{\ran}{\operatorname{ran}}
\newcommand{\dom}{\operatorname{dom}}
\renewcommand{\O}{\operatorname{O}}

\newcommand{\N}{\mathbb{N}}
\newcommand{\Z}{\mathbb{Z}}

\newcommand{\R}{\mathbb{R}}
\newcommand{\C}{\mathbb{C}}

\newcommand{\ud}{\mathrm{d}}

\newcommand{\Ce}{\mathrm{C}}

\begin{document}



\title[Fine scales of decay of operator semigroups]
{Fine scales of decay of operator semigroups}
\author{Charles J.K. Batty}
\address{St. John's College, Oxford OX1 3JP, United Kingdom}
\email{charles.batty@sjc.ox.ac.uk}
\author{Ralph Chill}
\address{Technische Universit\"at Dresden,
Institut f\" ur Analysis,
01062 Dresden,
Germany}
\email{ralph.chill@tu-dresden.de}
\author{Yuri Tomilov}
\address{\textbf{}
Institute of Mathematics\\
Polish Academy of Sciences\\
\'Sniadeckich 8\\
00-956 Warszawa, Poland
}
\email{ytomilov@impan.pl}

\date{\today}

\thanks{The research described in this paper was supported by the Leverhulme Trust and by the EPSRC grant EP/J010723/1.  The third author was also partially supported by the NCN grant DEC-2014/13/B/ST1/03153 and the EU grant ``AOS'', FP7-PEOPLE-2012-IRSES, No 318910.
}


\begin{abstract}
Motivated by potential applications to partial differential equations,
we develop a theory of fine scales of decay rates for operator semigroups.
The theory contains, unifies, and extends several notable results in the literature on decay of operator semigroups and yields a number of new ones. Its core is a new operator-theoretical method of deriving rates of decay
combining ingredients from functional calculus, and complex, real and harmonic analysis.  It also leads to several results of independent interest. 
\end{abstract}
\maketitle

\tableofcontents

\section {Introduction}

\subsection{Historical background, motivation and sample results}
The study of stability of solutions of the abstract Cauchy problem
\begin{equation*}\label{Cauchypr}
\left\{ \begin{array}{ll}
\dot{u} (t) + A u(t) = 0, & \quad t \ge 0 ,\\[2mm]
u(0) = x , & x\in X,
\end{array} \right.
\end{equation*}
where $-A$ is the generator of a
$C_0$-semigroup $(T(t))_{t \ge 0}$ on a Banach space $X$, is a classical subject
of functional analysis having numerous applications to partial
differential equations. Namely, asymptotic stability (or simply: stability), that
is convergence of all orbits of $(T(t))_{t \ge 0}$  to zero, and
exponential stability, that is, convergence of all orbits of
$(T(t))_{t \ge 0}$ to zero with exponential rate, are two of the main
building blocks of stability theory. Since the resolvent of the generator is often easier to compute than the semigroup, it is traditional and efficient to use the resolvent when dealing with both kinds of stability, and actually a number of resolvent criteria for stability are known in this context (see \cite[Chapter 5]{ABHN01}, \cite{ChTo07} and \cite{Ne96}).

Recently, the study of various PDEs revealed that the resolvent can also be used successfully in order to treat intermediate rates of convergence, thus distinguishing and quantifying fine modes of convergence to zero. This has become especially transparent for the damped wave equation
\begin{equation} \label{wave}
\begin{split}
 u_{tt} + a(x) u_t -\Delta u &= 0 \;\,\text{ in } \R_+ \times
 M,\\
 u & = 0 \;\,\text{ in } \R_+ \times \partial M , \\
 u(0,\cdot ) & = u_0 \text{ in } M , \\
 u_t (0,\cdot ) & = u_1 \text{ in } M ,
\end{split}
\end{equation}
which is one of the basic models in control theory. Here $M$ is a smooth, compact, connected Riemannian manifold with boundary, $a \in L^{\infty}(M)$ and $a \geq 0$. The wave equation can be rewritten as a first order Cauchy problem in the Hilbert space $X = H^1_0(M)\times L^2(M)$ with $A$ defined by
\begin{eqnarray*}
{\rm dom} (A)&=& (H^2(M)\cap H^1_0(M))\times H^1_0(M),\\
A &=& \begin{pmatrix} O & -I\\
                   -\Delta & a\end{pmatrix}.
\end{eqnarray*}
The operator $A$ is invertible, and $-A$ generates a non-analytic contraction semigroup $(T(t))_{t\geq 0}$. Since the natural norm on the Hilbert space $X$ corresponds to the energy of the system, any estimate for the rate of decay of the semigroup is an estimate for the rate of decay of the energy of the system.

Recall that a semigroup is exponentially stable if and only if $\| T(t)\| = \O(r(t))$ with $\lim_{t\to\infty} r(t)=0$. Hence, one cannot expect a uniform rate of decay for all solutions if the semigroup is only stable but not exponentially stable. Nevertheless, one may have uniform estimates of rates of decay for a dense set of initial values. The study of the damped wave equation and various similar PDEs has revealed that the resolvent can be used successfully to obtain quantitative rates of convergence of the form
\begin{equation}\label{rater}
\|T(t)A^{-1}\| = \O \left(r(t) \right), \qquad t \to \infty ,
\end{equation}
that is, uniform rates for {\em classical} solutions. 

In his pioneering work \cite{Le96}, Lebeau  discovered
that the spectrum of $A$ is contained in the strip $\{\lambda \in
\mathbb C: 0 < \Re\lambda \leq 2\|a\|_{\infty} \}$, the resolvent
$(\lambda+A)^{-1}$ admits an exponential estimate on the imaginary
axis and this forces the energy of solutions of \eqref{wave} to
decay at least with rate
$$
r(t)=\frac{\log (3+\log(3+t))}{\log(3+t)} \,, \qquad t \to \infty.
$$
His method was extended subsequently by Burq \cite{Bu98} and Lebeau and Robbiano \cite{LeRo97} to cover similar situations involving local energy decay for \eqref{wave} where $a=0$ and $M$ is a non-compact manifold
(``scattering in an exterior domain''), and global energy decay
for \eqref{wave} with Neumann boundary conditions, respectively. Moreover Lebeau's estimate was improved in \cite{Bu98} and \cite{LeRo97} 
to $1/\log (3+t)$. These results gave rise to a
number of  papers treating the rates of decay of solutions to PDEs
by abstract semigroup methods. Among them, we would like to
mention \cite{BaEnPrSchn06}, \cite{BoPe06}, \cite{BuHi07}, \cite{Ch09},
\cite{ChSchVaWu13}, \cite{LiRa05},  and especially the very recent
paper \cite{AnLe12} as samples. This last paper includes a
complete historical account, and a detailed discussion of
damped wave equations in the resolvent context.

The study of rates was put into the setting of Tauberian theorems for Laplace transforms in \cite{BaDu08}, by regarding the resolvent as the Laplace transform of the semigroup and partially inverting the Laplace transform.  This approach by pure complex analysis unified a number of known results (including those of Burq, Lebeau and others), improved some of them and applied to arbitrary growth of resolvents.  In particular, the following result was obtained. 

\begin{theorem}\label{chdu}
Let $(T(t))_{t \ge 0}$ be a bounded $C_0$-semigroup on a Banach space $X$ with generator $-A$.  Assume that $\sigma(A)\cap i\mathbb R$ is empty, and define
\begin{eqnarray*}
M(s) &=& \sup \{\|(ir+A)^{-1}\| : |r| \le s\}, \qquad s\ge 1, \\
M_{\log}(s) &=& M(s)\left( \log(1 + M(s)) + \log(1+s)\right).
\end{eqnarray*}
Then there exist positive constants $C,C',c,c'$ such that
\begin{equation} \label{genestintro}
\frac{c'} {M^{-1}(C't)} \le \|T(t)A^{-1}\| \le \frac{C}
{M_{\log}^{-1}(ct)}
\end{equation}
for all sufficiently large $t.$
\end{theorem}

Note that by \cite[Proposition 1.3]{BaDu08} any decay of
$\|T(t)A^{-1}\|$ to zero as $ t \to \infty$ implies that $\sigma(A) \cap i\R$ is empty. However, in many situations the spectrum approaches the
imaginary axis asymptotically forcing the function $M$ to grow at infinity. So we shall say (slightly loosely) that if $ M(s) \to \infty$ as $s \to
\infty$, then the resolvent, restricted to the imaginary axis, 
has a \emph{singularity at infinity}.

The upper estimate in \eqref{genestintro} was established by examining a function-theoretic method first used  in unquantified form in \cite{ArBa88} (see also \cite{Ba94}).  The gap between the lower and upper estimates is, in general, of ``logarithmic size'' and
 any improvement of \eqref{genestintro} should fall within that gap. (Note, however, that if $M$ grows exponentially, then the lower and upper estimates are of the same order.) It was conjectured in \cite{BaDu08} that the gap could be
bridged in the case of Hilbert spaces, but one cannot
expect rates better than $\big(M_{\log}^{-1}(ct)\big)^{-1}$ for
general Banach spaces. This conjecture was partially settled in
\cite{BoTo10}. It was proved  in \cite{BoTo10} that if $M$ grows
polynomially then it is possible to obtain a characterization of decay of $\|T(t)A^{-1}\|$ which is \emph{optimal} in the
sense that $M_{\log}^{-1}$ in \eqref{genestintro} is replaced by
$M^{-1}$. Namely the following theorem holds.

\begin{theorem}\label{borto}
Let $(T(t))_{t \ge 0}$ be a bounded $C_0$-semigroup on a Hilbert space
$X$, with generator $-A.$  Assume that $\sigma(A)\cap i\mathbb R$
is empty, and fix  $\alpha >0$.  Then the following statements
are equivalent:
\begin{enumerate}[\rm(i)]
\item $\|(is+A)^{-1}\|= \O(|s|^\alpha), \qquad |s| \to \infty.$
\item $\|T(t)A^{-1}\|= \O(t^{-1/\alpha}), \qquad t \to \infty.$
\end{enumerate}
\end{theorem}

Moreover, it was shown in \cite{BoTo10} that the logarithmic gap
in \eqref{genestintro} is unavoidable for semigroups on Banach spaces,
even if the growth of the resolvent is polynomial as in Theorem
\ref{borto}(i).

The result above has already found a
number of applications to the study of concrete PDEs, see e.g.\
\cite{AlCaGu11},  \cite{AmFeNi12}, \cite{AmMeReVa12},
\cite{BaMeNiWe11}, \cite{ElTo12}, \cite{FaMo12}, \cite{FeRi12},
\cite{Gr111}, \cite{Gr12a}, \cite{Gr122}, \cite{HaLi12},
\cite{Ni12}, \cite{NiVa12}, \cite{SaAlRi12}, \cite{Te12}. In view
of its importance, Theorem \ref{borto} (and Theorem \ref{chdu})
raised the natural and important problem of obtaining similar
results for growth scales which are finer than polynomial ones, or showing
that for other scales results of this kind cannot be true.

Recall that the optimal function-theoretical analogue of Theorem \ref{chdu} has a logarithmic correction term as in the right-hand side of \eqref{genestintro}  \cite[Theorem 3.8]{BoTo10}. Thus,  one cannot expect,
 in general, that a purely function-theoretical approach could produce
 a version of Theorem \ref{borto} for more general resolvent bounds, and
the need for an alternative approach to the problem of
sharpening \eqref{genestintro} becomes apparent. Such an
approach based on operator theory is proposed in
this paper. Note that Theorem \ref{borto} itself
  is proved in \cite{BoTo10} by means of an auxiliary operator construction. The framework of the present
  paper is different, more general, and in a sense more transparent.

The polynomial scale has a number of special (e.g.\ algebraic)
properties making it comparatively amenable to an
operator-theoretical approach. For other scales where these
properties are not available, our approach must be much more
involved and the task of finding appropriate scales is rather
non-trivial.  The scale which is closest to the polynomial one is the finer scale of regularly varying functions, which are products of polynomials and slowly varying functions. Such scales have been widely used as natural refinements of polynomial scales in various areas of analysis, including number theory, complex analysis and probability theory (see e.g.\ \cite[Sections 6-8]{BGT}). This suggests that if there is a generalization of Theorem \ref{borto} to more general classes of resolvent bounds, the scale of regularly varying functions should be one of the first candidates for such a class.

In this paper, we develop a general operator-theoretical approach
to the study of \emph{optimal} decay rates of operator semigroups
within the fine scales of regularly varying functions. Since our
primary motivation comes from the study of rates in \eqref{rater}
(the case of singularity at infinity), we first formulate several
partial cases of our results obtained below for the setting of
\eqref{rater}. We use the fact that the function $(\log s)^\beta$
is slowly varying with de Bruijn conjugate $(\log s)^{-\beta}$ (see Theorem \ref{thm.main}, Corollary \ref{cor.main} and Theorem \ref{thm.faster} below).

\begin{theorem}\label{regvarinf}
Let $(T(t))_{t \ge 0}$ be a bounded $C_0$-semigroup on a Hilbert space $X$, with generator $-A.$  Assume
that $\sigma(A)\cap i\mathbb R$ is empty, and fix  $\alpha >0$ and $\beta \ge 0.$
\begin{enumerate}[\rm (a)]
\item  \label{rvii1} The following statements are equivalent:
\begin{enumerate}[\rm (i)]
\item  $\|(is+A)^{-1}\|={\rm O} \left(|s|^\alpha (\log |s|)^{-\beta}\right), \qquad |s| \to \infty.$
\item  $\|T(t)A^{-1}\|={\rm O} \left(t^{-1/\alpha} (\log t)^{-\beta/\alpha}\right), \qquad t \to \infty$.
\end{enumerate}
\item \label{rvii2} If
\begin{equation*}
\|(is+A)^{-1}\|={\rm O} \left(|s|^\alpha (\log|s|)^\beta \right), \qquad |s| \to \infty,
\end{equation*}
then, for every $\varepsilon >0$,
\begin{equation*}
\|T(t)A^{-1}\|={\rm O} \left(t^{-1/\alpha}
(\log t)^{\varepsilon+\beta/\alpha}\right), \qquad  t \to \infty.
\end{equation*}
\end{enumerate}
\end{theorem}

We do not know whether (\ref{rvii2}) is true for $\varepsilon=0$.   If so,  the implication would become an equivalence.

It is apparent that our methodology  goes beyond clarifying
\eqref{rater} and it allows one to deal with rates of decay of
several other operator families stemming from bounded
$C_0$-semigroups on Hilbert spaces.  Thus it is significant from the point of view of both abstract operator theory and applications.

Observe that if $(T(t))_{t \ge 0}$ is a bounded $C_0$-semigroup on a Hilbert space $X$ with generator $-A$, then for each $x\in X$ the function  $-T(\cdot )A^{-1}x$ is a primitive of $T(\cdot )x$. Now instead of considering asymptotics of the primitives of $T(\cdot )x$, we may study asymptotics of the derivatives of $T(\cdot )x$.  In other words, we may consider the orbits of the form  $T(\cdot )Ax$ for $x \in \dom(A)$, and so study
long-time regularity of $(T(t))_{t \ge 0}$. This type of
asymptotic behaviour of semigroups has not been treated
systematically in the literature so far. Some related but very
partial results pertaining to orbits of analytic semigroups were
obtained in \cite{Du09}. More results are available in the
discrete setting, see e.g.\ \cite{Du08}, 
\cite{KMOT04}, \cite{KaPo08}and \cite{Ne93}. In this paper, the decay of
$T(\cdot )x$ for $x \in \ran(A)$ is studied systematically and the
resulting structure appears to be very similar to the one
established in the study of decay rates for the orbits $(T(t))_{t
\ge 0}$ starting from the domain of $A$ and discussed above.

Since we are interested in a  decay of $T(t)Ax$ to zero that is uniform with respect to $x \in \dom(A)$,  the problem which we address in this case is to quantify the decay of  $\|T(t)A(I+A)^{-1}\|$.  Thus our task is to identify spectral conditions on  $A$  and a corresponding ``optimal'' function $r$ such that
\begin{equation}\label{ratesatzero}
\|T(t)A(I+A)^{-1}\| = {\rm O}(r(t)), \qquad t \to \infty.
\end{equation}
 Remark that if $\|T(t)A(I+A)^{-1}\| \to 0$ as  $t \to \infty$, then the
spectrum of $A$ does not meet $i\mathbb R\setminus \{0\}$ and the
resolvent of $A$ is bounded on $i (\R\setminus (-1,1))$, see
Theorem \ref{resbound0} below. Therefore the only singularity of
the resolvent on the imaginary axis may be at zero and we have a
spectral situation which is in a sense opposite to the situation considered
above in the study of \eqref{rater}. We describe this situation by
saying that the resolvent restricted to the imaginary axis 
has a \emph{singularity at zero}. Basic examples in the
framework of a singularity at zero are provided, in particular, by
generators of bounded eventually differentiable semigroups, 
for example those arising in the study of delay differential equations \cite{Ba07}, \cite[Section VI.6]{EnNa00}.

To treat \eqref{ratesatzero}, it is natural to assume the
boundedness of $(\lambda+A)^{-1}$ outside a neighbourhood of zero
in  $i\mathbb R,$ and to relate the decay of $\|T(t)A(I+A)^{-1}\|$
to the growth of  $(\lambda+A)^{-1}$ near zero.  The first problem that we encountered on this way was that the mere
convergence of $\|T(t)A(I+A)^{-1}\|$ to zero was wide open. This
type of convergence can be considered as an extension of the
famous Katznelson-Tzafriri theorem (see Theorem \ref{KT}) with $L^1$-functions replaced by
certain bounded measures on the real half-line. Using a new
technique, we obtain an interesting generalization of the
Katznelson-Tzafriri theorem for a class of bounded measures.
Moreover, we are able to derive a partial analogue of Theorem
\ref{chdu} for the case of a singularity at zero thus equipping our
version of the Katznelson-Tzafriri theorem with rates. These
results are given below, see also Theorems \ref{KT+} and \ref{mlog}.

\begin{theorem}\label{katzintro}
Let $(T(t))_{t \ge 0}$ be a bounded $C_0$-semigroup on a Hilbert
space $X$ with generator $-A.$ Assume that  $E:= -i\sigma(A) \cap
\mathbb R$ is compact and of spectral synthesis, and moreover the
resolvent of $A$ is bounded on $i (\mathbb R \setminus (-\eta,\eta))$ for
some $\eta>0$.
\begin{enumerate}[\rm (a)]
\item  If $\mu$ is a finite measure on $\mathbb R_+$ such that its
Fourier transform vanishes on $E$, then
\begin{equation*}
\lim_{t \to \infty} \|T(t) \hat \mu(T) \| = 0,
\end{equation*}
where $\hat\mu(T) x:= \int_{0}^{\infty}T(s)x\, \ud\mu(s)$ for $x \in X$.
\item  Assume that $(\lambda+A)^{-1}$ has a singularity at zero {\rm(}so $E= \{ 0\}${\rm)}, let
\[
m(s): = \sup\{\|(ir + A)^{-1}\| : |r| \ge s\}, \qquad s>0 ,
\]
and assume that $\lim_{s \to 0+}s m(s)=\infty$.  Let $\varepsilon \in (0,1)$.
Then there exist positive constants $c, C, C_\varepsilon$ such that
\begin{equation*}
c m^{-1}(Ct) \le \| T(t)A(I + A)^{-1}\| \le C_\varepsilon {m}^{-1}(t^{1-\varepsilon}),
\end{equation*}
for all sufficiently large $t$.
\end{enumerate}
\end{theorem}

Our technique of employing fine scales of regularly varying
functions to the study of orbits decay proves to be efficient also
in the situation of a singularity at zero. In particular, we
obtain a counterpart of Theorem \ref{regvarinf}(\ref{rvii1}) in
that case, given in Theorem \ref{rateszerointro} (see also Theorem \ref{thm.main0}). Unfortunately, Theorem \ref{katzintro}(\ref{rvii2}) is not quite as strong as Theorem \ref{chdu}, because the correction term involves $t^\varepsilon$ instead of a logarithmic term. As a result, the theorem below covers only the case when the resolvent grows slightly slower than a power
of $|s|^{-1}$ (in analogy with Theorem \ref{regvarinf}(\ref{rvii1})). The
problem of characterizing the decay of $\|T(t)A(I+A)^{-1}\|$ in
the other case when the resolvent grows slightly faster than a
power of $|s|^{-1}$ remains open.

\begin{theorem}\label{rateszerointro}
Let $(T(t))_{t \ge 0}$ be a bounded $C_0$-semigroup on a Hilbert space,
with generator $-A$. Assume that $\sigma(A)\cap i\mathbb R \subset
\{0\}$,  and let $\alpha >1$, $\beta \ge 0$. The following statements
are equivalent.
\begin{enumerate}[\rm (i)]
\item $\|(is + A)^{-1}\|=\begin{cases}
\O \left(|s|^{-\alpha}(\log(1/|s|))^{-\beta}\right),& \qquad
s \to 0,\\
\O (1),& \qquad |s| \to \infty.
\end{cases}
$
\item  $\|T(t)A(I+A)^{-1}\|= \O \left(t^{-1/\alpha} (\log t)^{-\beta/\alpha}\right),
\qquad t \to \infty$.
\end{enumerate}
 \end{theorem}
 
 Note that Theorem \ref{rateszerointro} for  $\beta=0$ provides a resolvent characterization  of polynomial rates of decay for $\|T(t)A(I+A)^{-1}\|$. This characterization  holds for $\alpha=1$, too.  The case when $0<\alpha<1$ does not arise since $\|(is+A)^{-1}\|\ge |s|^{-1}$ if $0 \in \sigma (A)$.  For other results revealing properties of decay rates for $\|T(t)A(I+A)^{-1}\|$ we refer to Sections \ref{s.zero1} and \ref{s.zero2}.

The operator-theoretical approach of the present paper can also be
used successfully to treat decay rates of orbits of $(T(t))_{t \ge
0}$ starting in $\dom(A)\cap \ran(A)$, combining the situations of a singularity at infinity and a singularity at zero considered above. Since $\dom(A)\cap \ran(A) = \ran \left(A (I+A)^{-2}\right)$ (Proposition \ref{prop.aab}) and since we are interested in uniform rates of decay, the task of characterizing such rates can be considered as the task of characterizing the
property
\begin{equation}\label{ratesboth}
\|T(t)A(I+A)^{-2}\|= \O(r(t)), \qquad t \to \infty,
\end{equation}
in resolvent terms.

While  \eqref{ratesboth} implies that $\sigma (A)\cap
i\mathbb R \subset \{0\}$, it does not exclude growth of the
resolvent near zero and near infinity (along the imaginary axis).
At the same time, \eqref{ratesboth} imposes certain restrictions
on the resolvent growth (see Theorem \ref{resbound0+}) which might
serve as a starting point for obtaining \emph{optimal} decay rates
in \eqref{ratesboth}. The following result (see Theorem \ref{poldecayzero}) illustrates that point. It is a generalization of
Theorem \ref{borto} above.

\begin{theorem}
Let $(T(t))_{t\ge 0}$ be a bounded $C_0$-semigroup on a Hilbert
space $X$ with generator $-A,$ and assume that $\sigma(A)\cap
i\mathbb R = \{0\}$.   If there exist   $\alpha\ge1$ and $\beta >
0$  such that
\begin{equation}\label{resboth}
\|(is + A)^{-1} \| =\begin{cases} {\rm O} (|s|^{-\alpha}),& \quad
s \to 0, \\
{\rm O}(|s|^\beta),& \quad  |s| \to \infty,
\end{cases}
\end{equation}
 then
\begin{equation*}
 \| T(t)A^{\alpha}(I + A)^{-(\alpha+\beta)} \| = {\rm O}(t^{-1}), \qquad t \to \infty,
 \end{equation*}
 and
 \begin{equation}\label{decayboth}
  \| T(t)A(I + A)^{-2} \|={\rm O}(t^{-1/\gamma}),
 \end{equation}
where $\gamma=\max(\alpha,\beta)$.

Conversely, if \eqref{decayboth} holds for some $\gamma > 0$, then \eqref{resboth} holds for $\alpha=\max(1,\gamma)$ and $\beta=\gamma$.
\end{theorem}

After this review of the main results of this paper, we note that we derive as by-products a number of results of independent interest. These include
abstract converses to interpolation inequalities (Theorem
\ref{interpol2}), estimates of decay rates for semigroup orbits with
bounded local resolvents (Theorem \ref{thm.CRbound}), an extension
of the Katznelson-Tzafriri theorem to the setting of measure algebras
 (Theorem \ref{KT+}), and lower bounds for orbit decay in the Banach space setting (Corollaries \ref{lowbound0} and \ref{lowbound0+}).

\subsection{Strategy}

One of the main novelties of the paper is its operator-theoretical
method for deriving estimates for rates of decay which in many
cases happen to be sharp. Let us describe the method in some more
detail.

If $(T(t))_{t\ge 0}$ is a bounded $C_0$-semigroup on a Hilbert
space $X$ with generator $-A,$ then we look for decay rates of
$T(t)B$ for a bounded operator $B$ which takes one of the three
forms: $A^{-1}, A(I+A)^{-1}$ and $A(I+A)^{-2}$. Given a (upper)
 bound $M$ for the resolvent on the imaginary axis, we start by
 establishing lower bounds for decay rates of $T(t)B$  in terms of $M$. 
 Such bounds are known in the case of a singularity at infinity (see \eqref{genestintro}),  and they are obtained in this paper in the case of a  singularity at zero and in the case of singularities at both zero and infinity (Corollaries \ref{lowbound0} and \ref{lowbound0+}).

To obtain an upper bound for decay rates of $\|T(t)B\|$ matching the
lower bound mentioned above we then proceed in three steps, of which only the second is restricted to Hilbert spaces.  {\it
First,} we show that if the resolvent $(\lambda +A)^{-1}$ grows
regularly on the imaginary axis, then it is bounded in the
right half-plane when restricted to the range of an associated operator $W$ involving a fractional power and a Bernstein function $f$ of $A$, see Theorems \ref{thm.bounded} and \ref{thm.bounded0}. {\it Second,} in Theorem \ref{thm.CRbound}, we
prove that if $\|(\lambda+A)^{-1}W\|$ is bounded in the right
half-plane then $\|T(t)W\|$ decays like $t^{-1}$ as $t\to \infty$.
{\it Third,} using new abstract converses to interpolation
inequalities for Bernstein functions (Theorem \ref{interpol2}), we
deduce that if $\|T(t)W\|$ decays like $t^{-1}$, then on the domain of $A$, or the range of $A$, or the intersection of the two, the decay of
$(T(t))_{t \ge 0}$ is expressed in terms of $f$. In the context of Theorem \ref{borto} this step reduces to an application of the moment inequality for fractional powers.  For the finer scales of rates of decay, the argument is much more subtle and its effect is to improve some a priori bounds for the semigroup.  The decay obtained in this way matches the lower bound in many cases and then it is optimal (apart from Theorem
\ref{thm.faster} where the upper bound differs from the lower 
bound by an arbitrarily small power of a logarithm).

Operator Bernstein functions were used successfully in
\cite{GHT12} to deal with rates in mean ergodic theorems for
bounded $C_0$-semigroups on Banach spaces. While there are formal
similarities between \cite{GHT12} and the present paper, the
problems treated here are much more involved and technically and
ideologically demanding. While the mean ergodic theorem for
$C_0$-semigroups is a comparatively simple statement, the majority of
stability conditions for $C_0$-semigroups are deep results with
tricky proofs. This extends to the framework of the study of
rates.

Our approach establishes a certain structure for dealing with rates in three cases, when the resolvent of the semigroup generator has singularities at infinity, or at the origin, or at both of them. While the structure has many elements in common between the three cases, there are several essential differences between them which have to be addressed separately.

In this paper the approach is applied to the class of regularly varying rates which are close to polynomial rates.  When the resolvent grows rapidly and fairly regularly, the upper and lower bounds in Theorem \ref{chdu} are of the same order.  However there are some rates which grow regularly but  are intermediate between polynomial and exponential, where the optimal rate of decay of $\|T(t)A^{-1}\|$ in Theorem \ref{chdu} is not established for semigroups on Hilbert space.  For very irregular (but arbitrarily fast) rates $M$ of growth of the resolvent it is not possible to improve Theorem \ref{chdu} by replacing $M_{\log}^{-1}$ by $M^{-1}$, even for semigroups of  normal operators on Hilbert spaces (see Propositions \ref{normalinf} and \ref{normalinf0}).

\subsection{Notation and conventions}

In this paper, $X$ will be a complex Banach space, and will often be specified to be a Hilbert space.  We let $\mathcal{L}(X)$ denote the space of all bounded linear operators on Banach space $X$, and the identity operator will be denoted by $I$.   If $A$ is a linear operator on $X$, we denote the domain of $A$ by $\dom(A)$, the range of $A$ by $\ran(A)$, the spectrum of $A$ by $\sigma(A)$ and the resolvent set by $\rho(A)$.  If $A$ is closable, its closure is written as $\overline{A}$.

If $B$ is another linear operator on $X$, then we take $A+B$ and $AB$ to be the operators with
\begin{eqnarray*}
\dom(A+B) &=& \dom(A) \cap \dom(B), \\
\dom(AB) &=& \{ x \in \dom(B) : Bx \in \dom(A)\}.
\end{eqnarray*}

A complex variable may be denoted by either  $z$ or $\lambda$.  We shall use the symbol $\iota$ to denote the identity function on domains in $\C$.  The closure of a subset $E$ of $\C$ will be denoted by $\overline{E}$.

For $\varphi \in (0,\pi]$ we shall let $\Sigma_\varphi  := \{ \lambda \in \C \suchthat |\arg\lambda| < \varphi\}$ be the sector of angle $\varphi$ in $\C$. 
Note that $\Sigma_\pi$ is the slit plane $\C \setminus (-\infty,0]$. We may write
\begin{eqnarray*}
\C_+ := \Sigma_{\pi/2} &=& \{ \lambda \in \C \suchthat \Re\lambda>0 \}, \\
\R_+ &:=& [0,\infty).
\end{eqnarray*}

We shall consider integrals of functions $f$ with respect to positive Radon measures $\mu$ over $(0,\infty)$ or $[0,\infty)$.  We shall write such integrals as
$$
\int_{0+}^\infty f(s) \, \ud\mu(s)  \quad \text{or} \quad \int_{0}^\infty f(s) \, \ud\mu(s),
$$
respectively. If $g$ is an increasing right-continuous function on $(0,\infty)$ and $\mu$ is the associated Lebesgue-Stieltjes measure, we shall write
$$
\int_{0+}^\infty f(s) \, \ud{g}(s)
$$
instead of
$$
\int_{0+}^\infty f(s) \, \ud{\mu}(s).
$$
Let $M^b(\R)$ denote the space of all complex Borel measures of bounded variation on $\R$.  For $a \in \R$, we let $\delta_a$ be the Dirac measure at $a$.

We consider the spaces $L^1 (\R_+ )$ and $M^b (\R_+ )$ as (closed) subspaces of $L^1 (\R )$ and $M^b (\R )$, respectively, by extending functions or measures on $\R_+$ by $0$ on $(-\infty ,0)$. In addition, we view $L^1 (\R )$ as a closed subspace of $M^b (\R )$.   The standard convolution of two measures $\mu_1,\mu_2 \in M^b(\R)$ will be denoted by $\mu_1*\mu_2$.

The Fourier transform of a function $f\in L^1 (\R ; X)$ is defined by
\[
 \cF f (\xi ) := \int_\R e^{-i\xi t} f(t) \, \ud{t}, \qquad \xi\in\R  .
\]
This definition of the Fourier transform holds in particular for $f$ in the Schwartz space $\mathcal{S}(\R)$ of scalar-valued test functions and it then induces a Fourier transform for all vector-valued, tempered distributions.  The Fourier transform of  $\mu\in M^b (\R )$ is therefore given by
\[
 \cF \mu (\xi ) := \int_\R e^{-i\xi t} \, \ud\mu (t), \qquad \xi\in\R  .
\]
When $X$ is a Hilbert space, we shall also use the symbol $\cF$ to denote the Fourier transform induced on the Hilbert space $L^2(\R;X)$, so that $(2\pi)^{-1/2} \cF$ is a unitary operator.

We are interested in asymptotic properties of functions defined on intervals of the form $[a,\infty)$ for some $a>0$, with values in $(0,\infty)$.  We shall say that $f$ and $g$ are {\em asymptotically equivalent}, and we write $f \sim g$, or $f(s) \sim g(s)$, if 
$$
 \lim_{s \to \infty} \frac{f(s)}{g(s)} = 1.
$$ 
This defines an equivalence relation on such functions and we shall in effect be working with equivalence classes of functions.  This viewpoint provides a justification for sometimes not including precise statements about the domains of our functions (provided that each domain contains some interval of the form $(a,\infty)$) or repeatedly saying that an inequality holds for all sufficiently large $s$. 

Where we use the notation $f(s) \sim g(s)$, it will mean asymptotic equivalence as $s\to\infty$ unless otherwise  specified.  Occasionally we shall use the corresponding notation as $s\to0+$, but then it will be specified.  Thus, for positive functions $f$ and $g$ defined on $(0,a]$, the notation
$$
f(s) \sim g(s), \qquad s\to0+,
$$
means
$$
 \lim_{s \to0+} \frac{f(s)}{g(s)} = 1.
$$

For an increasing function $f: [a,\infty) \to (0,\infty)$ such that $\lim_{s\to\infty} f(s) = \infty$, the notation $f^{-1}$ may denote the inverse function, or more generally a \emph{right inverse}, of $f$, defined on the range of $f$, so that $f(f^{-1}(t)) = t$ for all $t$ in the range.  It may also denote an \emph{asymptotic inverse} of $f$, defined on an interval $[b,\infty)$, such that
\begin{equation} \label{asyinv}
f^{-1}(f(s)) \sim s, \quad   \quad  f(f^{-1}(s)) \sim s.
\end{equation}
It should be clear from the context which notion of inverse function is involved.

 We shall occasionally use the notation $f^\alpha$ to denote function $s \mapsto f(s)^\alpha$ when $\alpha\ne-1$, but we shall use $1/f$ to denote the reciprocal function of $f$, in order to avoid confusion with any inverse function.  Similarly $fg$ or $f.g$ will denote a pointwise product of two functions $f$ and $g$, and $f \circ g$ will denote composition.
 
 We shall use $C$ and $c$ to denote (strictly) positive constants, whose values may change from place to place.

\section{Preliminaries on some classes of functions}

In this section we review various classes of functions on $(0,\infty)$ with emphasis on the properties that we shall need.

\subsection{Bernstein functions, complete Bernstein functions and Stieltjes functions} \label{subsectfun}

In this subsection, we recall the definitions and some properties of complete Bernstein functions and Stieltjes functions.  Most of this material can be found in \cite{SSV10}.  In Section \ref{s.fc} we shall review the operator functional calculus associated with these functions, and that will be used in later sections of the paper.

Recall that a function $f \in \Ce^\infty (0,\infty)$ is {\em completely monotone} if 
\[
 (-1)^n f^{(n)} (\lambda ) \geq 0 \quad \text{ for every } n\in\N \cup \{0\} , \, \lambda\in (0,\infty ) .
\]
By Bernstein's theorem \cite[Theorem 1.4]{SSV10}, every completely monotone function $f$ is the Laplace transform of a positive Radon measure on $\R_+$, and $f$ extends to a holomorphic function in the right half-plane. A function $f\in \Ce^\infty(0, \infty)$ is called a {\em Bernstein function} if
\[
f\geq 0 \quad \text{and}\quad f' \text{ is completely monotone} .
\]
Clearly, every Bernstein function also extends to a holomorphic function in the right half-plane. 
By the L\'evy-Khintchine representation theorem \cite[Theorem 3.2]{SSV10}, a function $f$ is
a Bernstein function if and only if there exist  constants $a$,
$b\geq 0$ and a positive Radon measure $\mu_{\rm LK}$ on $(0,\infty)$
such that
\begin{align}
\nonumber & \int_{0+}^\infty\frac{s}{s+1}\,\ud\mu_{\rm LK}({s})<\infty, \qquad \text{ and} \\
\label{hpfc.e.bf}
& f(\lambda )=a+b\lambda +\int_{0+}^\infty (1-e^{-\lambda s}) \, \ud\mu_{\rm LK} ({s}), \qquad \lambda>0.
\end{align}
The triple $(a, b, \mu_{\rm LK})$ is uniquely determined by the corresponding Bernstein function $f$ and is called the {\em L\'evy-Khintchine triple} of $f$.   

The class of Bernstein functions is rather large and to ensure good
algebraic and function-theoretic properties of Bernstein
functions it is convenient, and also sufficient for  many purposes,
to consider the subclass consisting of complete Bernstein
functions. A function $f\in \Ce^\infty(0, \infty)$ is called a {\em complete Bernstein function} if it is a Bernstein function and the measure $\mu_{\rm LK}$ in the L\'evy-Khintchine triple has a completely monotone density with respect to Lebesgue measure \cite[Definition 6.1]{SSV10}. By \cite[Theorem 6.2]{SSV10}, every complete Bernstein function admits a representation of the form
\begin{equation} \label{compbern}
 f(\lambda ) = a + b\lambda + \int_{0+}^\infty \frac{\lambda}{s+\lambda} \,\ud\mu(s), \qquad \lambda >0,
\end{equation}
for some constants $a,b \ge 0$ and some positive Radon measure $\mu$ on $(0,\infty )$ satisfying 
\begin{equation} \label{mmu}
 \int_{0+}^\infty \frac{\ud\mu(s)}{s+1} <\infty .
\end{equation} 

\begin{remark} \label{berviahirsch}
Complete Bernstein functions admit various representations different from \eqref{compbern}. In particular, the following formula is used in some papers related to Bernstein functions (for example in \cite{Hir74}, \cite{HirschInt}, \cite{Hir75}, \cite{HirschFA}, \cite{Pus82}):
\begin{equation}\label{diffrepr}
f(\lambda)= a +\int_{0}^{\infty}\frac{\lambda}{1+ \lambda t}\,
\ud\nu(t)= a + \nu(\{0\})\lambda +\int_{0+}^{\infty
}\frac{\lambda}{1+ \lambda t}\,\ud\nu(t),
\end{equation}
where  $\nu$ is a positive Radon measure on $\R_+$ satisfying
\begin{equation*}
\int_{0}^{\infty}\frac{\ud\nu(t)}{1+t}<\infty,
\end{equation*}
and the pair $(a,\nu)$ is unique.

The representations \eqref{diffrepr} and \eqref{compbern} are equivalent by the change of variable $s = 1/t$, with $\nu$ being the push-forward measure of $\mu$ combined with an atom of mass $b$ at $0$, and vice versa.
\end{remark}

There are striking dualities between complete Bernstein functions and another class known as Stieltjes functions. A function $h\in\Ce^\infty(0, \infty)$ is a {\em Stieltjes function} if there exist constants $a,b\ge0$ and a positive Radon measure $\mu$ on $(0,\infty )$ satisfying \eqref{mmu} such that 
\begin{equation} \label{hpfc.e.stieltjes}
h(\lambda )=\frac{a}{\lambda} +b +\int_{0+}^\infty  \frac{\ud\mu(s)}{s+\lambda}, \qquad \lambda>0.
\end{equation}

The representation formulas \eqref{compbern} (for complete Bernstein functions) and \eqref{hpfc.e.stieltjes} (for Stieltjes functions) are unique, and they are called the {\em Stieltjes representations} for $f$ and $h$, respectively; see e.g. \cite[Chapter
2]{SSV10}. We write $f \sim (a,b, \mu)$ and $h \sim (a,b, \mu)$, and we call $(a,b,\mu)$ the {\em Stieltjes triple}, and $\mu$ the {\em Stieltjes measure} for $f$ and $h$, respectively. Note that
\begin{equation*}\label{ab}
a=\lim_{\lambda\to 0+} f(\lambda) = \lim_{\lambda\to 0+} \lambda \, h(\lambda ) , \qquad b = \lim_{\lambda\to\infty} \frac {f(\lambda)}{\lambda} = \lim_{\lambda \to \infty} h(\lambda) .
\end{equation*}
We shall be particularly interested in the Stieltjes (and complete Bernstein) functions with Stieltjes representations of the form $(0,0,\mu)$.  When $\mu$ is the Lebesgue-Stieltjes measure associated with an increasing right-continuous function $g$ we shall denote the Stieltjes function with representation $(0,0,\mu)$ by $S_g$, and we shall call it the {\em Stieltjes function associated with $g$}.

Note that every (complete) Bernstein function $f$ is increasing and every Stieltjes function is decreasing.  Comparison of (\ref{compbern}) and (\ref{hpfc.e.stieltjes}) shows that if $h\in \Ce^\infty(0, \infty)$ is a Stieltjes function then $f(\lambda ):=\lambda h(\lambda )$ is a complete Bernstein function, and conversely if $f\in \Ce^\infty(0, \infty)$ is a complete Bernstein function then $h(\lambda ):=f(\lambda )/\lambda$ is a Stieltjes function.  The classes of complete Bernstein functions and Stieltjes functions are preserved under various operations \cite[Theorem 6.2, Theorem 7.3, Corollary 7.4, Corollary 7.6]{SSV10}.  We present here some which will be used in the paper.

\begin{theorem}\label{sti-char}
Let $f$ be a non-zero function on $(0,\infty )$.  
\begin{enumerate}[\rm (a)]
\item  $f$ is a complete Bernstein function if and only if $1/f(\lambda)$ is a Stieltjes function.
\item  If $f$ is a complete Bernstein function, then $\lambda /f(\lambda)$ and $\lambda f(1/\lambda)$ are complete Bernstein functions.  Conversely, if $\lambda /f(\lambda)$ or $\lambda f(1/\lambda)$ is a complete Bernstein function, then $f$ is a complete Bernstein function.
\item \label{compcbf} If $f$ and $g$ are both complete Bernstein functions or both Stieltjes functions, then $g \circ f$ is a complete Bernstein function.
\end{enumerate}
\end{theorem}

Many examples of complete Bernstein functions, and hence of Stieltjes functions, are given in \cite[Chapter 15]{SSV10}.  We give here a few of the most elementary examples that will be relevant in this paper.

\begin{example}\label{contex}
(a) The function $h(\lambda):=\lambda^{-\gamma}$ ($\gamma \in (0,1)$) is a Stieltjes function with the Stieltjes representation
\[
h(\lambda )=\frac{\sin \pi\gamma}{\pi} \int_0^\infty
\frac{1}{s+\lambda} \, \frac{\ud s}{s^\gamma} \,, \qquad \lambda >0 .
\]
Accordingly, $f(\lambda )=\lambda^{1-\gamma} = \lambda h(\lambda)$ is a complete Bernstein function.

(b) For $\alpha \in (0,1)$, the function $f(\lambda) := \lambda(\lambda+1)^{-\alpha}$  is a complete Bernstein function with the Stieltjes representation
$$
f(\lambda)  =  \frac{ \sin \pi\alpha}{\pi} \int_1^\infty \frac{\lambda}{s+\lambda} \, \frac{\ud{s}}{(s-1)^\alpha}.
$$
Moreover, $\lambda(\lambda+1)^{-1}$ is a complete Bernstein function with Stieltjes triple $(0,0,\delta_1)$.

(c)    It follows from (a) and (b), together with Theorem \ref{sti-char}(\ref{compcbf}), that $f(\lambda) := \lambda^\alpha (1+\lambda)^{-\beta}$ is a complete Bernstein function whenever $0 \le \beta \le \alpha \le 1$.
\end{example}

Any Stieltjes function or complete Bernstein function can be extended to a holomorphic function on the slit plane $\Sigma_\pi$ by means of the formula (\ref{hpfc.e.stieltjes}) or (\ref{compbern}), respectively.  We shall often regard the functions as being defined on the slit plane in this way, without explicit comment.  The rate of decay or growth of such functions at infinity in sectors is determined by the rate on $(0,\infty)$, as shown by the following known fact \cite[Lemma 2]{Pus79}, \cite[Proposition 2.21 (c)]{Kw11}.

\begin{proposition} \label{prop.asymp}
Let $g$ be either a Stieltjes function or a complete Bernstein function with Stieltjes representation $(0,0,\mu)$.  Let $\lambda \in \Sigma_\pi$ and $\varphi=\arg \lambda$.  Then
\[
 \cos(\varphi/2) \, g(|\lambda |) \le |g(\lambda )| \le \sec(\varphi/2)\, g(|\lambda |).
\]
\end{proposition}

\begin{proof}  The second inequality follows from (\ref{hpfc.e.stieltjes}) or (\ref{compbern}) and the elementary inequality:
\begin{equation*}
|s+\lambda |^2 \ge \frac {(1+\cos\varphi)}{2} (s+|\lambda |)^2 = \big(\cos(\varphi/2) \,(s+|\lambda |)\big)^2, \qquad s \in (0,\infty).
\end{equation*}
For the first inequality, we can assume that $g\ne0$. By Theorem \ref{sti-char}, $1/g$ is a complete Bernstein function or a Stieltjes function.  Hence
\[
\frac{1}{|g(\lambda )|}  \le  \frac{1}{\cos(\varphi/2) \,g(|\lambda |)}.
\qedhere
\]
\end{proof}

\begin{subsection} {Slowly and regularly varying functions}

Most of the material in this subsection is standard (see \cite[Chapter 1]{BGT}, \cite[Sections IV.1-9]{Ko04} or \cite{Seneta}).

\begin{definition}\label{slowvar}
Let $\ell$ be a strictly positive measurable function defined on $[a,\infty)$ for some $a \in \R$ and satisfying
\begin{equation*}
\lim_{s \to \infty} \frac{\ell(\lambda s)}{\ell(s)}=1
\end{equation*}
for every $\lambda >0$. Then $\ell$ is said to be {\em slowly varying}.
\end{definition}
Clearly the value of $a$ is not important in this definition.  By defining $\ell(s)=\ell(a)$ for $s \in [0,a]$, we may assume that $a =0$.  

\begin{example}
(a) Standard examples of slowly varying functions include
\begin{enumerate}
\item[]  iterated logarithms \quad $\log_k(s):=\log ...\log s,  \quad k\in\N$,
\vskip5pt
\item[]  $\exp \left\{(\log s)^{\alpha_1} (\log_2(s))^{\alpha_2} ... (\log_k(s))^{\alpha_k})\right\}, \qquad
\alpha_i \in (0,1), \, 1 \le i \le k $,
\vskip5pt
\item[]  $ \exp \left( \dfrac{\log s}{\log \log s} \right)$.
\end{enumerate}
(b) It is a straightforward consequence of Definition \ref{slowvar} that the sum and product of two slowly varying functions is slowly varying.  Moreover, if $\ell$
is slowly varying then the following are also slowly varying.
\begin{enumerate}
\item[]  $\ell^\alpha: s \mapsto \ell(s)^\alpha, \qquad \alpha \in \R$,
\item[]  $\ell_\alpha: s \mapsto \ell(s^\alpha), \qquad \alpha>0$,
\item[]  $\ell.\log: s \mapsto \ell(s) \log s$.
\end{enumerate}
\end{example}

A proof of the following representation theorem, originally due to Karamata for continuous functions $\ell$, may be found in \cite[Theorem 1.3.1]{BGT} or \cite[Section IV.3]{Ko04}.

\begin{theorem} \label{slowvarrep}
The function $\ell$ is slowly varying if and only if it is of the form
\begin{equation*}
\ell(s)=c(s)\exp \left\{ \int_{a}^{s} \frac{\varepsilon(t)}{t} \, \ud t \right\},
\qquad s \ge a,
\end{equation*}
for some $a >0,$ where $c$ and $\varepsilon$ are measurable functions, $c(s) \to c >0$ and $\varepsilon(s) \to 0$ as $s \to \infty$.
\end{theorem}

The following corollary is easily deduced from Theorem \ref{slowvarrep}; see \cite[Theorem 1.5.6]{BGT}, \cite[p.18]{Seneta}.

\begin{corollary} \label{cor.sv}
Let $\ell$ be a slowly varying function, and $\gamma > 0$.    
\begin{enumerate}[\rm(a)]
\item There are positive constants $C,c$ such that 
$$
c \left( \frac{s}{t} \right)^\gamma \le \frac {\ell(t)}{\ell(s)} \le C \left( \frac{t}{s} \right)^\gamma
$$
for all sufficiently large $s,t$ with $t\ge s$.
\item  As $s\to\infty$,
\begin{equation} \label{svf}
s^\gamma \ell(s)\to \infty, \qquad s^{-\gamma}\ell(s)\to 0.
\end{equation}
\end{enumerate}
\end{corollary} 

On the other hand there are slowly varying functions such that
\[
\liminf_{s \to \infty}\ell(s)=0, \qquad \limsup_{s \to \infty}\ell(s)=\infty.
\]

\begin{definition}\label{regvar}
A positive function $f$ is called {\em regularly varying with index}
$\alpha \in \R$ if there is a slowly varying function $\ell$ such that
$$
f(s)=s^{\alpha}\ell(s), \qquad s \ge a.
$$
Such a function has a representation
\begin{equation*}
f(s) = s^\alpha \, c(s)\, \exp\left\{\int_{a}^{s} \frac{\varepsilon(t)}{t} \,
\ud t\right\}, \qquad s \ge a,
\end{equation*}
where $c$ and $\varepsilon$ are as in Theorem \ref{slowvarrep}.
\end{definition}

If $f$ is regularly varying of index $\alpha>0$, there is a strictly increasing, regularly varying function $g$ which is asymptotically equivalent to $f$ \cite[Theorem 1.5.3]{BGT}.  One can also arrange that $g$ is smooth \cite[Theorem 1.8.2]{BGT}. Moreover, $f$ has an asymptotic inverse in the sense of \eqref{asyinv}.  For example, one may take the asymptotic inverse of $f$ to be the usual inverse of a strictly increasing, continuous, function which is asymptotically equivalent to $f$.  The asymptotic inverse is regularly varying, it depends only on the asymptotic equivalence class of $f$, and it is unique up to asymptotic equivalence.  Indeed, the asymptotic equivalence classes of regularly varying functions with positive index form a group under composition; they also form a semigroup under pointwise multiplication \cite[Theorem 1.8.7]{BGT}. 

A convenient way to handle asymptotic inverses of regularly varying functions involves the {\em de Bruijn conjugate} $\ell^\#$ of the slowly varying function $\ell$ \cite[Section 1.5.7]{BGT}.  This is a slowly varying function $\ell^\#$ such that
\[
\text{$\ell(s) \ell^\#(s\ell(s)) \to 1$ \quad and  \quad $\ell^\#(s) \ell(s \ell^\#(s))) \to 1$ \quad as $s\to\infty$}.
\]
One can take
\[
\ell^\#(s) = \frac {(\iota.\ell)^{-1}(s)}{s} \, ,
\]
where $(\iota.l)(s) = s\ell(s)$. For this choice of $\ell^\#$, it is easy to see that if $\ell$ is increasing (resp., decreasing), $\ell^\#$ is decreasing (resp., increasing).  The group structure of the asymptotic equivalence classes of regularly varying functions immediately implies that one-sided asymptotic inverses are unique up to asymptotic equivalence.  Hence, if $k$ is a slowly varying function and either $\ell(s) k(s\ell(s)) \to 1$ or $k(s) \ell(s k(s))) \to 1$ as  $s \to \infty$, then $k \sim \ell^\#$.

\begin{example} \label{ex.dBinv}
A method for finding many de Bruijn conjugates is given in \cite[Appendix 5]{BGT}, including the following cases.
\begin{enumerate}[1.]
\item If $\ell(s) = (\log s)^\beta$ where $\beta\in\R$, then $\ell^\#(s) \sim (\log s)^{-\beta}$.

\item Let $\ell(s) = \exp \left((\log s)^\beta\right)$ where $0<\beta<1$.
\begin{enumerate}[(a)]
\item If $0 < \beta \le 1/2$, then $\ell^\#(s) \sim \exp \left(-(\log s)^\beta\right)$.   
\item If $\frac 12 \le \beta < \frac23$, then
\begin{eqnarray*}
\ell^\#(s) &\sim&  \exp \left( -(\log s)^\beta + \beta(\log s)^{2\beta-1} \right), \\
(1/\ell)^\#(s) &\sim& \exp \left( (\log s)^\beta + \beta(\log s)^{2\beta-1} \right).
\end{eqnarray*}
\end{enumerate}
For values of $\beta$ between $2/3$ and $1$, there are longer formulas of this type.
\end{enumerate}
\end{example}

\begin{proposition} \label{regvarinv}
Let $\ell$ be slowly varying, and let $\alpha>0$.  Then
\begin{enumerate}[\rm (a)]
\item  \label{rvi1} $\ell^{\#\#} \sim \ell$.
\item  \label{rvi2} If $f(s) \sim s^\alpha \ell(s^\alpha)$, then $f^{-1}(s) \sim s^{1/\alpha} \ell^\#(s)^{1/\alpha}$.
\item  \label{rvi3} If $g: (0,a] \to (0,\infty)$ and $g(s) \sim s^\alpha / \ell(s^{-\alpha})$ as $s\to0+$, then $g^{-1}(s) \sim s^{1/\alpha} / \ell^\# (1/s)^{1/\alpha}$ as $s\to0+$.
\end{enumerate}
\end{proposition}

\begin{proof}  The statements (\ref{rvi1}) and (\ref{rvi2}) are in \cite[Section 1.5]{BGT} and \cite[Section 1.6]{Seneta}.  For (\ref{rvi3}), note that $g(s) \sim 1/f(1/s)$ as $s \to 0+$, where $f$ is as (\ref{rvi2}).  Using the regular variation, one can easily deduce that $g^{-1}(s) \sim 1/f^{-1}(1/s)$ as $s\to0+$.
\end{proof}

The following lemma describes a common situation in which $\ell^\#$ has a particularly simple form.  Parts of the lemma appear in \cite[Section 1.5]{BGT} and \cite[Section 1.6]{Seneta}.  

\begin{lemma} \label{regvarinv2}
Let $\ell$ be a slowly varying function.  The following are equivalent:
\begin{enumerate}[\quad\rm(i)]
\item \label{dB1} $\ell^\# \sim 1/\ell$.
\item   \label{dB2} $\ell \left(s\ell(s)\right) \sim \ell(s)$.
\item \label{dB3} $\ell \left(s/\ell(s)\right) \sim \ell(s)$.
\end{enumerate}
If $\ell$ is monotonic and $\alpha>0$, these properties are equivalent to each of the following:
\begin{enumerate}
\item[\rm(iv)]  $\ell \left(s\ell(s)^{\alpha}\right) \sim \ell(s)$.
\item[\rm{(v)}]  $(\ell_\alpha)^\# \sim 1/\ell_\alpha$, where $\ell_\alpha(s) = \ell(s^\alpha)$. 
\end{enumerate}
\end{lemma}

\begin{proof}  The statements (\ref{dB2}) and (\ref{dB3}) are equivalent to the statements
\[
\ell(s) k(s\ell(s)) \to 1, \qquad k(s) \ell(sk(s)) \to 1
\]
respectively, when $k= 1/\ell$.  So their equivalence to (\ref{dB1}) follows from the one-sided uniqueness properties of $\ell^\#$ set out before Example \ref{ex.dBinv}.

Now assume that (\ref{dB2}) holds and $\ell$ is monotonic.  Replacing $s$ by $s\ell(s)$ gives
\[
\ell \Big(  {s} {\ell(s)\ell\big(s\ell(s)\big)} \Big) \sim \ell \big( {s}{\ell(s)} \big)  \sim \ell(s).
\]
Moreover, by (\ref{dB2}),
\[
{s} {\ell(s)^2} \sim  {s} {\ell(s)\ell\big(s\ell(s)\big)} ,
\]
so
\[
\ell \left( {s} {\ell(s)^2} \right) \sim \ell \Big( {s} {\ell(s)\ell\big(s\ell(s)\big)} \Big) \sim \ell(s)
\]
by the Uniform Convergence Theorem \cite[Theorem 1.2.1]{BGT}.  
Iterating this argument gives
\[
 \ell \left(s\ell(s)^\alpha\right) \sim \ell(s)
\]
whenever $\alpha$ is a power of $2$.  Then the monotonicity of $\ell$ gives it for all $\alpha>0$.  Thus (\ref{dB2}) implies (iv).  The converse follows by replacing $\ell(s)$ by $\ell(s)^{1/\alpha}$. 

Finally, let $k_\alpha = 1/\ell_\alpha$.  Assume that (iv) holds.  Replacing $s$ by $s^\alpha$ in (iv) gives
$$
\ell_\alpha(s) k_\alpha(s\ell_\alpha(s)) = \frac {\ell(s^\alpha)}{\ell\left(s^\alpha \ell(s^\alpha)^\alpha \right)}  \to 1.
$$
Then (v) follows from the one-sided uniqueness property of $(\ell_\alpha)^\#$.   Replacing $\alpha$ by $1/\alpha$ shows that (v) implies (iv).
\end{proof}

We shall say that a monotonic, slowly varying, function $\ell$ is {\em dB-symmetric} when the conditions of Lemma \ref{regvarinv2} are satisfied. Example \ref{ex.dBinv} shows that the following functions $\ell$ are dB-symmetric:
\begin{enumerate}
\item[] $\ell(s) = (\log s)^\beta$, for any $\beta \in \R$,
\item[] $\ell(s) = \exp \left( (\log s)^\beta \right)$, if $0 < \beta < \frac12$.
\end{enumerate}
If $\ell$ is dB-symmetric, it is clear from Lemma \ref{regvarinv2} that the following functions are also dB-symmetric:
\begin{enumerate}
\item[] $\ell_\alpha : s \mapsto \ell(s^\alpha)$, for any $\alpha>0$, 
\item[] $\ell^\alpha : s \mapsto  \ell(s)^\alpha$, for any $\alpha \in \R$.
\end{enumerate}   
It is not difficult to show that the product of two dB-symmetric functions is dB-symmetric.

For a regularly varying function $f$, let $f.{\log}$ be the following regularly varying function:
\[
(f.{\log})(s) = f(s) \log s.
\]
We shall need the following relations.

\begin{lemma} \label{lem.logpert}
Let $f(s) = s^\alpha \ell(s)$ be a regularly varying function with $\alpha>0$, and let $\delta>1/\alpha$.  The following hold for some constant $C$.
\begin{enumerate}[\rm (i)]
\item \label{logpert1} $\displaystyle (f.{\log}) \left( \frac {s} {(\log f(s))^\delta} \right) \le C f(s)$.
\item \label{logpert2} $\displaystyle f^{-1}(s) \le C (f.{\log})^{-1}(s) (\log s)^\delta$.
\end{enumerate}
When $\ell$ is increasing, these statements are also true for $\delta = 1/\alpha$.
\end{lemma}

\begin{proof}
(\ref{logpert1}).  By Corollary \ref{cor.sv}, $\log f(s) \sim \alpha \log s$.  Also,
\begin{eqnarray*}
(f.{\log}) \left( \frac {s} {(\log f(s))^\delta} \right) &=&
f(s) \frac {\log \left( \frac {s} {(\log f(s))^\delta} \right)} {(\log f(s))^{\alpha\delta}} \frac {\ell \left( \frac {s} {(\log f(s))^\delta} \right)} {\ell(s)}  \\
&\le& C f(s) (\log s)^{1-\alpha\delta} \frac {\ell \left( \frac {s} {(\log f(s))^\delta} \right)} {\ell(s)}  \\
&\le& C f(s) (\log s)^{1-\alpha\delta + \delta\gamma}
\end{eqnarray*}
for any $\gamma>0$ by Corollary \ref{cor.sv}, and for $\gamma=0$ if $\ell$ is increasing.

\noindent (\ref{logpert2}).  Replacing $s$ by $f^{-1}(s)$ in (\ref{logpert1}) gives
\[
(f.{\log}) \left( \frac {f^{-1}(s)} {(\log s)^\delta} \right) \le Cs.
\]
The claim follows since $(f.{\log})^{-1}$ is increasing and regularly varying.
\end{proof}

Finally in this section, we consider Stieltjes functions associated with regularly varying functions.

\begin{example} \label{stfng}
Let $\ell$ be a slowly varying function on $\R_+$, $\alpha\ge0$, and assume that $g(s):= s^\alpha \ell(s)$ is increasing. The associated Stieltjes function
\[
S_{g} (\lambda) := \int_{0+}^{\infty} \frac{\ud g(s)}{s+\lambda} = \int_{0}^{\infty} \frac{s^\alpha \ell(s)}{(s+\lambda)^{2}} \, \ud s,
\]
is defined if either integral is finite \cite[p.7]{Wi41}.  This occurs if $\alpha<1$,  by (\ref{svf}), or if $\alpha=1$ and $\int_0^\infty \frac {\ell(s)}{s+1} \, \ud s$ is finite.
\end{example}

We shall need the following abelian/Tauberian theorem of Karamata.  For the main results we shall need only the abelian parts (i)$\Rightarrow$(ii) and (iii)$\Rightarrow$(iv).

\begin{theorem}[Karamata]\label{karamata}
Let $g$ be an increasing function on $\mathbb R_+$, and let
$S_g$ be the associated Stieltjes function. 
Let $0< \sigma \le 1$, and $\ell$ be slowly varying on $\mathbb R_+$. 
\begin{enumerate}[\rm (a)]
\item \label{kar1}
The following are equivalent:
\begin{enumerate}[\rm(i)]
\item  $ g(s)\thicksim s^{1-\sigma}\ell(s) \qquad s \to \infty$;
\item
$S_g(\lambda) \thicksim \Gamma(\sigma)\Gamma(2-\sigma) \lambda^{-\sigma} \ell(\lambda), \qquad \lambda \to \infty$.
\end{enumerate}
\item \label{kar2} The following are equivalent:
 \begin{enumerate}[\rm(i)]
\item[\rm(iii)] $ g(s)\thicksim s^{1-\sigma} \ell(1/s), \qquad s \to 0+$;
\item[\rm(iv)] $S_g(\lambda) \thicksim \Gamma(\sigma)\Gamma(2-\sigma) \lambda^{-\sigma} \ell(1/\lambda), \qquad \lambda \to 0+$.
\end{enumerate}
\end{enumerate}
\end{theorem}

\begin{proof}  The first statement is proved in \cite[Theorem 1.7.4]{BGT} (the Tauberian implication is proved in \cite[Theorem 2.5]{Seneta}).  The proof of the second statement is very similar, using the same preliminary results from \cite[Sections 1.5,1.7]{BGT}. 
\end{proof}
\end{subsection}

\section{Functional calculus of sectorial operators} \label{s.fc}

In this section we recall basic properties of functional calculus of sectorial operators based on complete Bernstein and Stieltjes functions, originally due to Hirsch and extended to include fractional powers.  We concentrate on those properties which are needed for our main purposes in the later sections of this paper.   A much fuller account of the calculus for Bernstein functions can be found in the monograph \cite{SSV10}, and of the extended calculus in \cite{Haa06}.

\subsection{Sectorial operators}

Many parts of the paper will involve the notion of a sectorial operator which we recall now.  

\begin{definition}
A densely defined, linear operator $A$ on a Banach space $X$ is called \emph{sectorial} if $(-\infty,0)\subset \rho (A)$ and there exists $C>0$ such that
\begin{equation}\label{sector}
\|\lambda (\lambda+A)^{-1}\|\le C, \qquad \lambda >0.
\end{equation}
\end{definition}

Some authors require a sectorial operator to be injective, and some do not require it to be densely defined.  Some of the operators that we consider will not be injective.

Note that any sectorial operator $A$ is closed and by the Neumann series expansion, $(\lambda+A)^{-1}$ is defined and \eqref{sector} holds on the sector  $\Sigma_\varphi$, for some $\varphi \in (0,\pi]$.  Moreover,
\begin{equation} \label{approxid}
\lambda (\lambda+A)^{-1}x \to x, \quad \lambda\to\infty, \qquad x \in X,
\end{equation}
\cite[Proposition 2.1.1,d)]{Haa06}.

If $-A$ generates a bounded $C_0$-semigroup, then $A$ is sectorial.  The following standard lemma establishes that certain auxiliary  operators which play important roles in this paper are also sectorial even though they may not be negative generators of bounded semigroups.

\begin{lemma} \label{sectops}
Let $A$ be a sectorial operator on a Banach space $X$.  Then the operators $A^{-1}$ (if $A$ is invertible), $A(I+A)^{-1}$ and $A(I+A)^{-2}$ are sectorial.
\end{lemma}

\begin{proof}
Sectoriality of $A^{-1}$ and $A(I+A)^{-1}$ follows from the identities 
\begin{eqnarray*}
(\lambda+A^{-1})^{-1}&=&\lambda^{-1}-\lambda^{-2}(\lambda^{-1}+A)^{-1},\\
\left(\lambda+A(I+A)^{-1}\right)^{-1}&=&\frac{1}{\lambda+1}+\frac{1}{(\lambda+1)^2}\left(\frac{\lambda}{\lambda+1}+A\right)^{-1}, 
\end{eqnarray*}
which hold for $\lambda>0$ (see \cite[Proposition 2.1.1,b)]{Haa06} and \cite[Lemma 3.1]{Oka00}). To prove sectoriality of $A(I+A)^{-2}$, we note the identity
$$
\lambda (\lambda+A(I+A)^{-2})^{-1}= I - \lambda^{-1} (\mu+A)^{-1} \left( I - \mu^{-1} \left(\mu^{-1} + A \right)^{-1} \right),
$$
where $\mu>1$ and $\mu^{-1}$ are the roots of $\mu^2 - (\lambda^{-1} + 2)\mu + 1 = 0$.  Since $\mu>\lambda^{-1}$, sectoriality of $A$ implies that $\|\lambda^{-1} (\mu+A)^{-1} \|$ and $\| \mu^{-1} \left(\mu^{-1} + A \right)^{-1}\|$ are bounded for $\lambda>0$, hence $A(I+A)^{-2}$ is sectorial.
\end{proof}

Other information about sectorial operators may be found in \cite{Haa06} and \cite{Ma11} (in the latter sectorial operators are called non-negative operators).

\subsection{Hirsch functional calculus}

We now define complete Bernstein functions of sectorial operators and review those basic properties that we need in the sequel.  There are several different definitions of functional calculus of sectorial operators, and we shall describe some properties which cross over between the different definitions.  We try to present the ideas of functional calculus in a way which reveals the heuristics of our subsequent arguments, and to give the later proofs in ways which do not rely on any unjustified assumptions about compatibility of different definitions. 

Let $f$ be a complete Bernstein function with Stieltjes representation $(a,b,\mu)$, and let $A$ be  a sectorial operator on a Banach space $X$.
The next definition was essentially given in \cite[p.\ 255]{HirschInt}.

\begin{definition}\label{defoperbern}  Define an operator $f_0(A): \dom(A) \to X$ by
\begin{equation}\label{defbernop}
f_0(A)x=a x + b A x +\int_{0+}^{\infty}A (\lambda +A)^{-1}x\, \ud\mu(\lambda), \qquad x \in \dom(A).
\end{equation}
By \eqref{mmu}, this integral is absolutely convergent and $f_0(A)(I+A)^{-1}$ is a bounded operator on $X$, extending $(I+A)^{-1}f_0(A)$.  Hence $f_0(A)$ is closable as an operator on $X$.  Define
\begin{equation*}
f(A)=\overline {f_0(A)}.
\end{equation*}
We call $f(A)$ a {\it complete Bernstein function of $A$.}
\end{definition}

Actually Hirsch used the representation \eqref{diffrepr} for a complete Bernstein function $f$ and then defined the corresponding operator $f(A)$ as
\begin{equation*}
f^H(A)=\overline {f^H_0(A)},
\end{equation*}
where $f_0^H$ is defined on $\dom(A)$ by
\begin{equation}\label{hirschdef}
f^H_0(A)x:=a + \nu(\{0\}) A x +\int_{0+}^{\infty}A (I+ \lambda A)^{-1}x\, \ud\nu(\lambda), \qquad x \in \dom (A).
\end{equation}
By Remark \ref{berviahirsch} the representations \eqref{compbern} and \eqref{diffrepr} describe exactly the same classes of complete Bernstein functions.   To see that the resulting operators $f(A)$ and $f^H(A)$ coincide, it suffices to change variables in the same way as in Remark \ref{berviahirsch}, considering vector-valued integrals instead of scalar ones.  Thus we can use Hirsch's results even if we think of complete Bernstein functions $f$ as being represented by \eqref{compbern}.

When $-A$ generates a bounded $C_0$-semigroup, one can define the operators $f(A)$ for arbitrary Bernstein functions $f$  by adapting the L\'evy-Khintchine formula \eqref{hpfc.e.bf}.  A detailed discussion of this approach as well as of various properties of $f(A)$ can be found in \cite[Section 12]{SSV10} (where $A$ denotes the generator of the semigroup).  One can prove that the definition of complete Bernstein functions of semigroup generators in \cite{SSV10} is consistent with Hirsch's definition (Definition \ref{defoperbern}), but we shall not go into details.  we shall refer to some results in \cite{SSV10} for complete Bernstein functions even though they start from the L\'evy-Khintchine representation. 

By Definition \ref{defoperbern}, $\dom (A)$ is a core for $f(A)$. Using \eqref{approxid} and the fact that $(\lambda+A)^{-1}$ commutes with $(I+A)f_0(A)(I+A)^{-1}$, we infer that $\dom (A)$ is a core for $(I+A)f_0(A)(I+A)^{-1}$ as well. Since $(I+A)f_0(A)(I+A)^{-1}$ is closed as a product of a closed operator and a bounded operator, it follows that 
\begin{equation*}
f(A)=(I+A)f_0(A)(I+A)^{-1}.
\end{equation*}

\begin{remark} \label{boundedcbf}
If $f$ is a bounded complete Bernstein function then, in the Stieltjes representation of $f$, $b=0$ and Fatou's lemma implies that the measure $\mu$  is finite (see \cite[Corollary 3.7(v)]{SSV10}). In this case, $f_0(A)$ is bounded on $\dom(A)$, and therefore $f(A)$ is a bounded operator on $X$.
 It is also straightforward to see that $f(A)$ is bounded for any complete Bernstein function $f$ if $A$ is bounded (see \cite[Corollary 12.7]{SSV10}).
\end{remark}

\begin{remark}
In the sequel,  \eqref{defbernop} will usually be used for measures $\mu$ which are Lebesgue-Stieltjes measures associated with increasing, regularly varying functions $g$.  Then we may think of the integrals as being vector-valued Riemann-Stieltjes integrals rather than Bochner integrals. The theory of vector-valued Riemann-Stieljes integrals is presented in \cite[Section 1.9]{ABHN01}, \cite[Section III.3.3]{HilPhi}, \cite[Section 1]{Wi41}.  One can define complete Bernstein functions and operator Bernstein functions initially by means of Riemann-Stieltjes integrals. However, in the relevant literature, including \cite{HirschInt},  \cite{HirschFA}, \cite{Pus82} and \cite{SSV10}, \eqref{compbern} and \eqref{defoperbern}, or \eqref{diffrepr} and \eqref{hirschdef}, are standard ways to define complete Bernstein functions, and we have chosen to follow an established route.  
\end{remark}

Complete Bernstein functions of sectorial operators possess a number of properties which allow one to create a partial functional calculus for $A$.
However the set of complete Bernstein functions is not closed under pointwise multiplication, so the multiplicative properties of this process are restricted.  The subject was thoroughly investigated by Hirsch in the 1970s, and he proved the following properties of operator Bernstein functions in \cite[Th\'eor\`eme 1-3]{HirschInt} and \cite[Th\'eor\`eme 1]{HirschFA} (see also \cite[pp.\,200-201]{HirschFA}, \cite{Pus82} and \cite{Hir74}).

\begin{theorem}\label{hirsch}
Let $A$ be a sectorial operator on a Banach space $X,$ and let $f$ and $g$ be complete Bernstein functions.
Then the following statements hold.
\begin{enumerate}[\rm(i)]
\item \label{hiri} The operators $f(A)$ and $g(A)$ are sectorial.
\item \label{hirii} The composition rule holds: $$f (g (A))= (f\circ g)(A).$$
\item \label{hiriii} If $f g$ is also a complete Bernstein function, then the product rule holds: $$f(A)g(A)= (f g)(A).$$
\end{enumerate}
\end{theorem}

For $\alpha \in  (0,1)$, Example \ref{contex}(a) shows that $z^\alpha$ is a complete Bernstein function, and we write $A^\alpha$ for the corresponding complete Bernstein function of $A$. These fractional powers coincide with the standard fractional powers which are extensively studied in \cite{Ma11}. If $A$ is sectorial and $\alpha,\beta \in (0,1)$ then Theorem \ref{hirsch} implies that  $(A^\alpha)^\beta = A^{\alpha\beta}$, and  moreover  $A^\alpha A^\beta=A^{\alpha+\beta}$ if $\alpha +\beta \le 1$.    In this paper we shall need these properties for a larger range of $\alpha$ and $\beta$.  This is standard theory, but we put it in a broader context here.

To this aim, we shall use an extended holomorphic functional calculus, which has become a standard tool to deal with functions of sectorial operators.   To keep the presentation within reasonable limits we give only a very brief sketch of part of the extended holomorphic calculus and refer for further details to \cite[Chapters 1,2]{Haa06} and \cite[Sections II.9, II.15]{Weis}.

Let $H$ be the algebra of functions which are holomorphic in
$\Sigma_\pi=\mathbb C \setminus (-\infty,0]$.  Let $\mathcal H$
stand for the set of all $f \in H$ such that for any  $\varphi \in
(0,\pi)$  there exist $c \in \mathbb{C}$ and $C >0, \alpha>0$ (both depending on
$\varphi$) satisfying
\begin{equation}\label{inf}
|f(z)-c| \le C |z|^{\alpha}, \qquad   z \in \Sigma_\varphi.
\end{equation}
Let $\widetilde {\mathcal H}$ denote the set of all $f \in H$ such
that for any $\varphi \in (0,\pi)$ one has 
\begin{equation}\label{zeroinf}
|f(z)|\le C  \max(|z|^\alpha, |z|^{-\alpha}), \qquad z \in \Sigma_\varphi,
\end{equation}
 for some  $C, \alpha >0$. It follows easily from \eqref{compbern}, \eqref{mmu} and \eqref{hpfc.e.stieltjes} that $\widetilde{\mathcal H}$ contains all complete Bernstein functions and all Stieltjes functions, and one can set $\alpha=1$ for all $\varphi \in (0,\pi)$.

Let $\mathcal{C}(X)$ denote the set of all  closed, densely
defined, linear operators on $X$.

\begin{theorem}\label{calculus}
Let $A$ be a sectorial operator on a Banach space $X$. Then there
exists a well-defined mapping
\begin{eqnarray*}
\mathcal H &\mapsto& \mathcal{C}(X),\\
 f &\mapsto& f(A),
\end{eqnarray*}
called an extended holomorphic functional calculus, such that
\begin{enumerate}[\rm (i)]
\item \label{calci} $1(A)=I$ and $z(A)=A$;
\item \label{calcii} If $T \in \mathcal L(X)$ and $TA \subset AT$, then $Tf(A)\subset f(A)T$;
\item \label{calciii} $f(A)+g(A) = (f+g)(A),$  if $g(A) \in \mathcal L (X)$;
\item \label{calciv} $f (A) g(A) = (fg)(A),$  if $g(A) \in \mathcal L (X)$;
\item \label{calcv} if $g \in \mathcal H$ is such that $g(A)$ is sectorial and  there exists
$\vartheta : (0,\pi) \to (0,\pi)$ such that $\lim_{\varphi\to\pi-}
\vartheta(\varphi) = \pi$ and
$g(\Sigma_{\vartheta(\varphi)})\subset \Sigma_{\varphi}$ for all
$\varphi \in (0,\pi)$, then $(f \circ g) (A)=f(g(A))$.
\end{enumerate}
If $A$ is injective then there exists a mapping $ \widetilde
{\mathcal H} \mapsto \mathcal{C}(X), f \mapsto f(A)$, satisfying
the properties {\rm (i)-(v)} above.
\end{theorem}

\begin{remark} \label{calcrem}
We note the following facts about the extended holomorphic
functional calculus described in Theorem \ref{calculus}.
\begin{enumerate}[{\rm (i)}]
\item If $f(A) \in \mathcal L(X)$ and $\lambda \in \rho(A)$, then $f(A)(\lambda +A)^{-1}=(\lambda +A)^{-1}f(A)$.
\item If $f$ is a rational function whose poles all lie in $(-\infty,0]$, then $f(A)$ as defined in the extended
holomorphic functional calculus of Theorem \ref{calculus} agrees with the natural definition of $f(A)$.
\item \label{calcremciii} If $f \in \mathcal{H}$ is a complete Bernstein function, then $f(A)$ as in the calculus of
Theorem 3.7 agrees with $f(A)$ as defined in Definition 3.3.   This is
shown in [14, Theorem 4.12] for injective $A$.  The proof in the general case is
the same up to replacement of the regulariser $(z/(1+z)^2)^n$ by the regulariser
$(1+z)^{-2}$.
\item \label{calcremci} The composition rule Theorem \ref{calculus}\eqref{calcv} applies in particular
when $g(z) = z^{-1}$ (if $A$ is invertible), $z(1+z)^{-1}$, or
$z(1+z)^{-2}$.  In each case, it suffices to note Lemma \ref{sectops} and to use the fact that the mapping $z \mapsto 1/z$  preserves sectors.
\item \label{calcremcii} In particular, if $A$ is invertible, $g \sim (0,0,\nu)$ is a Stieltjes function and $f$ is the complete Bernstein function given by $f(z) = g(1/z)$, then
$$
g(A) = \int_{0+}^\infty (\lambda+A)^{-1} \ud\nu(\lambda) =
f(A^{-1}).
$$
\end{enumerate}
\end{remark}

It is easy to see that $\{z^\alpha: \alpha >0\} \subset \mathcal H$ and $\{ z^{\alpha}: \alpha \in \mathbb R \} \subset \tilde{\mathcal H}$. Thus if $A$ is sectorial, then  the ``fractional powers''  $A^\alpha$ for $\alpha > 0$ (and for all $\alpha \in \mathbb R$, if $A$ is invertible) are well-defined by means of the extended holomorphic functional calculus and, as we shall see below, they behave very similarly to the scalar functions $z^\alpha$.  These fractional powers coincide with the fractional powers  considered in \cite[Chapters 3,4]{Haa06}, \cite[Section II.15]{Weis} and \cite[Chapter 2]{MaSa01} where all properties of fractional powers of operators needed in this paper can be found. For the reader's convenience we recall and summarize them in the following statement.
  
 \begin{theorem}\label{fractional}
 Let $A$ be a sectorial operator on a Banach space $X$. Then for any $\alpha >0$ and $\beta >0$ the following statements hold.
 \begin{enumerate}[\rm(i)]
 \item If $T \in \mathcal L(X)$ and $TA \subset AT$ then  $T A^\alpha\subset A^\alpha T$.
 \item  \label{sgp} The semigroup law holds: $A^{\alpha+\beta}=A^{\alpha}A^\beta$.
 \item  \label{shiftdom} $\dom(A^\beta) = \dom((I+A)^\beta) = \ran ((I+A)^{-\beta})$.
 \item  \label{powers} If in addition $\alpha \in (0,1)$,  then  $A^\alpha$ is sectorial, and the composition law holds: $(A^\alpha)^\beta=A^{\alpha\beta}$.
 \end{enumerate}
 If $A$ is invertible then $A^{-\alpha} = (A^{-1})^\alpha$ for $\alpha>0$, 
 \eqref{sgp} is true for $\alpha,\beta \in \mathbb R$,  and \eqref{powers} holds for $\alpha \in (-1,1)$ and $\beta \in \mathbb R$.
 \end{theorem}

The following facts are elementary when $\alpha$ and $\beta$ are integers (Lemma \ref{sectops} and \cite[Proposition 9.4(d)]{Weis}) and they may be known for fractional powers, but we were unable to find them in the literature.  The fact that $(-\infty,0)$ is contained in the resolvent set of $A^\alpha(I+A)^{-1}$ when $0<\alpha<1$ follows from \cite[Theorem 4.1]{Haa05}, but the method given there does not establish the sectoriality estimate.  We are very grateful to Alexander Gomilko for providing the proof of that estimate.

\begin{proposition} \label{prop.aab}
Let $A$ be a sectorial operator on a Banach space $X$, and let $\alpha,\beta \in [0,1]$.  The following statements hold:
\begin{enumerate}[\rm(i)]
\item  \label{domim}
$\ran (A^{\alpha}(I+A)^{-(\alpha +\beta)})= \ran(A^{\alpha})\cap \dom(A^{\beta})$.
\item \label{aabsect} $A^\alpha(I+A)^{-\beta}$ is sectorial.
\end{enumerate}
\end{proposition}

\begin{proof}
To prove (\ref{domim}), we shall use Theorem \ref{fractional}  and the consequential facts that $A^{\alpha}(I+A)^{-\alpha}$ and $(I+A)^{-\beta}$ are bounded commuting operators, and $\ran (A^{\alpha})= \ran (A^{\alpha} (I+A)^{-\alpha})$.  Hence  
\begin{eqnarray*}
\ran (A^\alpha)\cap \dom(A^\beta) &=& \ran (A^{\alpha}(I+A)^{-\alpha}) \cap \ran ((I+A)^{-\beta}) \\
& \supset& \ran \big(A^{\alpha}(I+A)^{-(\alpha +\beta)}\big).
\end{eqnarray*}
Conversely, let $x \in \ran (A^{\alpha}) \cap \dom(A^{\beta})$. Then $x= A^{\alpha}y_1$ and $x=(I+A)^{-\beta}y_2$ for some $y_1 \in \dom(A^\alpha),\; y_2 \in X$.  Since $A^{\alpha}y_1=(I+A)^{-\beta}y_2 \in \dom(A^\beta)$, one has $y_1 \in \dom(A^{\alpha+\beta}) = \dom((I+A)^{\alpha+\beta})$.  Let $y_3=(I+A)^{\alpha+\beta}y_1$. Then
\begin{equation*}
x= A^{\alpha}y_1=A^{\alpha}(I+A)^{-(\alpha+\beta)}y_3.
\end{equation*}

For (\ref{aabsect}), first consider the case when $0 < \beta \le \alpha \le 1$.  By Example \ref{contex}(b), $f(\lambda) := \lambda^\alpha(1+\lambda)^{-\beta}$ is a complete Bernstein function.  By Theorem \ref{hirsch} (\ref{hiri}),(\ref{hirii}), $A^\alpha(I+A)^{-\beta} = f(A)$ which is sectorial.

Now consider the case when $0< \alpha < \beta = 1$.  Let $\lambda>0$, and
\[
f_\alpha(z)=\frac{z^\alpha}{1+z},
\]
and
\[
g_{\alpha,\lambda}(z)= \frac{1}{\lambda} - \frac{1}{f_\alpha(z)+\lambda}=
\frac{f_{\alpha}(z)}{\lambda (f_\alpha(z)+\lambda)} = \frac{z^\alpha}{\lambda \left(z^\alpha + \lambda(1+z)\right)},
\]
for $z\in \Sigma_\pi$.  For $\varphi\in (0,\pi)$, let $\gamma_\varphi$ be the contour given by 
\[
\gamma_\varphi:=
\{te^{-i\varphi} \suchthat t\ge0\} \cup \{te^{i\varphi} \suchthat t\ge0 \},
\]
taken in the downward direction.   Then
\begin{equation} \label{galint}
\int_{\gamma_\varphi} |g_{\alpha,\lambda}(z)|\frac{|\ud{z}|}{|z|}<\infty,
\end{equation}
and
\begin{equation*}
g_{\alpha,\lambda}(\mu)=
\frac{1}{2\pi i}\int_{\gamma_\varphi}\,
\frac{g_{\alpha,\lambda}(z)}{\mu-z} \,{\ud{z}},\quad |\arg \mu|<\varphi.
\end{equation*}
Letting $\varphi \to \pi-$, we obtain
\begin{equation} \label{galform}
g_{\alpha,\lambda}(\mu)=
\frac{\sin \pi \alpha}{\pi}
\int_0^\infty
\frac{(1-t)t^\alpha }{|e^{\pi i\alpha} t^\alpha+\lambda(1-t)|^2}\,(\mu+t)^{-1} \,\ud{t}.
\end{equation}
Letting  $\mu \to 0+$ we also have
\begin{equation} \label{Z0}
\int_0^\infty
\frac{(1-t)t^{\alpha-1}}{|e^{\pi i\alpha} t^\alpha+\lambda(1-t)|^2} \,\ud{t} =0.
\end{equation}

By the product rule of Theorem \ref{calculus}(\ref{calciv}), $f_{\alpha}(A) = A^\alpha (I+A)^{-1}$.  From \eqref{galint} and \cite[Section 2.3]{Haa06},
$$
g_{\alpha,\lambda}(A) = \frac {1}{2\pi i} \int_{\gamma_\varphi} g_{\alpha,\lambda}(z) (z+A)^{-1} \, \ud{z},
$$
for $\varphi \in (0,\pi)$ sufficiently large. Here the integral is absolutely convergent, so $g_{\alpha,\lambda}(A) \in \mathcal{L}(X)$.  Since
$$
 \left(\lambda + f_\alpha(z)\right)\left( \lambda^{-1} - g_{\alpha,\lambda}(z)\right) = 1,
$$
the product rule implies that 
\begin{equation} \label{galinv}
\lambda^{-1} - g_{\alpha,\lambda}(A) = \left(\lambda + A^\alpha (I+A)^{-1} \right)^{-1}.
\end{equation}

The same arguments as for \eqref{galform} show that
\[
g_{\alpha,\lambda}(A)= \frac{\sin \pi \alpha}{\pi} \int_0^\infty
\frac{(1-t)t^\alpha}{|e^{\pi i\alpha} t^\alpha+\lambda(1-t)|^2} \, (t+A)^{-1}\,\ud{t}.
\]
From this, (\ref{sector}) and (\ref{Z0}) we obtain 
\begin{eqnarray*}
\|g_{\alpha,\lambda}(A)\| &\le& \frac{C\sin \pi \alpha}{\pi} \int_0^\infty
\frac{|1-t|t^{\alpha-1}}{|e^{\pi i\alpha} t^\alpha+\lambda(1-t)|^2} \,\ud{t}\\
&=& \frac{2C\sin \pi \alpha}{\pi} \int_0^1 \frac {(1-t)t^{\alpha-1}} {|e^{\pi i\alpha}t^\alpha + \lambda(1-t)|^2} \,\ud{t}.
\end{eqnarray*}
For $t \in (0,1)$,
\begin{eqnarray*}
|e^{\pi i\alpha} t^\alpha+\lambda(1-t)|^2&=& t^{2\alpha}+(1-t)^2\lambda^2+
2\cos(\pi \alpha)\lambda t^\alpha (1-t) \\
&\ge& c_\alpha \left(t^{2\alpha}+(1-t)^2\lambda^2\right),
\end{eqnarray*}
for some  $c_\alpha>0$.  Hence
\begin{equation*}
\|g_{\alpha,\lambda}(A)\|\le C_\alpha \int_0^1 \frac{(1-t)t^{\alpha-1}}{t^{2\alpha}+(1-t)^2\lambda^2} \,\ud{t}.
\end{equation*}
Now
\begin{eqnarray*}
\int_0^{1/2}
\frac{(1-t)t^{\alpha-1}}{t^{2\alpha}+(1-t)^2\lambda^2}\,\ud{t}
&\le& \int_0^{1/2}\frac{t^{\alpha-1}}{t^{2\alpha}+(\lambda/2)^2} \,\ud{t} \\
&\le& \frac{1}{\alpha}\int_0^\infty \frac{\ud\tau}{\tau^2+(\lambda/2)^2}=\frac{\pi}{\alpha\lambda},\notag
\end{eqnarray*}
and
\begin{eqnarray*}
\int_{1/2}^1
\frac{(1-t)t^{\alpha-1}}{t^{2\alpha}+(1-t)^2\lambda^2}\,\ud{t}
&\le&
2^{1-\alpha}\int_{1/2}^1 \frac{(1-t)}{2^{-2\alpha}+(1-t)^2\lambda^2} \,\ud{t}
\\
&=&2^{1-\alpha}\int_0^{1/2} \frac{\tau }{2^{-2\alpha}+\tau^2\lambda^2} \,\ud\tau \\
&=&\frac{1}{2^\alpha \lambda^2}\log \left(1+\left(\frac{2^\alpha\lambda}{2}\right)^2\right)\le
\frac{1}{\lambda} \,,
\end{eqnarray*}
since
\[
\log (1+s^2)\le 2 \log (1+s) \le 2s,\quad s>0,
\]
Thus 
$$
\|\lambda g_{\alpha,\lambda}(A)\| \le C_\alpha, \quad \lambda>0.
$$
It follows from (\ref{galinv}) that $A^\alpha(I+A)^{-1}$ is sectorial when $0 < \alpha < 1$.

When $0 \le \alpha < \beta < 1$, we have 
$$
A^\alpha(I+A)^{-\beta} = \left(A^{\alpha/\beta}(I+A)^{-1}\right)^\beta,
$$
and this is sectorial, by the previous case together with Theorem \ref{fractional} (\ref{powers}).
\end{proof}

Let $A$ be a sectorial operator and $f$ be a complete Bernstein function.  Then  one has
\begin{equation}\label{ineqbern}
\|f(A)x\| \le C \|x\| f \left( \frac{\|A x\|}{\|x\|} \right), \qquad x \in \dom(A), x \ne 0,
\end{equation}
where $C$ is a constant independent of $x$ (and $f$).  This is shown in  \cite{Pus82} for sectorial operators, and in \cite[Corollary 12.8]{SSV10} where  $-A$ is assumed to be the generator of a bounded $C_0$-semigroup but $f$ may be {\em any} Bernstein function.  

If $A$ is invertible, we also have
\begin{equation}\label{ineqbern1}
\|f(A^{-1})x\| \le C \|x\| f \left( \frac{\|A^{-1}x\|}{\|x\|} \right), \qquad x \in X, x \ne 0.
\end{equation}
This follows since $A^{-1}$ is also sectorial.  Alternatively, one can easily pass between \eqref{ineqbern} and \eqref{ineqbern1} by considering $\varphi(z):= zf(1/z)$. Then $\varphi$ is also a complete Bernstein function by Theorem \ref{sti-char}, and $\varphi(A)= Af(A^{-1})$ by the product rule and composition rule for $g(z)=z^{-1}$, in Theorem \ref{calculus} and Remark \ref{calcrem}(\ref{calcremcii}).  Applying \eqref{ineqbern} with $f$ replaced by $\varphi$ gives
\begin{equation*}
\|f(A^{-1})Ax\|\le C \|A x\| f \left( \frac{\|x\|}{\|A x\|}\right), \qquad x \in \dom(A), x\ne0.
\end{equation*}
Setting $Ax=y$  we obtain \eqref{ineqbern1}. Conversely, applying \eqref{ineqbern1} with $f$ replaced by $\varphi$ gives
\begin{equation*}
\|f(A)A^{-1}x\|\le C \|A^{-1} x\| f \left( \frac{\|x\|}{\|A^{-1} x\|}\right), \qquad x \in X, x\ne0.
\end{equation*}
Setting $A^{-1}x=y$ (so $y \in \dom(A)$) we obtain \eqref{ineqbern}.  
 
When $f(z) = z^{\alpha} \; (0<\alpha<1)$, we recover the classical moment inequality for fractional powers in the forms
\begin{eqnarray} \label{momineq0}
\|A^{\alpha} x\| &\le& C \|x\|^{1-\alpha} \|Ax \|^\alpha, \qquad x \in \dom(A),\\
\label{momineq1} 
\|A^{-\alpha} x\| &\le& C \|x\|^{1-\alpha} \|A^{-1}x \|^\alpha, \qquad x \in X.
\end{eqnarray}

\section{Some estimates for semigroup asymptotics}

In this section we give some results relating different types of asymptotic estimates for semigroups. In Subsection  \ref{ss.inter} we present some inequalities which are related to \eqref{ineqbern1} and \eqref{momineq0} but apply to generators of bounded semigroups.   In Subsection \ref{ss.transfer} we show how certain resolvent estimates can be transferred to semigroup estimates in the case of bounded semigroups on Hilbert space.

\subsection{Moment and interpolation inequalities} \label{ss.inter}

We start by recalling the full moment inequality for sectorial operators, extending \eqref{momineq0} and (\ref{momineq1}).

\begin{proposition} \label{prop.momineq}
 Let $B$ be a sectorial operator, and $0 \le \alpha < \beta < \gamma$.  There is a constant $C$ such that
\begin{equation} \label{momineq4}
\|B^\beta x\| \le C \|B^\alpha x\|^{(\gamma-\beta)/(\gamma-\alpha)} \|B^\gamma x\|^{(\beta-\alpha)/(\gamma-\alpha)}, \qquad x \in \dom(B^\gamma).
\end{equation}
Hence if $S : X \to \dom(B^\gamma)$ is a linear operator and $B^\gamma S \in \mathcal{L}(X)$, then
\begin{equation} \label{momineq3}
\|B^\beta S \| \le C \|B^\alpha S\|^{(\gamma-\beta)/(\gamma-\alpha)} \|B^\gamma S\|^{(\beta-\alpha)/(\gamma-\alpha)}.
\end{equation}
If $B$ is invertible, then \eqref{momineq4} and \eqref{momineq3} hold whenever $\alpha<\beta<\gamma$.
\end{proposition}

\begin{proof}
For $\alpha=0$, \eqref{momineq4} is the standard inequality \eqref{momineq0} \cite[Theorem 15.14]{Weis}, \cite[Corollary 5.1.13]{MaSa01}.  The more general cases follow by replacing $\beta$ by $\beta-\alpha$, $\gamma$ by $\gamma-\alpha$ and $x$ by $B^\alpha x$.  Then \eqref{momineq3} follows on replacing $x$ by $Sx$.  When $B$ is invertible the range of the inequalities can be extended by replacing $x$ or $S$ by $B^{-n}x$ or $B^{-n}S$.
\end{proof}

Next we deduce an inequality of interpolation type which was proved in a slightly less general form (and with a slightly different proof) in \cite[Proposition 3.1]{BaEnPrSchn06}.

\begin{lemma}\label{interpoldec}
Let $(T(t))_{t \ge 0}$ be a bounded $C_0$-semigroup on a Banach space $X$, and let $B \in \mathcal L(X)$ be a sectorial operator commuting with $(T(t))_{t \ge 0}$.  Let $\gamma, \delta > 0$.  Then there exist positive constants $C,c$ such that
\begin{equation}   \label{momineq}
c \|T(Ct)B^\gamma\|^\delta  \le \|T(t)B^\delta\|^\gamma \le C \|T(ct)B^\gamma\|^\delta, \qquad t>0.
\end{equation}
In particular, the following statements are equivalent:
\begin{itemize}
\item [(i)] $\|T(t) B^\gamma\|= \O(t^{-1}), \qquad t \to \infty$.
\item[(ii)] $\|T(t) B\|= \O(t^{-1/\gamma}), \qquad t \to \infty$.
\end{itemize}
\end{lemma}

\begin{proof}
Take $n \in \N$ such that $n\gamma\ge\delta$, and apply Proposition \ref{prop.momineq} with $\alpha=0$, $\beta=\delta/n$ and $S=T(t/n)$.  Then
$$
\|B^{\delta/n} T(t/n)\| \le C \|T(t/n)\|^{1-\delta/(n\gamma)} \|B^\gamma T(t/n)\|^{\delta/(n\gamma)}  \le C \|B^\gamma T(t/n)\|^{\delta/(n\gamma)}.
$$
Now
$$
\|T(t)B^\delta\|^\gamma  \le \|T(t/n)B^{\delta/n}\|^{n\gamma} \le C \|T(t/n)B^\gamma\|^\delta.
$$
This gives the second inequality in (\ref{momineq}), and the first follows by interchanging $\gamma$ and $\delta$.  The final statement follows by taking $\delta=1$.
\end{proof}

Our next result gives more interpolation properties for the generator of a bounded $C_0$-semigroup.  We shall need them for our main results in Sections \ref{s.infinity} and \ref{s.zero2}.  To simplify the presentation here and in Section \ref{s.zero2} we introduce the shorthand notation: 
\begin{equation*}
B(A):=A(I+A)^{-1} \label{ba}
\end{equation*}
when $A$ is a sectorial operator.  By Lemma \ref{sectops}, $B(A)$ is sectorial. 
Thus the fractional powers $B(A)^{\alpha}, \alpha >0$, are well-defined, and by the product and composition rules in Theorem \ref{calculus}\eqref{calciv} and Remark \ref{calcrem}\eqref{calcremci},
\begin{equation} \label{powerrules}
B(A)^{\alpha}=A^{\alpha} (I+A)^{-\alpha}.
\end{equation}

\begin{theorem} \label{interpol2}
Let $-A$ be the generator of a bounded $C_0$-semigroup $(T(t))_{t\ge0}$ on a Banach space $X$, and let $B(A) = A(I+A)^{-1}$.  Assume that $T(t)A\ne0$ for each $t>0$.  There exists a constant $c>0$ such that the following hold.
\begin{enumerate}[\rm (a)]
\item  \label{int2a} If $A$ is invertible, $f$ is a complete Bernstein function,  $\gamma\le1$, and $A^\gamma f(A^{-1})$ is a bounded operator, then
\begin{equation} \label{intineq}
\left\| T(t_1)A^{\gamma} f(A^{-1}) \right\| \ge c \frac {\|T(t_1+t_2)A^{\gamma-1}\|}{\|T(t_2)A^{-1}\|}  f(\|T(t_2)A^{-1}\|)
\end{equation}
for all $t_1,t_2\ge0$.
\item \label{int2b} If $f$ is a bounded complete Bernstein function and $\gamma>0$, then
$$
\left\| T(t_1)B(A)^{\gamma} f(A) \right\| \ge c \frac {\|T(t_1+t_2)B(A)^{\gamma+1}\|}{\|T(t_2)B(A)\|}  f(\|T(t_2)B(A)\|)
$$
for all $t_1,t_2\ge0$.
\end{enumerate}
\end{theorem}

\begin{remark}  In our applications of Theorem \ref{interpol2}, we shall take $t_1 = t_2$ and choose specific values of $\gamma>0$, but the applications are rather delicate.  For example, the ratio
$$
\frac {\|T(2t)A^{\gamma-1}\|}{ \|T(t)A^{-1}\|}
$$
tends to $0$ as $t \to \infty$ in all the cases in which we are interested in Section \ref{s.infinity}.  The inequality \eqref{intineq} will be used to improve other estimates for the rate of decay of $\|T(t)A^{-1}\|$.
\end{remark}

\begin{proof} (\ref{int2a}) We can assume that $f$ is non-zero, and we let $\varphi$ be the complete Bernstein function given by $\varphi(z) = z/f(z), \, z>0$ (Theorem \ref{sti-char}).  Since $f$, $\varphi$ and the identity function $z$ are complete Bernstein functions, the product rule (Theorem \ref{hirsch}\eqref{hiriii}) yields
\begin{equation*}
f(A^{-1}) \varphi(A^{-1})=A^{-1}.
\end{equation*}
 Then
\begin{eqnarray*}
\|T(t_1+t_2) A^{\gamma-1}\| &=& \|T(t_1+t_2) A^\gamma f(A^{-1}) \varphi(A^{-1})\| \\
&\le& \|T(t_1)A^\gamma f(A^{-1})\| \, \|T(t_2) \varphi(A^{-1})\|.
\end{eqnarray*}
Thus
\begin{equation} \label{interpol3}
\|T(t_1)A^\gamma f(A^{-1})\| \ge \frac {\|T(t_1+t_2) A^{\gamma-1}\|}{\|T(t_2) \varphi(A^{-1})\|}.
\end{equation}
We now estimate $\|T(t) \varphi(A^{-1})\|$ from above for $t>0$.  Let $\varphi$ have Stieltjes representation $(a,b,\mu)$.  By \eqref{defbernop},
$$
T(t) \varphi(A^{-1}) = aT(t) + bT(t)A^{-1} + \int_{0+}^\infty T(t) A^{-1} \big(\lambda+A^{-1}\big)^{-1} \, \ud\mu(\lambda).
$$
Let $\tau = \|T(t)A^{-1}\|$.  Then 
$$
\left\| \int_{0+}^\tau T(t) A^{-1} \big(\lambda+A^{-1}\big)^{-1}  \, \ud\mu(\lambda) \right\| \le C \int_{0+}^\tau \ud\mu(\lambda)
$$
since $A^{-1}$ is sectorial and $(T(t))_{t\ge0}$ is bounded.  Moreover,
\begin{eqnarray*}
\left\| \int_{\tau+}^\infty T(t) A^{-1} \big(\lambda+A^{-1}\big)^{-1} \, \ud\mu(\lambda) \right\| &\le& \tau  \int_{\tau+}^\infty \big\|\big(\lambda+A^{-1}\big)^{-1}\big\| \, \ud\mu(\lambda) \\ 
&\le& C \tau  \int_{\tau+}^\infty \frac {\ud\mu(\lambda)}{\lambda}.
\end{eqnarray*}
Let $K = \sup_{t\ge0} \|T(t)\|$.  Then
\begin{eqnarray*}
\|T(t)\varphi(A^{-1})\| &\le& aK + b\tau + C \int_{0+}^\tau \ud\mu(\lambda) + C \tau \int_{\tau+}^\infty \frac {\ud\mu(\lambda)}{\lambda} \\
&\le& aK + b\tau + 2C \int_{0+}^\tau \frac{\tau}{\lambda+\tau} \, \ud\mu(\lambda) + 2C \int_{\tau+}^\infty \frac {\tau}{\lambda+\tau} \, \ud\mu(\lambda) \\
&\le& 2C \varphi(\tau).
\end{eqnarray*}
Putting $t=t_2$, (\ref{interpol3}) gives
\begin{eqnarray*}
\|T(t_1)A^\gamma f(A^{-1})\| &\ge& c \frac {\|T(t_1+t_2) A^{\gamma-1}\|}{\varphi(\|T(t_2) A^{-1}\|)} \\
&=& c \frac {\|T(t_1+t_2) A^{\gamma-1}\|}{\|T(t_2) A^{-1}\|} f(\|T(t_2)A^{-1}\|).
\end{eqnarray*}

(\ref{int2b})  The proof is similar to (\ref{int2a}).  We now use the product rule in the form
\begin{equation*}
f(A) \varphi(A)=A.
\end{equation*}
Observe also that $f(A)$ is bounded (Remark \ref{boundedcbf}), $\varphi(A)$ is closed and $\dom (A) \subset \dom (\varphi(A))$ (Definition \ref{defoperbern}), and hence the operator $\varphi(A) (I+A)^{-1}$ is bounded.  So
\begin{eqnarray*}
\|T(t_1+t_2) B(A)^{\gamma+1}\| &=& \|T(t_1+t_2) B(A)^{\gamma} f(A) \varphi(A) (I+A)^{-1}\| \\
&\le& \|T(t_1)B(A)^{\gamma} f(A)\| \, \|T(t_2) \varphi(A) (I+A)^{-1}\|.
\end{eqnarray*}
Thus
\begin{equation*} \label{0interpol3}
\|T(t_1)B(A)^{\gamma}  f(A)\| \ge \frac {\|T(t_1+t_2) B(A)^{\gamma+1}\|}{\|T(t_2) \varphi(A) (I+A)^{-1}\|}.
\end{equation*}
Now
\begin{multline*}
T(t) \varphi(A) (I+A)^{-1} \\
= aT(t)(I+A)^{-1} + bT(t)A(I+A)^{-1} + \int_{0+}^\infty T(t) A (\lambda +A)^{-1} (I+A)^{-1} \, \ud\mu(\lambda).
\end{multline*}
Let $\tau = \|T(t)A (I+A)^{-1}\|= \|T(t)B(A)\|$.  Estimating as in a) gives
$$
\|T(t)\varphi(A)(I+A)^{-1}\| \le  2C \varphi(\tau).
$$
The claim follows as in (\ref{int2a}).
\end{proof}

\begin{remark}  The inequality \eqref{ineqbern1} easily implies that
$$
\|f(A^{-1})\|  \le Cf(\|A^{-1}\|),
$$
for a constant $C$ independent of the complete Bernstein function $f$.  On the other hand, taking $t_1=t_2=\gamma=0$ in \eqref{intineq}  we obtain a reversed inequality
$$
\|f(A^{-1})\|  \ge c f(\|A^{-1}\|).
$$
The two inequalities together form an operator counterpart to Proposition \ref{prop.asymp}.
\end{remark}

\begin{remark}
Theorem \ref{interpol2} can be generalized in various ways.  For example, if $B$ is any bounded sectorial operator commuting with $T(t)$ for all $t>0$ (in particular, if $B=B(A)$), then the following version of \eqref{intineq} holds:
$$
\left\| T(t_1)B^{\gamma} f(B) \right\| \ge c \frac {\|T(t_1+t_2)B^{\gamma+1}\|}{\|T(t_2)B\|}  f(\|T(t_2)B\|)
$$
for $\gamma >0$.  It also holds for $\gamma \ge -1$ if $B$ is injective and $B^\gamma f(B)$ is a bounded operator (in particular, if $B = A^{-\delta}$ for $0<\delta\le1$).
\end{remark}

\subsection{Transference from resolvents to semigroups} \label{ss.transfer}

The final estimate of this section is for bounded semigroups on Hilbert space.  The following result shows how the effect of cancelling resolvent growth can be transferred to an estimate for the semigroup itself. When $B = A^{-\alpha}$, the result was obtained in \cite[Theorem 2.4]{BoTo10}. In our applications the operator $B$ will be a function of the generator $A$, such as the operator $W_{\alpha,\beta,\ell}(A)$ of Subsection \ref{ss.cancel}.  

\begin{theorem} \label{thm.CRbound}
Let $(T(t))_{t \ge 0}$ be a bounded $C_0$-semigroup on a Hilbert space $X$, with generator $-A$, and let $B : \dom(A) \to X$ be a linear operator which is bounded for the graph norm on $\dom(A)$, and such that $T(t)Bx=BT(t)x$ for all $t \ge 0$ and $x \in \dom(A)$.  Assume that
\begin{equation} \label{CRbnd}
\sup \left \{\|B (\lambda+A)^{-1}\|\suchthat \lambda\in\C_+ \right
\} <\infty.
\end{equation}
Then $T(t)B$ extends to a bounded linear operator (also denoted by $T(t)B$) on $X$ for each $t>0$, and $\|T(t)B\| = \O(t^{-1})$ as $t \to \infty$.
\end{theorem}

\begin{proof}
Let $x \in \dom(A)$.  For a fixed $\tau>0$ define $$f_\tau (t)=\begin{cases}  T(t)x,& 0 \le t \le \tau,\\
0, & t > \tau,
\end{cases}
$$
and
\begin{eqnarray*}
 \varphi_\tau(t)=( B T*f_\tau)(t) =
 \begin{cases} t B T(t)x, & 0  \le t \le \tau,\\
 \tau BT(t)x,& t > \tau.
 \end{cases}
 \end{eqnarray*}
Let $\widehat f_\tau$ and $\widehat \varphi_\tau$ be the Laplace transforms of these functions, so that $\widehat \varphi_\tau(\lambda) = B(\lambda+A)^{-1} \widehat f_\tau(\lambda)$ for $\lambda \in \C_+$.  By Plancherel's theorem, for $a>0$,
 \begin{eqnarray*}
 \int_{\mathbb R}\|\widehat
 \varphi_\tau(a+is)\|^2 \, \ud{s} &=& 2\pi \int_{\mathbb R}e^{-2a t}
 \|\varphi_\tau(t)\|^2 \, \ud{t} \\
 &\ge& 2\pi \int_{0}^{\tau} t^2 e^{-2at}
 \|B T(t)x\|^2 \, \ud{t}.
 \end{eqnarray*}
 Letting $C$ be the finite supremum in \eqref{CRbnd}, we have
 \begin{eqnarray*}
\int_{\mathbb R}\|\widehat
 \varphi_\tau(a+is)\|^2 \, \ud{s} &=& \int_{\mathbb R}\|B (a+is + A)^{-1}\widehat
 f_\tau(a+is)\|^2 \, \ud{s} \\
 &\le& C^2 \int_{\mathbb R} \|\widehat  f_\tau(a+is)\|^2 \, \ud{s}\\
 &=&C^2 2 \pi \int_{0}^{\tau} e^{-2a t} \|T(t)x \|^2 \, \ud{t},
  \end{eqnarray*}
again by Plancherel's theorem.   These two inequalities imply that
 \begin{eqnarray*}
 C^2 \int_{0}^{\tau} e^{-2a t} \|T(t)x\|^2 \, \ud{t} \ge
 \int_{0}^{\tau} t^2 e^{-2a t} \|B T(t)x\|^2 \, \ud{t}.
 \end{eqnarray*}
Letting $a\to0+$ one gets
  \begin{equation*} \label{CRces1}
 C^2 \int_{0}^{\tau}  \|T(t)x\|^2 \, \ud{t} \ge
 \int_{0}^{\tau} t^2  \|B T(t)x\|^2\, \ud{t}.
 \end{equation*}
Let  $K=\sup_{t\ge 0}\| T(t)\|$.  Then
 \begin{equation} \label{CRces}
C^2 K^2 \|x\|^2  \ge
 \frac{1}{\tau}\int_{0}^{\tau} t^2   \|B T(t)x\|^2 \, \ud{t}.
 \end{equation}
Hence,  for any $y \in X$,
 \begin{eqnarray*}
 |\langle \tau T(\tau)Bx , y\rangle| &=& \left|\frac{2}{\tau} \int_{0}^{\tau} t \langle  T(t)Bx, T^*(\tau-t)y\rangle \, \ud{t} \right| \nonumber \\
 &\le&  \left \{\frac{2}{\tau}\int_{0}^{\tau} t^2 \| T(t)B x\|^2 \, \ud{t}
 \right\}^{\frac12}  \left\{\frac{2}{\tau}\int_{0}^{\tau}\|T^*(\tau-t)y\|^2 \, \ud{t} \right \}^{\frac12}\nonumber \\
 &\le&  \left \{\frac{2}{\tau}\int_{0}^{\tau} t^2 \| T(t)B x\|^2 \, \ud{t}
 \right\}^{\frac12} \sqrt2 K \|y\| \\
 &\le& 2 C K^2 \|x\| \|y\|.\nonumber
   \end{eqnarray*}
This implies that $T(\tau)B$ has a bounded extension to $X$ with norm at most $2CK^2/\tau$.
\end{proof}

\begin{remark}
We do not know whether there is a converse of Theorem \ref{thm.CRbound}, for example whether \eqref{CRbnd} holds whenever $B$ is a bounded operator on $X$, commuting with $T(t)$ and satisfying $\|T(t)B\| = \O(t^{-1})$ as $t \to \infty$.  There is a result in function theory \cite[Lemma 2.5]{Har10} which says that under these assumptions  the boundary function of $\langle B (\cdot + A)^{-1}x,y \rangle$ on $i\mathbb{R}$ lies in  $\operatorname{BMO}(i\mathbb R)$.

The crux of the proof of Theorem \ref{thm.CRbound} is the estimate (\ref{CRces}) showing that the Ces\'aro means of the scalar function $t \mapsto t^2 \|BT(t)x\|^2$ are bounded.  For a positive measurable function boundedness of its Abel means is equivalent to boundedness of its Ces\'aro means.  On Hilbert space, the Abel means of this function are
$$
a \int_{\mathbb R_+} t^2 e^{-at} \|BT(t)x\|^2 \, \ud{t} = \frac{\alpha}{\pi} \int_\R \big\|B (\alpha + is + A)^{-2} x \big\|^2 \, \ud{s}, \quad a=2\alpha>0,
$$
by Plancherel's theorem.  Thus the assumption \eqref{CRbnd} can be replaced by
$$
\alpha \int_\R \|B (\alpha + is + A)^{-2} x \|^2 \, \ud{s} \le C\|x\|^2, \qquad \alpha>0, x \in X.
$$ 

On Hilbert space, the bounded operator-valued function $B(\cdot+A)^{-1}$ is an $L^2(\R_+;X)$-Laplace multiplier, as shown by means of Plancherel's Theorem.  Theorem \ref{thm.CRbound} holds for a $C_0$-semigroup on a Banach space $X$ provided that $B (\cdot +A)^{-1}$ is an $L^p(\R_+;X)$-Laplace multiplier for some $p \in [1,+\infty)$, in the sense that the convolution operator $f \mapsto BT * f$ is bounded on $L^p(\R_+,X)$.  The proof is as in Theorem \ref{thm.CRbound} with an application of H\"older's inequality replacing the Cauchy-Schwarz inequality.
 \end{remark}

\section{Singularity at infinity}  \label{s.infinity}

In this section we shall consider a bounded $C_0$-semigroup $(T(t))_{t\ge0}$, with generator $-A$, on a Hilbert space $X$ under the assumption that $\sigma(A) \cap i\mathbb{R}$ is empty.  First we recall and elaborate Theorem \ref{chdu}, where $X$ is any Banach space but the other assumptions are the same.

The spectral assumption that $\sigma(A) \cap i\R$ is empty is equivalent to the property that 
\begin{equation} \label{TtA}
\lim_{t\to\infty} \|T(t)(\omega+A)^{-1}\| = 0
\end{equation}
for any $\omega \in \rho(-A)$ \cite[Theorem 4.4.14]{ABHN01}.   We choose to take $\omega=0$.  
 
The rate of decay in \eqref{TtA} is closely related to the growth of the resolvent of $A$ on $i\R$.  Let $M$ be a function such that 
\begin{equation}\label{m}
\|(is+A)^{-1}\|\le M(s), \qquad s \in \mathbb R.
\end{equation}
We shall always make the natural assumption that $M$ is even, so we shall consider $M$ as being a function on $\R_+$.  It is also natural to assume that $M$ is increasing and continuous. 
 
Define also
\begin{equation} \label{defMlog}
M_{\log}(s) = M(s)\big( \log(1 + M(s)) + \log(1+s)\big), \qquad s \ge 0.
\end{equation}
It is shown in \cite[Theorem 1.5]{BaDu08} (see also \cite[Theorem 4.4.14]{ABHN01}) that 
\begin{equation} \label{genest+}
 \|T(t)A^{-1}\| = \O \Big( \frac{1}{M_{\log}^{-1}(ct)} \Big), \qquad t \to \infty,
\end{equation}
for any $c \in (0,1)$.

The smallest function $M$ satisfying \eqref{m} and our other assumptions is given by
\begin{equation}
M(s) = \sup \left\{\|(ir+A)^{-1}\| \suchthat |r| \le s \right\}, \qquad s\ge0. \label{defM}
\end{equation}
For this choice of $M$ it is a simple consequence of \cite[Proposition 1.3]{BaDu08} (see also \cite[Theorem 4.4.14]{ABHN01}) that there exist constants $c,C>0$ such that
\begin{equation} \label{genest-}
 \|T(t)A^{-1}\|  \ge \frac{c} {M^{-1}(Ct)}
\end{equation}
for all sufficiently large $t$.  Here we assume that $\lim_{s\to\infty} M(s) = \infty$ and $M^{-1}$ may be any right inverse of $M$.

The estimates \eqref{genest+} and \eqref{genest-} are both valid for bounded semigroups on any Banach space.  They raise the question whether, or when, it is possible to improve \eqref{genest+} to
\begin{equation} \label{genest++}
 \|T(t)A^{-1}\| = \O \Big( \frac{1}{M^{-1}(ct)} \Big), \qquad t \to \infty.
\end{equation}
In some cases, for example if $M(s) = e^{\alpha s}$ for $\alpha >0$, \eqref{genest+} and \eqref{genest++} are equivalent.  In many cases, each estimate is independent of $c$. In the case when $M(s) = s^\alpha$ for $\alpha>0$, the two estimates differ by a logarithmic factor.  In this case, \eqref{genest+} is optimal for arbitrary Banach spaces, but \eqref{genest++} holds when $X$ is a Hilbert space \cite{BoTo10}.  However, for some $M$ one cannot make this improvement even for normal operators on Hilbert space \cite[Example 4.4.15]{ABHN01}.  We give a more detailed analysis of normal semigroups in Subsection \ref{ss.normal}. 

In later subsections, we consider cases when $X$ is a Hilbert space and $M$ is regularly varying. 
In Subsection \ref{ss.cancel} we shall show that if $-A$ generates a bounded 
$C_0$-semigroup on a Banach space $X$ and $\sigma(A) \cap i\mathbb R$ is empty,
then for regularly varying functions $M$, the property \eqref{m}
is equivalent to
\begin{equation} \label{cancelled}
\|(\lambda+A)^{-1} f_M(A)\| \le C, \qquad \lambda \in \C_+,
\end{equation}
where $f_M(A)$ is defined by the extended functional calculus of Theorem \ref{calculus} for an appropriate function $f_M$ related to the classes of  Bernstein functions and Stieltjes functions discussed in Subsection \ref{subsectfun}.  In Subsections \ref{ss.slower} and \ref{ss.faster}, we pass from \eqref{cancelled} towards \eqref{genest++}.

\subsection{Normal semigroups}\label{ss.normal}
The following result gives the precise condition on $M$ which is both necessary (if $M$ is defined by \eqref{defM}) and sufficient for \eqref{genest++} to be valid for a semigroup of normal operators on Hilbert space.  In fact, it holds more generally for any bounded $C_0$-semigroup for which the norms of all the associated bounded operators are determined by $\sigma(A)$, that is,
$$
\|T(t) r(A)\| = \sup \big\{|e^{-\lambda t}r(\lambda)| \suchthat \lambda \in \sigma(A)\big\}
$$
for every rational function $r$ whose poles are outside $\sigma(A)$ and which is bounded at infinity.  This includes multiplication semigroups on $L^p$-spaces and spaces of continuous functions.  We call such $C_0$-semigroups {\it quasi-multiplication} semigroups.

\begin{proposition} \label{normalinf} 
Let $(T(t))_{t\ge0}$ be a quasi-multiplication semigroup on a Banach space $X$ with generator $-A$.  Assume that $\sigma(A) \subset \C_+$ and $\inf \{\Re\lambda \suchthat \lambda \in\sigma(A)\} = 0$.  Let $M$ be defined by \eqref{defM} and let $c>0$.  The following are equivalent:
\begin{enumerate}[\rm(i)]
\item  \label{normalcond1} There exists $C$ such that
\begin{equation} \label{normalest1}
\|T(t)A^{-1}\| \le \frac {C}{M^{-1}(ct)}, \qquad t\ge c^{-1} M(0);
\end{equation}
\item  \label{normalcond2} There exists $B$ such that
\begin{equation} \label{normalres}
\frac{M(\tau)}{M(s)} \ge c \log \left( \frac {\tau}{s} \right) - B, \qquad \tau>0,\, s\ge1.
\end{equation}
\end{enumerate}
\end{proposition}

\begin{proof}  Note first that $\|T(t)\|=1$ for all $t\ge0$, and
$$
{M(s)}^{-1} = \min \left\{|\mu - ir| \suchthat \mu \in \sigma(A), |r| \le s \right\}  \to  0, \quad s \to \infty.
$$
Now (\ref{normalest1}) is equivalent to
$$
\frac{e^{-t\alpha}}{|\mu|} \le \frac {C}{M^{-1}(ct)}, \qquad \mu =  \alpha + i\beta \in \sigma(A), \; t\ge c^{-1}M(0).
$$
This may be rewritten as
\begin{equation} \label{normalest2}
t\alpha \ge \log \left( \frac{M^{-1}(ct)}{C|\mu|} \right)
\end{equation}
for all such $\mu$ and $t$.

Assume that (\ref{normalcond1}) holds.  Let $t\ge c^{-1}M(0)$, and put $\tau = M^{-1}(ct)$.  From (\ref{normalest2}),
$$
M(\tau) \ge \frac{c}{\alpha} \log \left( \frac {\tau}{C|\mu|} \right).
$$
Given $s \ge 1$, take $\mu = \alpha + i\beta \in \sigma(A)$ such that
${M(s)}^{-1} = |\mu - ir|$ for some $|r| \le s$.  Then $\alpha \le k$, where $k = M(0)^{-1}$, and $|\beta| \le s+k$.  So
$$
\frac {M(\tau)}{M(s)} \ge \frac {c|\alpha + i(\beta -r)|}{\alpha} \log \left( \frac{\tau}{C(s+2k)} \right) \ge c \log  \left( \frac{\tau}{C(s+2k)} \right).
$$
Since  $\log (C(s+2k)) - \log s \le  \log ((2k+1)C)$ for $s\ge1$, it follows that (\ref{normalres}) holds whenever $\tau$ is in the range of $M^{-1}$.  For other values of $\tau$ one can apply the above with $\tau$ replaced by $\tau_n := M^{-1}(M(\tau)  + n^{-1}) > \tau$, and let $n\to \infty$.

Now assume that (\ref{normalcond2}) holds.  Given $t \ge c^{-1} M(0)$ and $\mu= \alpha+ i\beta \in \sigma(A)$ with $|\beta|>1$, take
$$
\tau = M^{-1}(ct), \qquad s = |\beta|.
$$
By \eqref{normalres},
$$
\frac {ct}{M(|\beta|)} \ge c \log \left( \frac {M^{-1}(ct)}{|\beta|} \right) - B.
$$
Rearranging this, using $\alpha M(|\beta|) \ge 1$ and $|\mu| \ge |\beta|$, and putting $C= \exp(B/c)$ gives (\ref{normalest2}), provided that $|\beta|\ge1$.

If there exist $\mu= \alpha+i\beta \in \sigma(A)$ with $|\beta| \le 1$, let
$$
\varepsilon = \inf \left\{ \alpha \suchthat \alpha + i\beta \in \sigma(A), |\beta|\le1 \right\} > 0.
$$
Putting $s = \sqrt\tau$ and then $s=1$ in (\ref{normalres}) shows that
\begin{equation} \label{normalest3}
M(\tau) \ge \left( \frac c2 \log\tau - B \right) M(\sqrt\tau) \ge \left( \frac c2 \log\tau - B \right)^2 M(1)\ge \frac {c^2(\log\tau)^2}{5\varepsilon}
\end{equation}
for all sufficiently large $\tau$.  Putting $\tau = M^{-1}(ct)$ shows that
$$
ct \ge \frac {c^2(\log M^{-1}(ct))^2}{5\varepsilon},
$$
and hence 
$$
\varepsilon t \ge \log M^{-1}(ct)
$$
for all sufficiently large $t$.   We can then choose $C$ sufficiently large that
$$
 \varepsilon t \ge \log \left( \frac {M^{-1}(ct)}{C \varepsilon} \right)
$$
whenever $t\ge c^{-1}M(0)$.  Then (\ref{normalest2}) holds for $\mu = \alpha + i\beta \in \sigma(A)$ with $|\beta| \le 1$ and $t \ge c^{-1}M(0)$.  Hence (\ref{normalest1}) holds.
\end{proof}

It is clear that \eqref{normalres} implies that $M$ grows at least logarithmically.  This corresponds to the elementary fact that, for a quasi-multiplication semigroup, $\|T(t)A^{-1}\|$ cannot decrease faster than exponentially.  The estimate \eqref{normalest3} shows that (\ref{normalres}) implies that $M(\tau)$ grows at least as fast as $(\log\tau)^2$.  More generally, any slowly varying function $M$ fails to satisfy (\ref{normalres}).  Given $B$ and $c>0$, choose $\lambda = e^{(B+2)/c}$.  Then \eqref{normalres} implies that
$$
\frac{ M(\lambda s)}{M(s)} \ge 2,  \qquad s \ge 1,
$$
so $M$ is not slowly varying.  In particular, the rate of decay of $\|T(t)A^{-1}\|$ for a quasi-multiplication semigroup cannot be given by \eqref{normalest1} unless the rate is slower than $\exp(-c t^{1/n})$ for all $n$.  

The following example shows the rate of decay for normal semigroups when $M$ grows logarithmically.

\begin{example} \label{log-growth}
Let $(T(t))_{t\ge0}$ be a quasi-multiplication semigroup on a Banach space $X$, with generator $-A$ such that
$$
\sigma(A) = \left\{ \frac{1}{\log s} + is \suchthat s\ge2 \right\}.
$$
Then
\begin{eqnarray*}
\log s \le \|(is-A)^{-1}\| &=& M(s) \le \log(s+1), \qquad s\ge2, \\
\frac{1}{M^{-1}(ct)} &\sim& e^{-ct}.
\end{eqnarray*}
However,
$$
\|T(t)A^{-1}\| = \sup \left\{ \frac{ \exp(-t/\log s)}{s} : s\ge2 \right\} = e^{-2\sqrt t},  \qquad t\ge (\log 2)^2.
$$
\end{example}

If $M(s) = (1+s)^\alpha$ for some $\alpha>0$, then (\ref{normalres}) holds.  More generally, if $M$ is regularly varying with index $\alpha>0$, then $M$ satisfies (\ref{normalres}).

On the other hand, rapid growth of $M$ does not on its own imply that (\ref{normalres}) holds.  See \cite[Example 4.4.15]{ABHN01}.

\subsection{Cancelling resolvent growth} \label{ss.cancel}

Here we shall show how regularly varying growth of $\|(is+A)^{-1}\|$ as $|s|\to\infty$ can be cancelled by restricting to the range of a suitable operator.  In the case of purely polynomial growth this was achieved by taking a (negative) fractional power of $A$ \cite[Lemma 3.2]{LaSh01}, \cite[Lemma 2.3]{BoTo10}, but we shall need a more complicated function of $A$.

The following lemma of Phragm\'en-Lindel\"of type may be known, but we have not been able to trace it in the literature.  The formulation stated here is more general than is needed for this section, but we shall need the stronger form, or variants of it, when we consider singularities at zero in Sections \ref{s.zero2} and \ref{sectwosing}.

 \begin{lemma}\label{phragmen}
Let $Y$ be a Banach space and $f :\overline{\mathbb C}_+ \setminus \{0\} \to Y$ be a function which is continuous on $\overline{\mathbb C}_+ \setminus \{0\}$, holomorphic in $\mathbb C_+$,  and bounded on  $i\mathbb R\setminus \{ 0\}$.   If there exists $C>0$ such that
\begin{equation*}
\|f (z)\|\le \frac{C}{\Re z}, \quad z \in \C_+,
 \end{equation*}
 then $f$ is bounded in $\overline{\mathbb C}_+ \setminus \{0\}$.
\end{lemma}

\begin{proof}Let $K = \sup \{\|f(z)\|: z\in i\R, z\ne0\}$.
For fixed $r,R$ with $0< r < 1 < R$, let
$$ A_{r,R} =\{z \in \overline{\mathbb C}_+: r \le |z|\le R\}.$$
Consider
\begin{equation*}
g_{r,R}(z):= \frac{z}{1+z}  \left(1+\frac{z^2}{R^2}\right)\left(1+\frac{r^2}{z^2}\right) f(z), \qquad z \in A_{r,R}.
\end{equation*}
Let
\begin{gather*}
\gamma_r=\{z \in \mathbb C_+: |z|=r\}, \quad \Gamma_R=\{z\in \mathbb C_+: |z |=R\},\\ J_{r,R}=\{is : r \le |s| \le R\},
\end{gather*}
so that
\begin{equation*}
\partial A_{r,R}=\gamma_r\cup\Gamma_R\cup J_{r,R}.
\end{equation*}
If $z \in \gamma_r$  then
\begin{equation*}
\left|\frac{z}{1+z}\right|\le \frac{r}{1-r}, \quad  \left|1+\frac{z^2}{R^2}\right|\le 2, \quad \left|1+\frac{r^2}{z^2}\right|= \frac{2 \Re z}{r},
\end{equation*}
so that
\begin{equation*}
\|g_{r,R}(z)\| \le \frac{r}{1-r}  \frac{4 \Re z}{r} \frac{C}{\Re z} =\frac {4 C}{1-r}.
\end{equation*}
Similarly, if $z \in \Gamma_R$ then
\begin{equation*}
\left|\frac{z}{1+z}\right|\le \frac{R}{R-1}, \quad  \left|1+\frac{z^2}{R^2}\right|=\frac{2 \Re z}{R}, \quad \left|1+\frac{r^2}{z^2}\right|\le 2,
\end{equation*}
thus
\begin{equation*}
\|g_{r,R}(z)\| \le  \frac{4C}{R-1}.
\end{equation*}
Finally, if $z \in J_{r,R}$ then
\begin{equation*}
\|g_{r,R}(z)\| \le 4K.
\end{equation*}
Applying the maximum principle to $g_{r,R}$ over $A_{r,R}$, letting $r \to 0+$ and $R \to \infty$  and setting $C'= 4\max(C,K)$, we infer that
\begin{equation*}
\sup_{z \in \mathbb C_+}\left\|\frac{z}{1+z}f(z)\right\| \le  C'.
\end{equation*}
In particular, $f$ is bounded on $\{z \in \mathbb C_+: |z| \ge 1\}$ and
\begin{equation*}
\|f(z)\| \le \frac{2C'}{|z|}, \qquad |z| \le 1.
\end{equation*}
Now, by applying  a standard Phragm\'en-Lindel\"of principle for half-planes \cite[Corollary VI.4.2]{Con78}  to $h(z):=f(z^{-1})$ for $z \in \mathbb C_+$, we conclude that $f$ is bounded on $\{z \in \mathbb C_+: |z| \le 1\}$ and this completes the proof.
\end{proof}

\begin{definition}
Let $-A$ be the generator of a bounded $C_0$-semigroup, and assume that $A$ is invertible.  Let $\beta \in (0,1]$, and let $\ell$ be a slowly varying function such that $g : s \mapsto s^{1-\beta} \ell(s)$ is increasing on $\R_+$.  Let $S_g$ be the Stieltjes function associated with $g$.  For $\alpha \ge \beta$, define
\begin{equation}\label{stw2}
W_{\alpha,\beta,\ell} (A):= A^{-(\alpha -\beta)}\int_{0+}^{\infty} (s+ A)^{-1} \, \ud\big( s^{1-\beta} \ell(s)\big) = A^{-(\alpha-\beta)} S_g (A).
\end{equation}
See Example \ref{stfng} and Remark \ref{calcrem}, (\ref{calcremciii}) and (\ref{calcremcii}), for the definition of $S_g$, the convergence of the integral and compatibility with the extended holomorphic functional calculus.

For $0<\alpha<\beta$, define
\begin{eqnarray}\label{stw3}
W_{\alpha,\beta,\ell} (A)&:=& \int_{0+}^{\infty} A^{\beta-\alpha} (s+A)^{-1} \,
\ud\big( s^{1-\beta} \ell(s)\big) \\
&=& \int_{0}^{\infty} A^{\beta-\alpha} (s+A)^{-2} s^{1-\beta} \ell(s) \, \ud{s}.  \label{stw4}
\end{eqnarray}
Since $A$ is sectorial and invertible,
\begin{equation*}
\|(s+A)^{-1}\| \le \frac{C}{1+s} \,, \qquad \|A(s+A)^{-1}\| \le C.
\end{equation*}
The moment inequality (\ref{momineq3}) gives 
$$
\|A^{\beta-\alpha} (s+A)^{-1}\| \le \frac{C}{(1+s)^{1+\alpha-\beta}} \,.
$$
Then convergence of the integrals in (\ref{stw3}) and (\ref{stw4}) follows from the discussion of Example \ref{stfng}.  In particular, $W_{\alpha,\beta,\ell}(A)$ is a bounded operator on $X$, and $A^{-(\beta-\alpha)} W_{\alpha,\beta,\ell}(A) = S_g(A)$, by the product rule (Theorem \ref{calculus}(\ref{calciv})).
\end{definition}

When $\ell(s)=1$ for all $s>0$, $W_{\alpha,\beta,\ell}(A) = A^{-\alpha}$ for all $\beta \in (0,1)$. The following result is already known in that special case \cite[Lemma 3.2]{LaSh01}, \cite[Lemma 2.3]{BoTo10}.

\begin{theorem} \label{thm.bounded}
Let $-A$ be the generator of a bounded $C_0$-semigroup $(T(t))_{t \ge 0}$ on a Banach space $X$ such that $i\R \subset \rho(A)$. Let $\ell$ be a slowly varying function on $\mathbb R_+$, let $\alpha >0$ and $\beta \in (0,1]$. Assume that $g: s \mapsto s^{1-\beta}\ell(s)$ is increasing.  The following statements are equivalent: \begin{enumerate}[\rm(i)]
\item  \label{boundedd1} $\displaystyle{ \|(is+A)^{-1}\|=\O \left( \frac{|s|^\alpha}{\ell(|s|)} \right), \qquad |s| \to \infty}$;
\item \label{boundedd3} $\displaystyle \sup_{z \in \mathbb C_+} \|(z+A)^{-1} W_{\alpha,\beta,\ell}(A)\|< \infty$.
\end{enumerate}
\end{theorem}

\begin{proof}  For $z \in \rho(-A)$ and $k\in\N$, observe that 
\begin{equation} \label{resident3}
(z+A)^{-1}A^{-k} =\frac{1}{(-z)^k} (z+A)^{-1}
-\sum_{i=0}^{k-1}\frac{1}{(-z)^{i+1}}
A^{-(k-i)}.
\end{equation}
We shall show that  (\ref{boundedd1}) is equivalent to
\begin{equation}
\| A^{-(\alpha-\beta)}(is+A)^{-1} \| = \O \left( \frac{|s|^\beta}{\ell(|s|)} \right), \qquad |s| \to \infty.\label{bounded-ses}
\end{equation}
Write $\alpha-\beta = m + \gamma$ where $m \in \Z$, $m\ge-1$ and $0 \le \gamma < 1$.

Assume that (\ref{boundedd1}) holds.  Then 
\begin{equation*} \label{bounded-half}
\| A(is+A)^{-1}\|= \|I - is (is+A)^{-1}\| = \O \left(\frac{|s|^{\alpha+1}}{\ell(|s|)} \right), \qquad |s| \to \infty.
\end{equation*}
By the moment inequality (\ref{momineq3}),
\begin{equation*} \label{bounded-1}
\| A^{1-\gamma}(is+A)^{-1}\| = \O \left(\frac {|s|^{\alpha+1-\gamma}}{\ell(|s|)} \right), \qquad |s| \to \infty.
\end{equation*}
Using (\ref{resident3}) for $k=m+1$, it follows that  
\begin{eqnarray*}
\| A^{-(\alpha-\beta)}(is+A)^{-1} \|&=& \|A^{1-\gamma} A^{-(m+1)} (is+A)^{-1} \| \nonumber \\ 
&=& \O \left(\frac{|s|^{\alpha+1-\gamma-(m+1)}}{\ell(|s|)} \right) \nonumber \\
&=&  
\O \left( \frac{|s|^\beta}{\ell(|s|)} \right), \qquad |s| \to \infty.
\end{eqnarray*}

Conversely, if $\alpha\ge\beta$, (\ref{bounded-ses}) implies that
\begin{equation*} \label{bounded-sesqui}
\|A^{-(\alpha-\beta-1)}(is+A)^{-1} \|=\O \left(\frac{|s|^{\beta+1}}{\ell(|s|)} \right), \qquad |s| \to \infty.
\end{equation*}
Since $m = (1-\gamma)(\alpha-\beta) + \gamma (\alpha -\beta -1)$ and $(1-\gamma) \beta + \gamma (\beta+1) = \alpha-m$, the moment inequality (\ref{momineq3}) gives
$$
\|A^{-m} (is+A)^{-1} \|=\O \left(\frac{|s|^{\alpha-m}}{\ell(|s|)} \right), \qquad |s| \to \infty.
$$
Then (\ref{resident3}) for $k=m$ gives (\ref{boundedd1}).

If $\alpha<\beta$, then $m=-1$ and (\ref{bounded-ses}) implies that
\begin{eqnarray*} 
\|A^{-(\alpha-\beta+1)}(is+A)^{-1} \| &=& \frac{\left\|A^{-(\alpha-\beta+1)} - A^{-(\alpha-\beta)}(is+A)^{-1} \right\|}{|s|} \\
&=& \O \left(\frac{|s|^{\beta-1}}{\ell(|s|)} \right), \qquad |s| \to \infty.
\end{eqnarray*}
Since $0 = \gamma(\alpha-\beta) + (1-\gamma) (\alpha -\beta+1)$ and $\gamma \beta + (1-\gamma) (\beta-1) = \alpha$, the moment inequality (\ref{momineq3}) gives (\ref{boundedd1}).  This completes the proof that (\ref{boundedd1}) is equivalent to (\ref{bounded-ses}).

Next, we consider the Stieltjes function $S_g$ associated with $g$, defined as in Example \ref{stfng}.  By Karamata's Theorem \ref{karamata}(\ref{kar1}) with  $\sigma =\beta$,
\begin{equation} \label{stasymp}
S_{g}(\lambda) =\O \left(\lambda^{-\beta}\ell(\lambda)\right), \qquad \lambda \to \infty.
\end{equation}
By Proposition \ref{prop.asymp},
\begin{equation}\label{equiv}
|S_{g}(-is)| =\O \left( |s|^{-\beta} \ell(|s|) \right), \qquad |s| \to \infty.
\end{equation}
Now we observe that 
\begin{eqnarray}
{(is+A)^{-1}W_{\alpha,\beta,\ell}(A)}
 &=& \int_{0+}^\infty (is+A)^{-1} A^{-(\alpha-\beta)} (\lambda+A)^{-1} \, \ud \big(\lambda^{1-\beta} \ell(\lambda) \big) \nonumber \\
&=& A^{-(\alpha-\beta)} (is+A)^{-1} \int_{0+}^{\infty} \frac{\ud\big(\lambda^{1-\beta}\ell(\lambda) \big)}{\lambda-is} \nonumber \\
&& \phantom{X} - \int_{0+}^{\infty} \frac{1}{\lambda-is} A^{-(\alpha-\beta)}(\lambda+A)^{-1}\, \ud\big( \lambda^{1-\beta}\ell(\lambda) \big) \nonumber\\
&=& \label{formula} S_g(-is) A^{-(\alpha-\beta)} (is+A)^{-1} \\
&& \phantom{X} -\int_{0+}^{\infty}\frac{1}{\lambda-is} A^{-(\alpha-\beta)}(\lambda+A)^{-1}\, \ud\big( \lambda^{1-\beta}\ell(\lambda) \big). \nonumber
\end{eqnarray}
The last integral is bounded, uniformly for $|s|\ge1$, by the same argument as for (\ref{stw3}) together with the fact that $|\lambda-is| \ge 1$.

Now assume that (\ref{boundedd1}) holds, so (\ref{bounded-ses}) holds.  By (\ref{bounded-ses}) and (\ref{equiv}), 
$$
\|S_g(-is) A^{-(\alpha-\beta)} (is+A)^{-1}\| \le C, \qquad |s|\ge1.
$$
Hence $\|(is+A)^{-1}W_{\alpha,\beta,\ell}(A)\|$ is bounded for $|s|\ge1$, and then for all real $s$.  We now apply Lemma \ref{phragmen}, with
$$
f(z):= (z+A)^{-1} W_{\alpha,\beta,\ell}(A),
$$
and we deduce that  (\ref{boundedd3}) holds.

Conversely, assume that (\ref{boundedd3}) holds.   Letting $z\to is$ and using (\ref{formula}) shows that $\|S_g(-is) A^{-(\alpha-\beta)} (is+A)^{-1} \|$ is bounded for $s \in \mathbb{R}, \, |s|>1$.  Proposition \ref{prop.asymp} and Theorem \ref{karamata}(\ref{kar1}) show that (\ref{bounded-ses}) holds.  Then (\ref{boundedd1}) follows.
\end{proof}

\subsection{Resolvent growth slower than $s^\alpha$}  \label{ss.slower}
In this and the next subsection we consider cases when $-A$ generates a bounded $C_0$-semigroup $(T(t))_{t\ge0}$ on a Hilbert space $X$, $\sigma(A) \cap i\R$ is empty and
\begin{equation} \label{reslest}
\|(is+A)^{-1}\| = \O \left( \frac {|s|^\alpha}{\ell(|s|)} \right), \qquad |s|\to\infty,
\end{equation}
where $\ell$ is slowly varying and monotonic.  The upper bound for $\|T(t)A^{-1}\|$ in (\ref{genest+}) is then valid for $M(s) = C s^\alpha/\ell(s)$.   Then
$$
M_{\log}(s) \sim  C(1+\alpha) \frac{s^\alpha \log s}{\ell(s)} = C(1+\alpha)(M.{\log})(s),
$$
so we may replace $M_{\log}$ by $M.{\log}$ in (\ref{genest+}).

If we put $k(s) = 1/\ell(s^{1/\alpha})$, then Proposition \ref{regvarinv} gives
$$
(M.{\log})^{-1}(t) \sim t^{1/\alpha} (k.\log)^\#(t)^{1/\alpha}.
$$
Thus we obtain that
\begin{equation} \label{klogest}
\|T(t) A^{-1}\| = \O \left( \big( t(k.\log)^\#(t)\big)^{-1/\alpha} \right), \qquad t\to\infty.
\end{equation}
When $k.\log$ is dB-symmetric, this becomes
\begin{equation}\label{llogest}
\|T(t) A^{-1}\| = \O \left( \frac{(\log t)^{1/\alpha}} {t\ell(t^{1/\alpha})^{1/\alpha}}  \right), \qquad t \to \infty.
\end{equation}
On the other hand, if we have
$$
\|(is+A)^{-1}\| \ge c \frac {|s|^\alpha}{\ell(|s|)}
$$
for large $|s|$, then (\ref{genest-}) and Proposition \ref{regvarinv}(\ref{rvi2}) give
\begin{equation*} \label{llowest}
\|T(t) A^{-1}\| \ge \frac {c } {tk^\#(t^{1/\alpha})^{1/\alpha} } \,.
\end{equation*}
Thus, assuming (\ref{reslest}), the optimal result would be to establish that
\begin{equation} \label{kest}
\|T(t) A^{-1}\| = \O \left( \big( t k^\#(t) \big)^{-1/\alpha} \right), \qquad t \to \infty,
\end{equation}
or, assuming that $\ell$ is dB-symmetric, 
 \begin{equation} \label{lest}
\|T(t) A^{-1}\| = \O \left( \big(t\ell(t^{1/\alpha})\big)^{-1/\alpha} \right), \qquad t \to \infty.
\end{equation}

We now give one of our main results showing that (\ref{lest}) holds when $\ell$ is increasing, i.e., $\|(is+A)^{-1}\|$ grows slightly slower than $|s|^\alpha$.  The case when $\ell$ is decreasing will be considered in Subsection \ref{ss.faster}.

\begin{theorem} \label{thm.main}
Let $(T(t))_{t\ge0}$ be a bounded $C_0$-semigroup on a Hilbert space $X$, with generator $-A$.  Assume that $\sigma(A) \cap i \mathbb{R}$ is empty, and that
$$
\|(is+A)^{-1}\| = \O \left( \frac{|s|^\alpha}{\ell(|s|)} \right), \qquad |s| \to \infty,
$$
where $\alpha>0$ and $\ell$ is increasing and slowly varying.  Then
\begin{equation} \label{mainest}
\|T(t)A^{-1}\| = \O \left( \big(t\ell(t^{1/\alpha})\big)^{-1/\alpha} \right), \qquad t \to \infty.
\end{equation}
 
\end{theorem}

\begin{proof}  We can assume that $T(t) \ne 0$ for each $t>0$.  First we note a known upper bound for $\|T(t)A^{-1}\|$.  Since $\|(is+A)^{-1}\| = \O( |s|^\alpha)$, \cite[Theorem 2.4]{BoTo10} gives
\begin{equation} \label{BTest}
\|T(t)A^{-1}\| \le \frac {C}{t^{1/\alpha}} \,.
\end{equation}

By Theorem \ref{thm.bounded},
$$
\|(\lambda+A)^{-1} W_{\alpha,1,\ell}(A)\| \le C,  \qquad \lambda \in \mathbb{C_+}.
$$
By Theorem \ref{thm.CRbound},
$$
\|T(t) W_{\alpha,1,\ell}(A)\| \le \frac {C}{t}, \qquad t>0.
$$
Let $S_\ell$ be the Stieltjes function associated with $\ell$ (Example \ref{stfng}) and let $f_\ell$ be the complete Bernstein function defined by
$$
f_\ell(s) = S_\ell(1/s), \qquad s>0.
$$
Using \eqref{stw2} and Remark \ref{calcrem}(\ref{calcremcii}), we have
$$
\|T(t) A^{-(\alpha-1)} f_\ell(A^{-1})\| \le \frac {C}{t}, \qquad t>0.
$$
By Theorem \ref{interpol2}(\ref{int2a}), with $\gamma = 1-\alpha$, 
$$
 f_\ell( \|T(t) A^{-1}\|) \le  \frac {C \|T(t) A^{-1} \|}{t \|T(2t) A^{-\alpha}\|} \,.
$$
Letting $f_{\alpha,\ell}(s) = s^{\alpha-1} f_\ell(s)$, we have
\begin{equation} \label{fatest2}
 f_{\alpha,\ell}( \| T(t) A^{-1}\|) \le  \frac {C\|T(t) A^{-1} \|^\alpha}{t\|T(2t) A^{-\alpha}\|}\,.
\end{equation}
By Theorem \ref{karamata}(\ref{kar1}) with $g=\ell$ and $\sigma=1$, $S_\ell(s) \sim s^{-1} \ell(s)$ as $s \to \infty$.  Hence $f_{\alpha,\ell}(s) \sim s^\alpha \ell(1/s)$ as $s \to 0+$.  Let $k(s) = 1/\ell(s^{1/\alpha})$.  Then
$$
f_{\alpha,\ell}(s) \sim \frac {s^{\alpha}} {k(s^{-\alpha})}\,, \qquad s\to0+,
$$ 
By Proposition \ref{regvarinv}(\ref{rvi3}),
$$
f_{\alpha,\ell}^{-1}(s)  \sim \left(\frac {s} {k^\#(1/s)}\right)^{1/\alpha}, \qquad s\to0+.
$$
If the right-hand side of \eqref{fatest2} is sufficiently small, it follows that
\begin{equation} \label{Lest}
\|T(t)A^{-1}\| \le  \frac { C\|T(t) A^{-1} \|}{(tL(t) \|T(2t) A^{-\alpha}\|)^{1/\alpha}} \,,
\end{equation}
where
\begin{equation} \label{bigL}
L(t) = k^\# \left( t \frac {\|T(2t) A^{-\alpha} \|}{\|T(t) A^{-1}\|^\alpha} \right).
\end{equation}
Let $\psi(s) = (s k^\#(s))^{1/\alpha}$.   Since $\psi$ is regularly varying with positive index, we can choose $k^\#$ so that $\psi$ is strictly increasing and continuous (see the remarks following Definition \ref{regvar}; in fact, we can choose $k^\#$ to be increasing, since $k$ is decreasing).  Then
\begin{eqnarray} \label{psiest}
\frac {1}{\|T(t)A^{-1}\|} &\ge& c t^{1/\alpha} \, \frac {\|T(2t)A^{-\alpha}\|^{1/\alpha}}{\|T(t)A^{-1}\|} \, k^\# \left( t \frac {\|T(2t)A^{-\alpha}\|}{\|T(t)A^{-1}\|^\alpha} \right)^{1/\alpha} \\
&=& c \, \psi \left(t \frac {\|T(2t)A^{-\alpha}\|}{\|T(t)A^{-1}\|^\alpha} \right). \nonumber
\end{eqnarray}
If the right-hand side of \eqref{fatest2} is bounded away from zero, then \eqref{psiest} also holds for some $c>0$, since $\|T(t)A^{-1}\|$ is bounded and $\psi$ is bounded on bounded intervals.  Hence \eqref{psiest} holds for all $t>0$, for some $c>0$.

Since $\psi$ is (asymptotically equivalent to) an increasing function,
$$
t \frac {\|T(2t)A^{-\alpha}\|}{\|T(t)A^{-1}\|^\alpha} \le \psi^{-1} \left( \frac {C}{\|T(t)A^{-1}\|} \right) \,.
$$
By Proposition \ref{regvarinv}(\ref{rvi2}),
$$
\psi^{-1} (s) \sim s^\alpha k^{\#\#}(s^\alpha) \sim \frac {s^\alpha}{\ell(s)}, \qquad s \to\infty,
$$
so
\begin{equation} \label{impest}
\| T(2t)A^{-\alpha} \| \le \frac {C} {t \ell \big( \|T(t)A^{-1}\|^{-1} \big)} \,,
\end{equation}
for large $t$ and then for all $t>0$.  Then (\ref{BTest}) gives
$$
\|T(2t)A^{-\alpha}\|  \le \frac {C}{t \ell(t^{1/\alpha}))} \,.
$$
Applying Lemma \ref{interpoldec} with $B = A^{-1}$, $\gamma=\alpha$ and $\delta=1$, it follows that
$$
\|T(t)A^{-1}\|  \le C \|T(ct)A^{-\alpha}\|^{1/\alpha}  \le \frac {C}{(t\ell(t^{1/\alpha}))^{1/\alpha}} \,.
$$
\end{proof}

\begin{corollary}  \label{cor.main}
In addition to the assumptions of Theorem \ref{thm.main}, assume that $\ell$ is dB-symmetric.  Then
\begin{equation*} \label{genest+++}
\|T(t)A^{-1}\| = \O \Big( \frac{1}{M^{-1}(t)} \Big), \qquad t \to \infty,
\end{equation*}
where $M^{-1}$ is any asymptotic inverse of $s^\alpha/\ell(s)$.
\end{corollary}

\begin{proof}
By Proposition \ref{regvarinv}(\ref{rvi2}),
$$
M^{-1}(t) \sim t^{1/\alpha} k^\#(t)^{1/\alpha},
$$
where $k(t) = 1/\ell(t^{1/\alpha})$.  By Lemma \ref{regvarinv2},  $k^\#(t) \sim \ell(t^{1/\alpha})$ if (and only if) $\ell$ is dB-symmetric.
\end{proof}

\begin{remark} \label{iterest}
Corollary \ref{cor.main} establishes the optimal estimate  \eqref{genest++}  when $\ell$ is dB-symmetric.  However the upper bound (\ref{mainest}) is not as sharp as \eqref{genest++} when $\ell$ is not dB-symmetric.  Under the assumptions of Theorem \ref{thm.main} the proof takes the upper bound (\ref{BTest}) from \cite{BoTo10} and improves it to the upper bound (\ref{mainest}).  Given an estimate
\begin{equation} \label{genest4}
\|T(t)A^{-1}\| = \O\left((tm(t))^{-1/\alpha} \right),  \qquad t \to \infty,
\end{equation}
where $m$ is increasing and slowly varying, under the same assumptions the same argument shows that
$$
\|T(t)A^{-1}\| = \O \left( \left(t\ell(t^{1/\alpha}m(t)^{1/\alpha})\right)^{-1/\alpha} \right), \qquad  t \to \infty.
$$
This process can be iterated, so one obtains \eqref{genest4} for each of the following functions $m$:
$$
m(t) = 1, \; \ell(t^{1/\alpha}), \; \ell(t^{1/\alpha} \ell (t^{1/\alpha})^{1/\alpha}),  ......
$$
In some cases this process stabilises (up to asymptotic equivalence) after a finite number of iterations at the optimal estimate in Corollary \ref{cor.main}.  This is analogous to B\'ek\'essy's method of finding the de Bruijn conjugate of many slowly varying functions \cite[Proposition 2.3.5]{BGT} (see Example \ref{explog}).
\end{remark}

\begin{remark} \label{altstart}
An alternative to the known upper bound \eqref{BTest} is the estimate \eqref{klogest}.  In the context of Theorem \ref{thm.main}, Lemma \ref{lem.logpert}(\ref{logpert2}) shows that this gives
$$
\|T(t)A^{-1}\| = \O  \left( \left(\frac  {\log t}{tk^\#(t)} \right)^{1/\alpha}\right),  \qquad t\to\infty,
$$
where $k(s) = 1/\ell(s^{1/\alpha})$.  This is \eqref{genest4} with
$$
m(t) = \frac {k^\#(t)}{\log t}.
$$
Following Remark \ref{iterest}, we obtain
$$
\|T(t)A^{-1}\| = \O \left( \left( t \ell \left( \left( \frac {t k^\#(t)} {\log t} \right)^{1/\alpha} \right) \right)^{-1/\alpha} \right), \qquad  t \to \infty.
$$ 
In many cases, this is asymptotically equivalent to the optimal estimate \eqref{kest}. 
\end{remark}

\begin{example} \label{explog}
We take $\alpha=1$ (for simplicity of presentation) and $\ell(s) = \exp\left((\log s)^\beta\right)$ where $0<\beta < 1$ (see Example \ref{ex.dBinv}).  In this case, the process in Remark \ref{iterest}, starting from $m(t)=1$, stabilises at the optimal estimate
\begin{equation} \label{optest}
\|T(t)A^{-1}\| \le \frac{C}{M^{-1}(t)} \sim \frac {C}{t (1/\ell)^\#(t)}.
\end{equation}
When $1/2 < \beta < 3/4$, the process stabilises after two iterations at the optimal estimate:
$$
\|T(t)A^{-1}\| \le  \frac {C}{t} \exp \left(- \left((\log t)^\beta + \beta (\log t)^{2\beta-1} \right)\right).
$$
When $(n-1)/n < \beta < n/(n+1)$, $n$ iterations are needed.

Starting from
$$
m(t) = \frac {(1/\ell)^\#(t)}{\log t},
$$
as in Remark \ref{altstart}, the process stabilises at the optimal estimate \eqref{optest} after one iteration for every $\beta \in(0,1)$.
\end{example}

\begin{remark} \label{easier}
If \eqref{genest-} and \eqref{genest++} both hold for some regularly varying function $M$ of index $\alpha>0$, then the ratio
\begin{equation} \label{ratio}
\frac{ \|T(2t)A^{-1}\|}{\|T(t)A^{-1}\|} 
\end{equation}
is bounded away from zero (it is bounded above, for any bounded semigroup).

On the other hand, if \eqref{ratio} is bounded away from zero, then so is
\begin{equation*} \label{ratio+}
\frac{ \|T(2t)A^{-\alpha}\|}{\|T(t)A^{-1}\|^\alpha},
\end{equation*}
by Lemma \ref{interpoldec}.  Then the function $L$ in the proof of Theorem \ref{thm.main} is asymptotically equivalent to $k^\#$, and one can then pass easily from \eqref{bigL} to \eqref{genest+++}.  This argument would not require the assumption that $\ell$ is dB-symmetric, and it would not use Theorem \ref{interpol2} or \eqref{genest+}.  However we do not see any way to prove directly that \eqref{ratio} is bounded away from zero.

The assumptions of Theorem \ref{thm.main} are not sufficient to ensure that  \eqref{ratio} is bounded away from zero.  However, we may assume in addition that
\begin{equation} \label{resbd-}
\|(is+A)^{-1}\| \ge c M(|s|).
\end{equation}
If one can prove that \eqref{genest++} holds under this additional assumption, then it can be deduced that \eqref{genest+++} holds without the additional assumption.  This follows by means of a direct sum argument which is used in the proof of Theorem \ref{thm.faster} below.

Under the assumptions  of Theorem \ref{thm.main} and \eqref{resbd-}, we can obtain from \eqref{genest+} and \eqref{genest-} that
$$
\frac{ \|T(2t)A^{-\alpha}\|}{\|T(t)A^{-1}\|^\alpha} \ge \frac {c}{(\log t)^\alpha}.
$$
It then follows from \eqref{Lest} that
$$
\|T(t)A^{-1}\| = \O \left( \left( t k^\#\left(\frac {t}{(\log t)^\alpha} \right)\right)^{-1/\alpha} \right), \qquad t\to \infty.
$$
The direct sum argument can then be used to show that this estimate holds under the assumptions of Theorem \ref{thm.main}, without \eqref{resbd-}.  In general this estimate is slightly worse than (\ref{mainest}), but in many cases, for example $\ell(s) = (\log s)^\beta$,  one can recover the estimate of Corollary \ref{cor.main}.  
\end{remark}

\subsection{Resolvent growth faster than $s^\alpha$} \label{ss.faster}

Now we consider the case when \eqref{reslest} holds with $\ell$ decreasing so $\|(is+A)^{-1}\|$ grows slightly faster than $|s|^\alpha$.  Unfortunately, our result is not quite optimal in this case as a logarithmic term still appears in (\ref{mainest2}) below.  However it does improve (\ref{klogest}) and \eqref{llogest} as the logarithm has an arbitrarily small power.  
If $\ell$ is dB-symmetric, then $k^\#(t) = \ell(t^{1/\alpha})$ (see the proof of Corollary \ref{cor.main}), so (\ref{mainest2}) is directly comparable with (\ref{llogest}) and (\ref{lest}).   

\begin{theorem} \label{thm.faster}
Let $(T(t)_{t\ge0})$ be a bounded $C_0$-semigroup on a Hilbert space $X$, with generator $-A$.  Assume that $\sigma(A) \cap i \mathbb{R}$ is empty, and that
$$
\|(is+A)^{-1}\| = \O \left( \frac {|s|^\alpha}{\ell(|s|)} \right), \qquad |s|\to\infty,
$$
where $\alpha>0$ and $\ell$ is decreasing and slowly varying.  Then, for any $\varepsilon>0$,
\begin{equation} \label{mainest2}
\|T(t)A^{-1}\| = \O \left( \frac { (\log t)^\varepsilon}{(t k^\#(t))^{1/\alpha}} \right), \qquad t \to \infty,
\end{equation}
where $k(t) = 1/\ell(t^{1/\alpha})$ and $k^\#$ is the de Bruijn conjugate of $k$.  
\end{theorem}

\begin{proof}  Much of the proof is similar to Theorem \ref{thm.main}.  We can assume that $\|T(t)\| \ne 0$ for each $t>0$.  Given $\varepsilon \in (0,\alpha^{-2})$, let $\beta = 1 - \varepsilon \alpha^2 \in (0,1)$.  By replacing $\ell$ by an asymptotically equivalent, decreasing function, we may assume, without loss, that $g: s \mapsto s^{1-\beta} \ell(s)$ is increasing on $\R_+$.  Let $S_\ell$ be the Stieltjes function associated with $\ell$, and let
\begin{eqnarray*}
f_\ell(s) &=& S_\ell(1/s), \qquad s>0.
\end{eqnarray*}
Then $f_\ell$ is a complete Bernstein function, and
$$
W_{\alpha,\beta,\ell}(A) = A^{-(\alpha-\beta)} f_\ell(A^{-1}).
$$
By Theorem \ref{thm.bounded}, 
$$
 \sup_{z \in \mathbb C_+} \|(z+A)^{-1} W_{\alpha,\beta,\ell}(A)\|< \infty.
$$
By Theorem \ref{thm.CRbound},
$$
\|T(t) A^{-(\alpha-\beta)} f_\ell(A^{-1}) \| \le \frac {C}{t}, \qquad t>0.
$$
By Theorem \ref{interpol2}(\ref{int2b}), with $\gamma = \beta-\alpha$, 
$$
 f_\ell( \|T(t) A^{-1}\|) \le  \frac {C \|T(t) A^{-1} \|}{t \|T(2t) A^{-(\alpha-\beta+1)}\|} \,.
$$
Letting $f_{\alpha,\ell}(s) = s^{\alpha-\beta} f_\ell(s)$,  we have
\begin{equation} \label{fatest}
 f_{\alpha,\ell}( \| T(t) A^{-1}\|) \le \frac {C\|T(t) A^{-1} \|^{\alpha-\beta+1}}{t\|T(2t) A^{-(\alpha-\beta+1)}\|}.
\end{equation}

By Theorem \ref{karamata}(\ref{kar1}) with $g(\lambda) = \lambda^{1-\beta} \ell(\lambda)$ and $\sigma=\beta$, $S_\ell(s) \sim s^{-\beta}\ell(s)$ as $s \to \infty$.  Hence 
$$
f_{\alpha,\ell}(s) \sim s^\alpha \ell(1/s) = \frac {s^{\alpha}} {k(s^{-\alpha})}, \qquad  s\to0+.
$$ 
By Proposition \ref{regvarinv}(\ref{rvi3}),
$$
f_{\alpha,\ell}^{-1}(s) \sim \left(\frac {s} {k^\#(1/s)}\right)^{1/\alpha}, \qquad s\to0+.
$$
If the right-hand side of \eqref{fatest} is small, it follows that
\begin{equation*} \label{Lest2}
\|T(t)A^{-1}\| \le  \frac { C\|T(t) A^{-1} \|^{1 + (1-\beta)/\alpha}}{(t L(t) \|T(2t) A^{-(\alpha-\beta+1)}\|)^{1/\alpha}},
\end{equation*}
where
$$
L(t) = k^\# \left( t \frac {\|T(2t) A^{-(\alpha-\beta+1)} \|}{\|T(t) A^{-1}\|^{\alpha-\beta+1}} \right).
$$
Let $\psi(s) = (s k^\#(s))^{1/\alpha}$.   Since $\psi$ is regularly varying with positive index, we can choose $k^\#$ so that $\psi$ is strictly increasing and continuous (see the remarks following Definition \ref{regvar}; in fact, we can also choose $k^\#$ to be decreasing, since $k$ is increasing).  Then
\begin{eqnarray} \label{psiest2}
\frac {1}{\|T(t)A^{-1}\|} &\ge& c t^{1/\alpha} \, \frac {\|T(2t)A^{-(\alpha-\beta+1)}\|^{1/\alpha}}{\|T(t)A^{-1}\|^{1+(1-\beta)/\alpha}} \, L(t)^{1/\alpha} \\
&=& c \, \psi \left(t \frac {\|T(2t)A^{-(\alpha-\beta+1)}\|}{\|T(t)A^{-1}\|^{\alpha-\beta+1}} \right). \nonumber
\end{eqnarray}
If the right-hand side of \eqref{fatest} is bounded away from zero, then \eqref{psiest2} also holds for some $c>0$, since $\|T(t)A^{-1}\|$ is bounded and $\psi$ is bounded on bounded intervals.  Hence \eqref{psiest2} holds for all $t>0$, for some $c>0$.

Since $\psi$ is strictly increasing,
$$
t \frac {\|T(2t)A^{-(\alpha-\beta+1)}\|}{\|T(t)A^{-1}\|^{\alpha-\beta+1}} \le \psi^{-1} \left( \frac {C}{\|T(t)A^{-1}\|} \right) \,.
$$
By Proposition \ref{regvarinv}(\ref{rvi2}),
$$
\psi^{-1} (s) \sim s^\alpha k^{\#\#}(s^\alpha) \sim \frac {s^\alpha}{\ell(s)},
$$
so
\begin{equation} \label{impest2}
\| T(2t)A^{-(\alpha-\beta+1)} \| \le \frac {C \|T(t)A^{-1}\|^{1-\beta}} {t \ell \big( \|T(t)A^{-1}\|^{-1} \big)} \,,
\end{equation}
for large $t$.  By Lemma \ref{interpoldec} with $B=A^{-1}$, $\gamma=1$ and $\delta=\alpha-\beta+1$,
\begin{eqnarray*}
\|T(2Kt)A^{-1}\|^{1+(1-\beta)/\alpha} &\le& C \|T(2t)A^{-(\alpha-\beta+1)}\|^{1/\alpha}
\end{eqnarray*}
for some $K$.  Hence
$$
\|T(2Kt)A^{-1}\|  \le \frac {C}{(t \ell(\|T(t)A^{-1}\|^{-1}))^{1/\alpha}} \left( \frac {\|T(t)A^{-1}\|} {\|T(2Kt)A^{-1}\|} \right)^{\frac{1-\beta}{\alpha}}.
$$

Now we shall temporarily assume that 
\begin{equation} \label{tempass}
cM(s) \le \|(is+A)^{-1}\| \le  M(s),
\end{equation}
where $M(s) = s^\alpha/\ell(s)$ for large $s$.  We know from \eqref{genest+} and \eqref{genest-} that
$$
\frac{c}{M^{-1}(Ct)} \le \|T(t)A^{-1}\| \le \frac{C}{(M.{\log})^{-1}(ct)}.
$$
Since $\ell$ is decreasing and slowly varying and $M^{-1}$ is regularly varying,
$$
\frac{1}{ \ell(\|T(t)A^{-1}\|^{-1})} \le \frac {1}{\ell(c^{-1}M^{-1}(Ct))} \sim \frac {1}{\ell(M^{-1}(t))}.
$$
Hence
$$
\|T(2Kt)A^{-1}\| \le \frac {C}{(t \ell(M^{-1}(t)))^{1/\alpha}} \left( \frac {M^{-1}(Ct)}{(M.{\log})^{-1}(ct)} \right)^{\frac{1-\beta}{\alpha}}
$$
for large $t$.  Since $M$ is regularly varying of index $\alpha$, $M^{-1}$ and $(M.{\log})^{-1}$ are both regularly varying of index $1/\alpha$.  Since $\ell$ is decreasing, we can apply Lemma \ref{lem.logpert}(\ref{logpert2}), with $\ell$ replaced by $1/\ell$ and $\delta = \alpha$, to obtain that for $a,b>0$ there exists $c_{a,b}>0$ such that
$$
\frac { (M.{\log})^{-1}(at)}{M^{-1}(bt)}  \ge \frac{c_{a,b}}{(\log t)^{1/\alpha}}
$$
for large $t$.  Moreover,
$$
M^{-1}(t) \sim t^{1/\alpha} k^\#(t)^{1/\alpha},
$$
and  $k(t k^\#(t)) \sim (k^\#(t))^{-1}$, by definition of $k^\#$, so
$$
\ell(M^{-1}(t)) \sim \ell(t^{1/\alpha} k^\#(t)^{1/\alpha}) \sim k^\#(t).
$$
Hence
$$
\|T(2Kt)A^{-1}\| \le \frac {C (\log t)^\varepsilon}{\left(t k^\#(t)\right)^{1/\alpha}}
$$
for large $t$.  Replacing $t$ by $t/(2K)$ and changing the value of $C$ we obtain
$$
\|T(t)A^{-1}\| \le \frac {C (\log t)^\varepsilon}{\left(t k^\#(t)\right)^{1/\alpha}}
$$
for large $t$.

Now we no longer assume (\ref{tempass}), and we obtain the same result by means of an artificial device.  Let $(S(t))_{t\ge0}$ be a bounded $C_0$-semigroup on a Hilbert space $Y$, with generator $-B$ which does satisfy (\ref{tempass}).  For example, we can take a normal semigroup on a Hilbert space with
$$
\sigma (B) = \left\{  2M(s)^{-1} + is \suchthat  s \ge 2 \right\}.
$$
We then consider the $C_0$-semigroup
$$
T(t) \oplus S(t)
$$
on $X \oplus Y$.  This satisfies the assumption (\ref{tempass}), so we find that 
$$
\|T(t)A^{-1}\| \le \big\|T(t)A^{-1} \oplus S(t)B^{-1} \big\| \le \frac {C (\log t)^\varepsilon}{\left( t k^\#(t) \right)^{1/\alpha}} \,.
$$
\end{proof}

\begin{remark}
A variation of the proof above proceeds from \eqref{impest} by the method of Remark \ref{altstart}.  This uses the estimate \eqref{klogest} in the form
$$
\|T(t) A^{-1}\|^{-1} \ge c t^{1/\alpha} (k.\log)^{\#}(t)^{1/\alpha}.
$$
Since the function $s \mapsto s^{1-\beta} \ell(s)$ is asymptotically equivalent to an increasing function,
$$
\frac{\ell(\|T(t)A^{-1}\|^{-1})}{\|T(t)A^{-1}\|^{1-\beta}} \ge c t^{(1-\beta)/\alpha} (k.\log)^\#(t)^{(1-\beta)/\alpha} \ell \left(t^{1/\alpha} (k.\log)^{\#}(t)^{1/\alpha} \right).
$$
By \eqref{impest2},
$$
\| T(2t)A^{-(\alpha-\beta+1)} \| \le \frac{C} {t^{(\alpha-\beta+1)/\alpha} (k.\log)^\#(t)^{(1-\beta)/\alpha} \ell \left(t^{1/\alpha} (k.\log)^{\#}(t)^{1/\alpha} \right)} \,.
$$
Applying Lemma \ref{interpoldec} with $B = A^{-1}$, $\gamma=\alpha-\beta+1$ and taking $\beta$ arbitrarily close to $1$, it follows that
\begin{eqnarray*}
\|T(t)A^{-1}\|    &\le& \frac {C_\varepsilon}{\left[t (k.\log)^\#(t)^\varepsilon \ell \left(t^{1/\alpha}(k.\log)^{\#}(t)^{1/\alpha} \right) ^{1+\varepsilon}\right]^{1/\alpha}} \,.
\end{eqnarray*}
We do not expect that this gives better estimates than the simpler \eqref{mainest2}.
\end{remark}

\section{Singularity at zero: general results} \label{s.zero1}

\subsection{Preliminary background} \label{ss.background}

In this section  we treat the rates of decay of the derivatives $-T(t)Ax$ for  $x \in \dom(A)$, where $(T(t))_{t \ge 0}$ is a bounded $C_0$-semigroup on a Banach space or a Hilbert space $X$, with generator $-A$.  In other words, we study the orbits of $(T(t))_{t \ge 0}$ starting at points $y=Ax$ in the range of $A$.  We shall see that such decay corresponds  to properties of the resolvent of $A$ permitting a singularity at $0$ but requiring boundedness at infinity on the imaginary axis.

If the decay of such orbits is uniform for $\|x\|\le1$, the semigroup must be eventually differentiable, i.e., for sufficiently large $t$, $T(t)$ maps $X$ into $\dom(A)$ and $AT(t) \in \mathcal{L}(X)$.  For an eventually differentiable, bounded $C_0$-semigroup on a Banach space, Arendt and Pr\"uss \cite[Theorem 3.10]{AP92} (see \cite[Theorem 4.4.16]{ABHN01}) showed that
\begin{equation} \label{uniform1}
\lim_{t\to\infty} \|AT(t)\| = 0
\end{equation}
if and only if
\begin{equation}
\sigma(A) \cap i\mathbb{R} \subset \{0\}. \label{specsub0}
\end{equation}

In fact the Arendt-Pr\"uss theorem follows from (the corollary of) the Katznelson-Tzafriri theorem for $C_0$-semigroups which we recall below (see also \cite{Vu92}, \cite{Es92}, \cite{Bat90} and the survey \cite{ChTo07}).  Let $E$ be a closed subset of $\R$.  A function  $f \in L^1(\R)$ is said to be of {\it spectral synthesis with respect to $E$} if there is a sequence $(f_n)$ in $L^1(\R)$ such that $\lim_{n\to\infty} \|f- f_n\|_{L^1} = 0$ and, for each $n$, $\cF f_n$ vanishes in a neighbourhood of $E$.

The closed subset $E\subset \R$ is said to be of {\it spectral synthesis} if every function $f\in L^1 (\R )$ whose Fourier transform $\cF f$ vanishes on $E$ is of spectral synthesis with respect to $E$.  Any countable closed subset of $\R$ is of spectral synthesis \cite[p.\,230]{Ka68}. 

\begin{theorem} \label{KT}
Let $(T(t))_{t\ge0}$ be a bounded $C_0$-semigroup on a Banach space $X$, with generator $-A$.  Let $f \in L^1(\R_+)$ be of spectral synthesis with respect to $E:= i\sigma(-A) \cap \R$.  Then
\begin{equation} \label{KT1}
\lim_{t\to\infty} \|T(t) \hat f(T)\| = 0,
\end{equation}
where 
$$
\hat f(T)x = \int_0^\infty f(t) T(t)x \, \ud{t}, \qquad x \in X.
$$
Conversely, if $f \in L^1(\R_+)$ and \eqref{KT1} holds, then $\cF f = 0$ on $E$.
\end{theorem}

\begin{corollary} \label{convtozero}  Let  $(T(t))_{t \ge0}$ be a bounded $C_0$-semigroup on a Banach space $X$ with generator $-A$.  The following statements are equivalent.
\begin{enumerate} [\rm(i)]
\item \label{KTcori} $\lim_{t \to \infty} \|T(t)A (I+A)^{-2}\|=0$;
\item \label{KTcorii}  $\sigma (A)\cap i\mathbb R \subset \{0\}$.
\end{enumerate}
\end{corollary}

\begin{proof}
This follows from applying the Katznelson-Tzafriri theorem, with
$$
f(t) = e^{-t} - te^{-t}, \qquad t\ge0.
$$
Then
$$
\cF f(s) = is(1+is)^{-2}, \quad s \in \R; \qquad
\hat f(T) = A(I+A)^{-2}.
$$
Since $\cF f(0) = 0$  and $\{0\}$ is a set of spectral synthesis, $f$ is of spectral synthesis with respect to $\{0\}$.
\end{proof}

\begin{remark}
If $(T(t))_{t\ge0}$ is eventually differentiable, then $(1+A)^2 T(\tau)$ is a bounded operator for some $\tau>0$.  Moreover, 
$$
AT(t+\tau) = (I+A)^2 T(\tau) T(t) A(I+A)^{-2},
$$
so Corollary \ref{convtozero}\eqref{KTcori} is equivalent to \eqref{uniform1} in this case.  Thus the Arendt-Pr\"uss theorem follows from the Katznelson-Tzafriri theorem.
\end{remark}

The range of $A(I+A)^{-2}$ is $\ran(A) \cap \dom(A)$ (Proposition \ref{prop.aab}(\ref{domim})), and Corollary \ref{convtozero} describes decay of orbits on that space in an appropriate uniform sense.  We defer consideration of the rate of convergence in Corollary \ref{convtozero} (\ref{KTcori}) until Section \ref{sectwosing}.  In this and the next section, we consider instead the question whether the decay of orbits is uniform with respect to the graph norm on $\dom(A)$, i.e., whether
\begin{equation*}
\lim_{t \to \infty} \sup \left\{\|T(t)Ax\|: x \in \dom(A), \|x\|_{\dom(A)} = 1 \right\}=0.
\end{equation*}
This can be reformulated as
\begin{equation}  \label{uniform0}
\lim_{t \to \infty} \|T(t)A (\omega+A)^{-1}\|=0,
\end{equation}
for any $\omega \in \rho(-A)$.    This property is intermediate between \eqref{uniform1} and the statements of Corollary \ref{convtozero}, and it is independent of the choice of $\omega$.  When $A$ is invertible, \eqref{uniform0} is equivalent to the much studied notion of exponential stability, i.e.,  $\lim_{t\to\infty} \|T(t)\|=0$, for which the rate of decay is always at least exponential (see \cite[Chapter 5]{ABHN01}).  Thus we are interested only in cases when $0 \in \sigma(A)$.

Our goal in this section and Section \ref{s.zero2} is to develop a framework for decay of the form \eqref{uniform0} in terms of the spectrum and resolvent of the generator in similar form to the one from Section \ref{s.infinity}.    However, the decay given by \eqref{uniform0} is much less studied in the literature. Thus we shall prove several statements clarifying the limitations imposed by \eqref{uniform0} and  the consequences following from the decay of orbits in this sense.

In the later subsections, we shall assume that the semigroup is bounded, but 
first it is instructive to make a few remarks relating to that assumption.  We begin by observing that \eqref{uniform0} implies (\ref{specsub0}), without the assumption of boundedness. 
 
\begin{proposition} \label{spec0}
Let $(T(t))_{t\ge0}$ be a $C_0$-semigroup on a Banach space $X$, with generator $-A$.  Assume that
$$
\lim_{t\to\infty} \|T(t)A(\omega+A)^{-1}\|=0
$$
 for some $\omega\in\rho(-A)$.  Then, for each $\eta>0$,
\begin{equation*} \label{nospecinf}
\inf \{ \Re\lambda : \lambda \in \sigma(A), |\lambda| > \eta \} > 0.
\end{equation*}
In particular,
$$
\sigma(A) \cap \{\lambda \in \mathbb{C} : \Re\lambda \le 0\} \subset \{0\}.
$$
\end{proposition}

\begin{proof}  We may assume that $\omega>0$ and $\|T(t)\| \le Me^{(\omega-1)t}$ for all $t>0$.  Define $f_t \in L^1(\R_+) \subset M^b(\R_+)$ and $\mu_t \in M^b(\R_+)$ by
\begin{eqnarray*}
f_t(s) &=& \begin{cases}  \omega e^{-\omega(s-t)}, &s\ge t, \\ 0 &0\le s<t, \end{cases} \\
\mu_t &=&  \delta_t-f_t.
 \end{eqnarray*}
Here $\delta_t$ is the unit mass at $t$ and $f_t \in L^1(\R_+)$ is regarded as an absolutely continuous measure.  Then
\begin{equation*}
T(t)A (\omega + A)^{-1}=\int_{0}^{\infty} T(s)  \, \ud\mu_t (s).
\end{equation*}
By the spectral mapping (inclusion) theorem for the Hille-Phillips functional calculus \cite[Theorem 16.3.5]{HilPhi}, 
\begin{equation} \label{specinc}
\big \{ e^{-\lambda t}\lambda(\omega+\lambda)^{-1}: \lambda \in \sigma (A) \big\} \subset \sigma \left(T(t)A(\omega+A)^{-1} \right),
\end{equation}
and then
\begin{eqnarray*}
 \|T(t)A (\omega+A)^{-1}\| \ge 
\sup \big\{e^{-{\Re\lambda t}} |\lambda| |\omega+\lambda|^{-1}: \lambda \in \sigma (A)\big\}.
\end{eqnarray*}
If $\lambda \in \sigma(A)$ and $|\lambda| \ge \eta > 0$, then $|\lambda||\omega+\lambda|^{-1}\ge \eta(\omega+\eta)^{-1}$.  Taking $t$ such that $\|T(t)A(\omega+A)^{-1}\| \le \eta/(e(\omega+\eta))$, it follows that $e^{-\Re\lambda t} \le 1/e$ and $\Re\lambda \ge 1/t$ for all such $\lambda$.  
\end{proof}

We shall see in Theorem \ref{resbound0} that $\|(is+A)^{-1}\|$ is bounded for $|s| \ge \eta > 0$ if the assumptions of Proposition \ref{spec0} are satisfied and $(T(t))_{t\ge0}$ is bounded.

The following simple example shows that \eqref{uniform0} does not imply that the semigroup is bounded.

 \begin{example}
 Let $X=L^2([0,1])$. Define a $C_0$-semigroup $(T(t))_{t \ge 0}$ on $X$ by
 $$
 (T(t)f)(s)=e^{-(s^{3/2}+is)t}f(s), \qquad t\ge0, \, \, f \in L^2([0,1]), \,\, s\in[0,1].
 $$
 The semigroup $(T(t))_{t \ge 0}$ has a bounded generator $-A$ given by $(-A f)(s)=-(s^{3/2}+is)f(s)$ for $s \in [0,1]$, and it is easy to check that
 $$
 \|AT(t)\|=\O(t^{-2/3}), \qquad t \to \infty.
 $$
 Consider then the semigroup  $(\mathcal T(t))_{t \ge 0}$,  with generator ${-\mathcal A}$, on the space $X\oplus X$ given by the operator matrix
  $$
\mathcal T(t) = \begin{pmatrix} T(t) & -tAT(t) \\ 0 & T(t) \end{pmatrix}, \qquad t \ge 0.
$$
Then a simple calculation shows that there exist $c,C >0$ such that
 $$
 \|\mathcal T(t)\|\ge c t^{1/3} \qquad \text{and} \qquad \|\mathcal A\mathcal T(t)\|\le C t^{-1/3}, \quad t > 0.
 $$
 \end{example}

For semigroups of normal operators on Hilbert space, \eqref{uniform0} does imply boundedness.  We omit the easy proof of the following proposition.

\begin{proposition} \label{normal0}
Let $(T(t))_{t\ge0}$ be a $C_0$-semigroup of normal operators on a Hilbert space $X$, with generator $-A$.   The following are equivalent:
\begin{enumerate}[\rm(i)]
\item \label{sing1} For some/all $\omega \in \rho(-A)$, $\lim_{t\to\infty} \|T(t)A(\omega+A)^{-1}\| = 0$.
\item \label{sing2} $\sigma(A) \subset \C_+ \cup\{0\}$ and $\inf \{ \Re\lambda : \lambda \in \sigma(A), |\lambda| > \eta \} > 0$ for some/all $\eta>0$.
\item \label{sing0}$(T(t))_{t\ge0}$ is bounded, $\sigma(A) \cap i\R \subset \{0\}$ and $\sup_{|s|\ge \eta} \|(is+A)^{-1}\| < \infty$ for some/all $\eta>0$.
\end{enumerate}
\end{proposition}

Proposition \ref{normal0} shows that there are  bounded semigroups of normal operators on Hilbert spaces, with $\sigma(A) \cap i\R \subset \{0\}$, for which property  (\ref{sing1}) of Proposition \ref{normal0} does not hold.  The next example shows that there are bounded semigroups on Hilbert space which satisfy property (\ref{sing2}), but not property (\ref{sing1}).

\begin{example} \label{counterex0}
Let $X_0$ and $X_1$ be Hilbert spaces and consider the Hilbert space $X=X_0\oplus X_1$.  Let $(T(t))_{t \ge 0}$ be a bounded $C_0$-semigroup on $X$ given by $T(t)=T_0(t)\oplus I$, where $(T_0(t))_{t\ge0}$ is a bounded, but not exponentially stable, $C_0$-semigroup on $X_0$ with generator $-A_0$ such that $\sigma(A_0) \subset \C_+$.  If $-A$ is the generator of $(T(t))_{t \ge 0}$ then $\sigma (A)\cap i\mathbb R=\{0\}$.  Moreover, (\ref{uniform0}) does not hold; if it did, since $A_0$ is invertible it would follow that $\|T_0(t)\|\to 0$, a contradiction.

We may take $(T_0(t))_{t\ge0}$ to be a bounded semigroup on a Hilbert space, which is not exponentially stable, with $\inf \{\Re\lambda:  \lambda \in \sigma(A)\} > 0$  (see \cite[Example 5.1.10]{ABHN01}).  Then $(T(t))_{t\ge0}$ satisfies \eqref{sing2} of Proposition \ref{normal0}, but \eqref{sing1} does not hold.
\end{example}

\subsection{Rates of decay} \label{ss.decay0}

For the rest of Section \ref{s.zero1} and throughout Section \ref{s.zero2}, we shall consider bounded semigroups.  For simplicity of notation, we shall set $\omega = 1$ in \eqref{uniform0}.   

Our main objective is to deduce (\ref{uniform0}) as a consequence of suitable spectral assumptions.  We shall need to assume strong conditions which are consistent with Proposition \ref{normal0} and which also exclude examples such as Example \ref{counterex0}.  In fact, we shall show that the property (\ref{sing0}) of Proposition \ref{normal0} implies \eqref{uniform0} for all semigroups on Hilbert space, and not only for normal semigroups.  This will be deduced from Theorem \ref{KT+} which is a version of the Katznelson-Tzafriri Theorem for semigroups on Hilbert spaces for certain measures which are not absolutely continuous.  This extension is of independent interest, and it will be the subject of Subsection \ref{s.KT+}.  In this subsection we shall consider other aspects of the rates of decay in \eqref{uniform0}.

\begin{remark}
For bounded semigroups on Banach spaces the property (\ref{specsub0}) implies that $\lim_{t\to\infty} T(t)A(I+A)^{-1} = 0$ in the strong operator topology. This was explicitly shown in \cite[Example, p.802]{Bat90} using a Tauberian theorem for vector-valued Laplace-Stieltjes transforms.  It follows from Corollary \ref{convtozero} using the uniform boundedness of $T(t)A(I+A)^{-1}$ and the density of the range of $(I+A)^{-1}$.
\end{remark}

We consider a bounded $C_0$-semigroup $(T(t))_{t\ge0}$, with generator $-A$, on a Banach space or Hilbert space $X$.  We assume that $\sigma(A) \cap i\R = \{0\}$ and that, for some/all $\eta>0$, $\sup_{|s|\ge \eta} \|(is+A)^{-1}\| < \infty$.   We aim to exhibit the possible rates of decay of $\|T(t)A(I+A)^{-1}\|$ in terms of the growth of $\|(is+A)^{-1}\|$ as $|s| \to 0$.  For this purpose, let $m$ and $N$ be decreasing functions on $(0,\infty)$ such that
\begin{eqnarray}
\label{mbound} \|(is+A)^{-1}\| &\le& m(|s|), \qquad s \ne 0, \\
\label{Nbound} \|T(t)A(I+A)^{-1}\| &\le& N(t), \qquad t>0.
\end{eqnarray}
The smallest possible functions are given by
\begin{eqnarray}
m(s) &=&\sup \{\| (ir+A)^{-1}\|: |r|\ge s\}, \quad s>0, \label{defm} \\
N(t)&=&\sup \{\|T(\tau)A(I+A)^{-1}\| : \tau\ge t\}, \quad t>0. \label{defN}
\end{eqnarray}
This function $m$ is continuous, and we shall always assume continuity of $m$ so that $m$ has a right inverse $m^{-1}$ defined on an interval of the form $[a,\infty)$.

The function $N$ defined by \eqref{defN} may not be continuous, but it is lower semicontinuous and right-continuous.  Assuming that $\lim_{t\to\infty} N(t) = 0$, we define
\begin{equation} \label{defN-1}
N^*(s) = \min \{ t\ge0 : N(t) \le s \}, \qquad s>0,
\end{equation}
so that $N(N^*(s)) \le s$ for all $s>0$, and $N(N^*(s))=s$ if $s$ is in the range of $N$.

Since we assume that $0 \in \sigma (A)$, elementary theory of resolvents implies that $m(s)\ge s^{-1}$ for all $s >0$.  The hypothesis that there might be a corresponding lower bound for $N$ is too naive, because examples of the type considered in Example \ref{counterex0} show that $N$ can decay arbitrarily slowly even for semigroups of normal operators on Hilbert space.  On the other hand the next result shows that examples where $N$ decays faster than $t^{-1}$ must be of the form considered in Example \ref{counterex0}.

\begin{theorem} \label{split}
Let $(T(t))_{t \ge 0}$ be a $C_0$-semigroup on a Banach space $X$ with generator $-A$. If $0 \in \sigma(A)$, then at least one of the following two properties holds:
\begin{enumerate}[\rm(i)]
\item \label{spliti} $\limsup_{t \to \infty} t \|T(t)A(I+A)^{-1}\| > 0$;
\item \label{splitii} There are closed $T$-invariant subspaces $X_0,X_1$ of $X$ such that
\begin{enumerate}[\rm(a)]
\item $X = X_0 \oplus X_1$,
\item $T(t)x = x$ for all $t\ge0$, $x\in X_1$, and
\item the generator $-A_0$ of the restriction of $T$ to $X_0$ is invertible.
\end{enumerate}
\end{enumerate}
\end{theorem}

\begin{proof}  Assume first that $0$ is a limit point of $\sigma (A)$.
Then there exists $\{\lambda_n: n \ge 1\}\subset \sigma (A) \setminus \{0\}$ such that $\lambda_n \to 0$ as $n \to \infty$.  Setting $t_n = (\Re\lambda_n)^{-1}$, and using \eqref{specinc}, we infer that
$$
\limsup_{t \to \infty} t \|T(t)A(I+A)^{-1}\|\ge \lim_{n \to \infty} \left (\frac{t_n e^{-\Re \lambda_n t_n} \Re\lambda_n} {|1+\lambda_n|}\right)  = \frac{1}{e} \,.
$$
So (\ref{spliti}) holds.

Now assume that $0$ is not a limit point of $\sigma (A)$.  Then $X$ can be decomposed into the direct sum of $T$-invariant subspaces $X=X_0\oplus X_1$ such that $A_1:=A\!\!\!\upharpoonright_{X_1} \in {\mathcal L}(X_1)$,
$\sigma(A_1)=\{0\},$ and for $A_0:=A\!\!\!\upharpoonright_{X_0}$ one has $\sigma (A_0)=\sigma(A)\setminus \{0\}$.  If (i) is false, then
\begin{eqnarray*}
\liminf_{t \to \infty} t\|A_1e^{-tA_1}\| &\le& \limsup_{t\to\infty} t\|A_1e^{-tA_1}(I+A_1)^{-1}\| \, \|I+A_1\|  \\
&\le& \limsup_{t \to \infty} t \|T(t)A(I+A)^{-1}\| \, \|I+A_1\| = 0.
\end{eqnarray*}
By \cite[Theorem 2.1]{KMOT04}, $A_1=0$.  Thus (\ref{splitii}) holds.
\end{proof}

Thus the rates are of interest only in the case when $N(t)$ decreases no faster than $t^{-1}$ as $t\to\infty$, and $m(s)$ increases at least as fast as $s^{-1}$ as $s \to 0+$.  In that case the bound in \eqref{resest0} below  gives $m(s) = \O( N^*(cs))$  for small $s>0$.  

\begin{theorem} \label{resbound0}
Let $(T(t))_{t\ge0}$ be a bounded $C_0$-semigroup on a Banach space $X$, with generator $-A$.  Assume that
\begin{equation} \label{resbd0ass}
\lim_{t\to\infty} \|T(t)A(I+A)^{-1}\| = 0.
\end{equation}
Let $N$ be a decreasing function such that \eqref{Nbound} holds and $\lim_{t\to\infty} N(t) = 0$, and let $N^*$ denote any function such that $N(N^*(s)) \le s$ for all $s \in (0,1)$.  Then $\sigma(A) \cap i\mathbb{R} \subset \{0\}$ and, for any $c \in (0,1)$,  
\begin{equation}\label{resest0}
\|(is+A)^{-1}\| =  \begin{cases} \O \left( N^*(c|s|) + |s|^{-1} \right),  \qquad &s\to0,   \\ \O(1), &|s|\to\infty. \end{cases}
\end{equation}
\end{theorem}

\begin{proof}  Proposition \ref{spec0} shows that $\sigma(A) \cap i\mathbb{R} \subset \{0\}$.  Alternatively, the arguments which follow show that $A$ has no approximate eigenvalues in $i\R \setminus \{0\}$.  Since $-A$ generates a bounded semigroup, $\sigma(A) \subset \overline{\C}_+$ and $\sigma(A) \cap i\mathbb{R}$ consists only of approximate eigenvalues.

Let $K = \sup_{t\ge0} \|T(t)\|$.  Let $s \in \mathbb{R} \setminus \{0\}$, $t>0$ and $x \in \dom(A)$.  We use the formula
\begin{eqnarray*}
ise^{ist} x
 &=&  is e^{ist} \int_0^t e^{-is\tau} T(\tau) (is +A)x \, \ud\tau +is T(t)x \\
 &=& is e^{ist} \int_0^t e^{-is\tau} T(\tau) (is +A)x \, \ud\tau  + T(t)(I+A)^{-1}(is+A)x \\
 && \phantom{X}  - (1-is)T(t)A(I+A)^{-1} x.
\end{eqnarray*}
Since $\|T(t) (I+A)^{-1} \| \le K$, this gives
$$
|s| \, \|x\|  \le K (|s| t + 1)\|(is+A)x\|  + |1-is| N(t) \|x\|.
$$
Hence
\begin{equation} \label{Nest}
(|s| - |1-is|N(t))\|x\| \le K(|s|t+1) \|(is+A)x\|.
\end{equation}

Set $t = N^*(c|s|)$.  For $|s|$ sufficiently small,
$$
|s| - |1-is|N(t) \ge |s|(1-|1-is|c) >0.
$$
For any $K'>K$,  \eqref{Nest} gives
\begin{equation} \label{Nest2}
\|(is+A)^{-1}\| \le \frac{K \left(|s| N^*(c|s|) + 1\right)}{|s|(1-|1-is|c)}\le \frac{K'}{1-c} \left(N^*(c|s|) + |s|^{-1} \right),
\end{equation}
for $|s|$ sufficiently small.

Since $\lim_{t\to\infty} N(t) =0$, we may set $t=\tau$ with $N(\tau) < 1$.
For $|s|$ sufficiently large,
$$
|s| - |1-is|N(\tau) > 1,
$$
so \eqref{Nest} gives
\begin{equation*}
\|(is+A)^{-1}\| \le \frac{K (|s|\tau+1)}{|s|-|1-is|N(\tau)} = \O(1), \qquad |s|\to\infty.
\qedhere
\end{equation*}
\end{proof}

Theorem \ref{resbound0} is analogous to \cite[Proposition 1.3]{BaDu08}.  It allows one to get lower bounds for the decay of $\|T(t)A(I+A)^{-1}\|$ in terms of the growth of the resolvent near the origin, analogous to \eqref{genest-}.  

\begin{corollary} \label{lowbound0}
Let $(T(t))_{t\ge0}$ be a bounded $C_0$-semigroup on a Banach space $X$, with generator $-A$.  Assume that $\sigma(A) \cap i\mathbb{R} = \{0\}$ and that 
\begin{equation}\label{addasump}
\lim_{s\to0} \max \left(\|s(is+A)^{-1}\|,\|s(-is+A)^{-1}\| \right)=\infty.  
\end{equation}
Let $m$ be the function defined by \eqref{defm} and $m^{-1}$ be any right inverse for $m$.  Then there exist $c>0$ and $c'>0$ such that
$$
\|T(t)A(I+A)^{-1}\| \ge c m^{-1}(c't)
$$
for all sufficiently large $t$.
\end{corollary}

\begin{proof}  We may assume that (\ref{resbd0ass}) holds.  Let $N$ and $N^*$ be defined by \eqref{defN} and \eqref{defN-1} respectively. By Theorem \ref{resbound0},
$$
m(s) \le C \left(N^*(cs) + \frac 1s \right), \qquad 0 < s \le 1.
$$
Rearranging this and using the assumption \eqref{addasump},
$$
N^*(cs) > \frac {m(s)}{2C}
$$
for all sufficiently small $s>0$. For $t$ sufficiently large, put  $s = c^{-1} N(t)$.  Then $N^*(cs) \le t$.  Hence
$$
m(m^{-1}(2Ct)) = 2Ct \ge 2CN^*(cs) > m(s).
$$
Since $m$ is decreasing,
$$
m^{-1}(2Ct) < s = c^{-1}N(t) \le c^{-1}K \|T(t)A(I+A)^{-1}\|,
$$
for all sufficiently large $t$.
\end{proof}

\begin{remark} \label{remlb0} The assumption \eqref{addasump} in Corollary \ref{lowbound0} cannot be omitted, because of the trivial semigroup where $A=0$, for example.  For a $C_0$-semigroup of contractions, it can be weakened to the assumption that
$$
\liminf_{s\to0} \max \left(\|s(is+A)^{-1}\|,\|s(-is+A)^{-1}\|\right) > 1.
$$
This follows from  \eqref{Nest2}.
\end{remark}

\smallskip

The ideal counterpart to  Corollary \ref{lowbound0} would be to show that if  \eqref{mbound} holds, then 
\begin{equation} \label{genest0}
 \|T(t)A(I+A)^{-1}\| = \O\left(  m^{-1}(ct) \right), \qquad t \to\infty,
\end{equation}
for some $c>0$.  However for some $m$ it is not possible to get this sharper estimate, even for semigroups of normal operators on Hilbert space, as the following proposition shows. The proposition gives the necessary (when $m$ is defined by \eqref{defm}) and sufficient condition on $m$ for \eqref{genest0} to be true for semigroups of normal operators on Hilbert spaces.  It is presented in the more general context of quasi-multiplication semigroups introduced in Section \ref{s.infinity} and it should be compared to  Proposition \ref{normalinf}.  The conditions (\ref{normalres}) for $M$ and (\ref{normalres0}) for $m$ are clearly dual to each other:  $M$ satisfies (\ref{normalres}) if and only if $m(s) := M(1/s)$ satisfies (\ref{normalres0}). Later we shall consider more general semigroups on Hilbert space, and we shall establish \eqref{genest0} when $m(s) = s^{-\alpha}$ for some $\alpha \ge1$, and weaker estimates for more general $m$, see Theorems \ref{polydec00}, \ref{thm.main0} and \ref{mlog}. 

\begin{proposition} \label{normalinf0}
Let $(T(t))_{t\ge0}$ be a quasi-multiplication semigroup on a Banach space $X$ with generator $-A$.  Assume that 
$$
0 \in \sigma(A) \subset \C_+ \cup \{0\} \qquad \text{and} \qquad \|(is+A)^{-1}\| = \O(1), \quad |s|\to \infty.
$$
Let $c>0$, $m$ be defined by \eqref{defm}, and $m^{-1}$ be any right inverse for $m$.  Then the following are equivalent:
\begin{enumerate}[\rm(i)]
\item  \label{normalcond01} There exists $C$ such that
\begin{equation} \label{normalest01}
\|T(t)A(I+A)^{-1}\| \le  Cm^{-1}(ct), \qquad t\ge c^{-1} m(1);
\end{equation}
\item  \label{normalcond02} There exists $B$ such that
\begin{equation} \label{normalres0}
\frac{m(\tau)}{m(s)} \ge c \log \left( \frac {s}{\tau} \right) - B, \qquad 0<\tau,s\le1.
\end{equation}
\end{enumerate}
\end{proposition}

\begin{proof}  Note first that the assumptions imply that $\|T(t)\|=1$ for all $t\ge0$, 
$$
{m(s)}^{-1} = \min \left\{|\mu + ir| : \mu \in \sigma(A), |r| \ge s\right\} \le s, \qquad s>0,
$$
and there exists $\varepsilon>0$ such that $\Re\mu \ge \varepsilon$ for all $\mu \in \sigma(A)$ with $|\mu|\ge 1$.  Moreover (\ref{normalest01}) is equivalent to
$$
{e^{-t\alpha}}\left| \frac {\mu}{1+\mu} \right| \le  C m^{-1}(ct), \qquad \mu =  \alpha + i\beta \in \sigma(A) \setminus \{0\}, \; t\ge c^{-1}m(1),
$$
and hence to
\begin{equation} \label{normalest02}
t\alpha \ge \log \left( \frac {1}{C m^{-1}(ct)} \left|\frac {\mu}{1+\mu} \right| \right)
\end{equation}
for all such $\mu$ and $t$.

Assume that (\ref{normalcond01}) holds.  Let $t\ge c^{-1}m(1)$, and put $\tau = m^{-1}(ct)$.  From (\ref{normalest02}),
$$
m(\tau) \ge \frac{c}{\alpha} \log \left( \frac {|\mu|} {C \tau |1+\mu|} \right)  \ge \frac{c}{\alpha} \log \left( \frac {|\mu|} {2C \tau} \right),
$$
if $\mu = \alpha+i\beta \in \sigma(A)$ and $|\mu|\le 1$.

Given $s>0$ with $s < \varepsilon$, take $\mu = \alpha + i\beta \in \sigma(A)$ such that ${m(s)}^{-1} = |\mu + ir| $ for some $|r| \ge s$.  Then $\alpha \le |\mu+ir| \le s < \varepsilon$, so $|\mu| \le 1$. If $|\mu| \ge s/2$, then  
$$
\frac {m(\tau)}{m(s)} \ge \frac {c|\alpha + i(\beta +r)|}{\alpha} \log \left( \frac {s}{4C\tau} \right) \ge c \log \left(\frac {s}{\tau} \right) - c \log  (4C).
$$
If $|\mu| < s/2$, then ${m(s)}^{-1} \ge |r|-|\mu| \ge s/2$.  Since $m(\tau) \ge 1/\tau$,
$$
\frac{m(\tau)}{m(s)} \ge \frac{s}{2\tau} \ge c \log \left( \frac{s}{\tau} \right) - B,
$$
for some $B\ge c \log(4C)$.  

It follows that (\ref{normalres0}) holds whenever $0 \le s < \varepsilon$ and $\tau$ is in the range of $m^{-1}$.  For other values of $\tau$ one can apply the above with $\tau$ replaced by $m^{-1}(m(\tau)  - n^{-1})$, and let $n\to \infty$.  If $\varepsilon\le s\le1$, one can use
$$
\frac{m(\tau)}{m(s)} \ge \frac{m(\tau)}{m(\varepsilon)} \ge c \log \left( \frac{\varepsilon}{\tau} \right) - B \ge  c \log \left( \frac{s}{\tau} \right) - B + c \log \varepsilon.
$$

Now assume that (\ref{normalcond02}) holds.  Given $t \ge c^{-1} m(1)$ and $\mu= \alpha+ i\beta \in \sigma(A)$ with $|\beta|\le1$, take
$$
\tau = m^{-1}(ct), \qquad s = |\beta|.
$$
By \eqref{normalres0},
$$
\frac {ct}{m(|\beta|)} \ge c \log \left( \frac {|\beta|}{m^{-1}(ct)} \right) - B.
$$
Rearranging this, using $\alpha m(|\beta|) \ge 1$, and putting $C= 2\exp(B/c)$ gives (\ref{normalest02}), provided that $|\mu|/2 \le |\beta|\le1$.  If $|\beta| < |\mu|/2$, then $\alpha > (\sqrt3/2) |\mu|$, and
$$
t\alpha = \frac{\alpha}{c} m(m^{-1}(ct)) > \frac {\sqrt3 \, |\mu|}{2c m^{-1}(ct)} \ge \log \left( \frac {\sqrt3 \, |\mu|}{2c m^{-1}(ct)} \right).
$$
So we establish (\ref{normalest02}) in this case also with $C = 2c/\sqrt3$.  Taking the maximum of two values of $C$ we have established (\ref{normalest02}) whenever $|\beta| \le 1$.

Now consider $\mu = \alpha + i\beta \in \sigma(A)$ with $|\mu| \ge 1$, so $\alpha \ge \varepsilon$.  Proceeding in a similar way to the last part of the proof of Proposition \ref{normalinf} (or applying the same arguments to the function $M(s) := m(1/s)$), one sees that there exists $C$ such that
$$
\varepsilon t \ge \log  \left( \frac {1}{Cm^{-1}(ct)}  \right) \ge \log  \left( \frac {1}{Cm^{-1}(ct)} \left| \frac {\mu}{1+\mu} \right| \right)
$$
whenever $t \ge c^{-1}m(1)$.  Then (\ref{normalest02}) holds for $\mu = \alpha + i\beta \in \sigma(A)$ with $|\beta| \ge 1$ and $t \ge c^{-1}m(1)$.  Hence (i) holds.
\end{proof}

\subsection{An extension of the Katznelson-Tzafriri Theorem} \label{s.KT+}

The following result is an extension of the Katznelson-Tzafriri Theorem \ref{KT} in the case of semigroups on Hilbert spaces to some measures which are not absolutely continuous.

\begin{theorem} \label{KT+}
Let $-A$ be the generator of a bounded $C_0$-semigroup $(T(t))_{t\geq 0}$ on a Hilbert space $X$. Assume that $E := i\sigma (-A) \cap \R$ is compact and of spectral synthesis. Assume in addition that, for some $\eta\geq 0$,
\begin{equation} \label{nosinginf}
\sup_{|s|\geq \eta} \| (is+A)^{-1}\| < +\infty.
\end{equation}
Let $\mu\in M^b (\R_+ )$ be such that $\cF \mu$ vanishes on $E$. Then
\[
 \lim_{t\to\infty} \| T(t) \hat \mu (T) \| = 0 ,
\]
where
\[
\hat \mu (T) x = \int_0^\infty T(s)x \, \ud\mu (s), \qquad x\in X  .
\]
\end{theorem}

\begin{proof}
Let $\varphi\in \cS (\R )$ be such that $\cF \varphi$ has compact support and such that $\cF \varphi = 1$ on a neighbourhood of $E$. We decompose the measure $\mu$ as follows
\[
 \mu = \mu * \varphi + \mu * (\delta_0 - \varphi ) =: \mu_0 + \mu_1 ,
\]
and we note that the measure $\mu_0$ is absolutely continuous with respect to Lebesgue measure: let $f\in L^1 (\R )$ be the density function for $\mu_0$.

Consider the function
\[
 F (t) := \int_\R T(t+s) \, \ud\mu (s),  \qquad t\in\R ,
\]
where we have extended the semigroup $T$ by $0$ on $(-\infty ,0)$ and the integral is convergent in the strong operator topology.  Since $\mu$ is supported in $\R_+$, we have
\[
 F(t) = T(t)\hat \mu(T), \qquad t\geq 0 .
\]
Following the decomposition of $\mu$, we define also the functions
\[
 F_0 (t) := \int_\R T(t+s) f(s) \, \ud{s},  \qquad t\in\R ,
\]
and
\[
 F_1 (t) := \int_\R T(t+s) \, \ud\mu_1 (s), \qquad t\in \R  ,
\]
so that $F = F_0 + F_1$. In these two definitions, we have also extended the semigroup $T$ by $0$ on $(-\infty ,0)$.

Note that
\[
 \cF \mu_0 (\xi ) = \cF f (\xi ) = \cF \mu (\xi ) \, \cF \varphi (\xi ), \qquad \xi\in\R .
\]
Since $\cF\mu$ vanishes on $E$, the Fourier transform $\cF f$ also vanishes on $E$. Since $E$ is of spectral synthesis, there exists a sequence $(f_n)_{n\ge2} \subset L^1 (\R )$ such that each Fourier transform $\cF f_n$ vanishes on a neighbourhood of $E$ and such that $\lim_{n\to\infty} \| f_n -f\|_{L^1} = 0$. Furthermore, we may assume that each Fourier transform $\cF f_n$ has compact support. Set
\[
  F_n (t) := \int_\R T(t+s) f_n(s) \, \ud{s},  \qquad n\geq 2, \, t\in\R .
\]
We have, by Parseval's formula,
\begin{align*}
 F_n (t) &= \lim_{\alpha\to0+} \int_0^\infty e^{-\alpha s} T(s) f_n(s-t) \, \ud{s} \\
 & = \lim_{\alpha\to0+} \frac{1}{2\pi} \int_\R (\alpha+i\xi +A)^{-1} \cF f_n (-\xi ) e^{i\xi t} \, \ud\xi  \\
 &= \frac{1}{2\pi} \int_\R (i\xi +A)^{-1} \cF f_n (-\xi ) e^{i\xi t} \, \ud\xi, \qquad t\in\R.
\end{align*}
Here the function under the integral is well defined since $\cF f_n$ vanishes on a neighbourhood of $E$. Moreover, the function $\xi \mapsto (i\xi+A)^{-1} \cF f_n (-\xi )$ is continuous and has compact support. Hence, by the Riemann-Lebesgue theorem,
\begin{equation} \label{rieleb}
 \lim_{|t|\to\infty} \| F_n (t)\| = 0, \qquad n\geq 2 .
\end{equation}
Since
\[
 \lim_{n\to\infty} \sup_{t\in\R} \| F_n (t) -F_0 (t) \| = 0 ,
\]
we therefore obtain
\begin{equation} \label{F00}
 \lim_{|t|\to\infty} \| F_0 (t)\| = 0 .
\end{equation}

Let us now examine the function $F_1$. Take $x \in X$.  The function $t \mapsto e^{-t} T(t)x$ is in $L^2(\R_+;X)$.  Extending this function by zero on $(-\infty,0)$,  Plancherel's theorem implies that 
\begin{equation} \label{f11}
 \int_\R \|(1+i\xi +A)^{-1}x\|^2 \, \ud\xi \le C^2 \|x\|^2.
\end{equation}
For $\alpha \in (0,1)$, let
$$
F_{1,\alpha,x}(t) = \int_\R e^{-\alpha(t+s)} T(t+s)x \, \ud\mu_1(s), \qquad t\in\R.
$$
Then 
\begin{equation} \label{f13}
\lim_{\alpha\to0+} F_{1,\alpha,x}(t) = F_1(t)x.
\end{equation}
Moreover, $F_{1,\alpha,x} \in L^1(\R;X)$ and its Fourier transform is
\begin{eqnarray*}
\lefteqn{(\cF F_{1,\alpha,x})(\xi)} \\
 &=& \cF \mu(\xi) (1- \cF \varphi(-\xi)) (\alpha+i\xi+A)^{-1}x  \\
 &=& \cF \mu(\xi) (1- \cF \varphi(-\xi)) \left(I + (1-\alpha) (\alpha+i\xi+A)^{-1} \right) (1+i\xi+A)^{-1}x,
\end{eqnarray*}
by the resolvent identity.   The assumption \eqref{nosinginf} extends by the Neumann series to boundedness of $\|(\alpha+i\xi+A)^{-1}\|$ for small $\alpha>0$ and $|\xi|\ge R$, and then for all $\alpha>0$ and $|\xi|\ge R$ since the semigroup is bounded.  Since $\cF\varphi = 1$ in a neighbourhood of $E$, there is a constant $C$ (independent of $\alpha$, $\xi$ and $x$) such that 
\begin{equation} \label{f12}
\|(\cF F_{1,\alpha,x})(\xi)\| \le C \|(1+i\xi+A)^{-1}x\|, \qquad 0< \alpha < 1,\; \xi \in \R.
\end{equation}
Moreover,
\begin{eqnarray*} 
\lefteqn{\lim_{\alpha\to0+} (\cF F_{1,\alpha,x})(\xi)} \\
&=& \begin{cases} \cF \mu(\xi) (1- \cF \varphi(-\xi))(i\xi+A)^{-1}x  &  \text{if $\cF \varphi(-\xi) \ne 1$}, \\ 0 &\text{otherwise}. \end{cases}  \\
&=:& G_{1,x}(\xi).\nonumber 
\end{eqnarray*} 
Using \eqref{f11}, \eqref{f12}, and the Dominated Convergence Theorem, it follows that $G_{1,x} \in L^2(\R;X)$, $\|G_{1,x}\|_{L^2(\R;X)} \le C\|x\|$ and $$
\lim_{\alpha\to0+} \|\cF F_{1,\alpha,x} - G_{1,x}\|_{L^2(\R;X)} = 0.
$$
Using Plancherel's theorem again shows that 
$$
\lim_{\alpha\to0+} \|F_{1,\alpha,x} - \cF^{-1}G_{1,x}\|_{L^2(\R;X)} = 0,
$$
where $\cF^{-1}G_{1,x}$ is inverse $L^2$-Fourier transform of $G_{1,x}$.  
From this and \eqref{f13}, we deduce that
\begin{equation} \label{f135}
F_1(t)x = (\cF^{-1}G_{1,x})(t)
\end{equation}
for almost all $t>0$. By Plancherel's theorem once more,
\[
 \int_\R \| F_1(t)x \|^2 \ \ud{t} \leq C^2 \, \| x\|^2, \qquad x\in X .
\]

Let $t\geq 0$ and $x\in X$. We compute:
\begin{eqnarray*}
\lefteqn{\int_0^t T(t-s) F_1 (s)x \, \ud{s}} \\ 
 &=& \int_0^t T(t-s) \int_\R T(s+r)x \, \ud\mu_1 (r) \, \ud{s} \\
& =& \int_0^t \int_\R T(t+r)x \, \ud\mu_1 (r) \, \ud{s} - \int_{0}^t \int_{-t}^{(-s)-} T(t+r)x \, \ud\mu_1 (r) \, \ud{s}\\
& =& t \, F_1 (t)x + \int_{-t}^{0} r T(t+r)x \, \ud\mu_1 (r) ,
\end{eqnarray*}
so that
\begin{equation} \label{f14}
F_1 (t)x =  \frac{1}{t} \int_0^t T(t-s) F_1 (s)x \, \ud{s} - \int_{-t}^{0}  \frac{r}{t} T(t+r)x \, \ud\mu_1 (r) .
\end{equation}
We estimate the first term on the right-hand side:
\begin{equation} \label{f15}
 \Big\| \frac{1}{t} \int_0^t T(t-s) F_1 (s)x \, \ud{s} \Big\| \leq \frac{1}{t} \int_0^t K \, \| F_1 (s)x \| \, \ud{s}  \leq K C \, \frac{1}{\sqrt{t}} \, \| x\|,
\end{equation}
where $K = \sup_{t\ge0} \|T(t)\|$.  For the second term on the right-hand side of \eqref{f14}, we note that
\begin{equation*} \label{f16}
 \lim_{t\to\infty}  \int_{-t}^{0}  \left\| \frac{\tau}{t} T(t+\tau) \right\| \, \ud\mu_1 (\tau) = 0,
\end{equation*}
by the Bounded Convergence Theorem for the bounded measure $\mu_1$.
Thus we have
\[
\lim_{t\to\infty} \| F_1 (t) \| = 0.
\]
From this and \eqref{F00}, we obtain
\[
 \lim_{t\to\infty} \| F(t) \| = \lim_{t\to\infty} \| T(t)\hat \mu(T)\| = 0 ,
\]
which is the claim.
\end{proof}

Now we return to the situation of Subsection \ref{ss.decay0}.  We assume that $\sigma (A) \cap i\R = \{ 0\}$ and that $\sup_{|s|\geq \eta} \| (is +A)^{-1}\| < +\infty$ for some $\eta>0$.  Since one-point sets are of spectral synthesis, we can apply Theorem \ref{KT+} with $\mu\in M^b (\R_+ )$ given by
\begin{equation} \label{mufor0}
 \mu = \delta_0 -e_1, \qquad    e_1(s) = e^{-s}, \quad s\ge0.
\end{equation}
Then $\cF\mu(0)=0$ and $\hat \mu(T) = I - (I+A)^{-1} = A(I+A)^{-1}$, so the conclusion of Theorem \ref{KT+} is that 
\[
 \lim_{t\to\infty} \| T(t) A (I+A)^{-1} \| = 0 .
\]
The following result includes an estimate of the rate of decay of $\|T(t)A(I+A)^{-1}\|$ in terms of the growth of  $\|(is+A)^{-1}\|$ as $s \to 0$.

\newcommand{\tm}{\tilde m}

\begin{theorem}\label{mlog}
Let $(T(t))_{t \ge 0}$ be a $C_0$-semigroup on a Hilbert space $X$ with generator $-A$, and assume that $\sigma(A) \cap i\mathbb R = \{0\}$ and that $\sup_{|s|\geq \eta} \| (is +A)^{-1}\| < +\infty$ for some $\eta>0$.
Then
\begin{equation} \label{uniform00}
 \lim_{t\to\infty} \| T(t) A (I+A)^{-1} \| = 0 .
\end{equation}
More precisely, let $m : (0,1) \to (1,\infty)$ be a continuous increasing function such that
\begin{equation*}\label{0resolvpolynom1}
\| (is+ A)^{-1}\| = \O(m(|s|))  \qquad s \to 0.                        
\end{equation*}
Let $\varepsilon\in(0,1)$.  Then
\begin{equation}\label{polydec01}
\|T(t) A(I+A)^{-1} \|=\O \, \left(  m^{-1}(t^{1-\varepsilon}) \right), \qquad t \to \infty.
\end{equation}
\end{theorem}

\begin{proof}
We have already shown how \eqref{uniform00} follows from  Theorem \ref{KT+}.  To establish \eqref{polydec01}, we follow the proof of that theorem with $\mu$ given by \eqref{mufor0} and $E = \{0\}$.  We take $\varphi \in \mathcal{S}(\R)$ so that $\cF \varphi$ has compact support and $\cF \varphi = 1$ on $[-1,1]$. Let
$$
\mu_0 = \mu * \varphi, \qquad \mu_1 = \mu - \mu_0.
$$
Then $\mu_0$ is absolutely continuous with density function
$$ 
f(s) = \varphi(s) - \int_0^\infty \varphi(s-\tau) e^{-\tau} \, \ud{\tau},
$$
and
$$
\cF f(\xi) = \frac {i\xi}{1+i\xi} \, \cF \varphi(\xi).
$$
Take a $C^\infty$-function $\psi$ such that $0 \le \psi \le 1$, $\psi = 1$ on $[-1,1]$ and the support of $\psi$ is contained in $[-2,2]$.  For $0 < r \le 1/2$, let $g_r$ be the Schwartz function such that
$$
\cF g_r(\xi) = \cF f(\xi) \psi_r(\xi) = \left( \frac{i\xi}{1+i\xi} \right) \psi \left( \frac{\xi}{r} \right).
$$
The last equality holds because $\psi(\xi/r)=0$ if $|\xi|\ge1$ and $(\cF \varphi)(\xi) = 1$ if $|\xi|\le1$.  Then $\cF g_r(\xi) = 0$ if $|\xi| \ge 2r$, and, for $|\xi| \le 2r$ and $j\ge1$,
\begin{eqnarray*}
(\cF g_r)^{(j)}(\xi) &=& \sum_{n=0}^{j-1}  \frac{(-1)^{j-n+1}i^{j-n}j!}{(1+i\xi)^{j-n+1}n!} \frac {1}{r^n} \psi^{(n)} \left( \frac{\xi}{r} \right) +   \left(\frac{i\xi}{1+i\xi}\right) \frac{1}{r^j} \psi^{(j)}\left(\frac{\xi}{r}\right).
\end{eqnarray*}
Hence
\begin{alignat*}{3}
|\cF g_r(\xi)|  &\le |\xi|, &\|\cF g_r\|_{L^1} &\le 4r^2, \\
|(\cF g_r)^{(j)}(\xi)| &\le \frac {C_j}{r^{j-1}} \left(1 + \frac{|\xi|}{r} \right), &\qquad \big\|(\cF g_r)^{(j)}\big\|_{L^1} &\le \frac{8C_j}{r^{j-2}}.
\end{alignat*}
Now, for $s \in \R$,
\begin{eqnarray*}
|g_r(s)| &=& \left| \frac{1}{2\pi} \int_{-2r}^{2r} \cF g_r(\xi) e^{i\xi s} \, \ud{\xi} \right| \le \frac{1}{2\pi} \int_{-2r}^{2r} |\xi|\, \ud\xi = \frac{2r^2}{\pi}, \\
|s^2g_r(s)| &=& \left| \frac{1}{2\pi} \int_{-2r}^{2r}  \left(\cF g_r\right)''(\xi) e^{i\xi s} \, \ud{\xi} \right|   \le  C, 
\end{eqnarray*}
so
$$
\|g_r\|_{L^1} \le \int_{|s|\le1/r} \frac{2r^2}{\pi} \, \ud{s}   + \int_{|s|\ge1/r} \frac{C}{s^2} \, \ud{s} \le Cr.
$$
Let $f_r = f - g_r$.  Then $f_r \in L^1(\R)$, $\cF f_r = 0$ on $[-r,r]$, and 
$$
\|f - f_r\|_{L^1} \le Cr,  \quad \|\cF f_r\|_{L^1} \le C, \quad \|(\cF f_r)^{(j)}\|_{L^1} \le \begin{cases} C_j &j=0,1, \\ \noalign{\vskip5pt} \dfrac{ C_j}{r^{j-2}}, \quad& j\ge2. \end{cases} \\
$$
These functions $f_r$ (as $r\to0+$) replace the functions $f_n$ (as $n\to\infty$) in the proof of Theorem \ref{KT+}. Accordingly, we define
$$
F_r(t) = \int_\R  T(t+s) f_r(s) \, \ud{s}, \qquad 0<r<1/2, \, t>0.
$$
Instead of applying the Riemann-Lebesgue theorem as we did in the proof of Theorem \ref{KT+} to obtain \eqref{rieleb}, we integrate by parts $k$ times, and we obtain that
\begin{align*} \label{fr1}
F_r(t) &=  \frac {i^k}{2\pi t^k} \int_\R \left(\frac {\ud}{\ud\xi}\right)^{k} \left( (i\xi + A)^{-1} \cF f_r(-\xi) \right) e^{i\xi t} \, \ud{\xi} \\
&= \frac {1}{2\pi i^k t^k} \int_\R \sum_{j=0}^k \begin{pmatrix} k \\ j \end{pmatrix} i^{k-j} (k-j)!  (i\xi+A)^{-(k-j+1)} (\cF f_r)^{(j)}(-\xi) e^{i\xi t} \, \ud{\xi}. \nonumber
\end{align*}
Since $m(r) \ge r^{-1} \ge 2$, we have
$$
\|F_r(t)\| \le \frac{C}{t^k} \left(m(r)^{k+1} + m(r)^k + \sum_{j=2}^k \frac{m(r)^{k-j+1}}{r^{j-2}} \right) \le \frac {Cm(r)^{k+1}}{t^k}.
$$
Now 
$$
\|F_0(t)\| \le \|F_r(t)\| + K \|f-f_r\|_{L^1} \le C \left( \frac{m(r)^{k+1}}{t^k} + r  \right),
$$
where $K = \sup_{t\ge0} \|T(t)\|$.  

Now we want to show that 
\begin{equation}\label{oneovert}
\|F_1(t)\|= \O(t^{-1}), \qquad t \to \infty.
\end{equation}
The estimate \eqref{f15} is not adequate for this purpose, and we use a different argument.

Let $x \in X$.  First note that the resolvent identity yields
\begin{equation}\label{in1}
\|(i\xi + A)^{-1}x\|\le (C+1) \|(i\xi +1 + A)^{-1}x\|, \qquad |\xi| \ge 1,
\end{equation}
where $C:=\sup_{|\xi| \ge 1} \|(i\xi +A)^{-1}\|$. Similarly,
\begin{equation}\label{in2}
\|(-i\xi + A^*)^{-1}x\|\le (C+1)\|(-i\xi +1 + A^*)^{-1}x\|, \qquad |\xi|\ge 1.
\end{equation}

Set
$$
g(\xi):=\cF \mu (\xi)(1-\cF \varphi(-\xi)) = \frac{i\xi}{1+i\xi} \big(1- \cF\varphi(-\xi)\big), \qquad \xi \in \mathbb R,
$$
and note that $g$ is zero on $(-1,1)$ and $g' \in L^1(\R) \cap L^\infty(\R)$.  Let $ x \in X$.  It follows from \eqref{f135} that $F_1(\cdot)x$ is the inverse Fourier transform (in the $L^2$-sense) of the function $G_{1,x}$ which has derivative
\begin{equation} \label{g1x'}
G_{1,x}'(\xi) = - g(\xi) (i\xi+A)^{-2}x  + g'(\xi) (i\xi+A)^{-1}x.
\end{equation}
Thus $G_{1,x}' \in L^2(\R;X)$, and $G_{1,x}$ belongs to the first-order vector-valued Hilbert-Sobolev space $H^1(\R;X)$.  Hence 
$$
t F_1(t)x = - i \cF^{-1}(G_{1,x}')(t).
$$
Since $g' \in L^1(\R)$, the second term on the right-hand side of \eqref{g1x'} is in $L^1(\R,X)$, and its inverse Fourier transform is bounded by $(2\pi)^{-1} \|g'\|_1 C \|x\|$.   To handle the first term, take $y \in X$.  By the Cauchy-Schwarz inequality, \eqref{in1} and \eqref{in2} and Plancherel's theorem,
\begin{eqnarray*}
\lefteqn{\frac{1}{2\pi}\int_{\mathbb R} \left| \langle g(\xi)(i\xi + A)^{-2}x,y\rangle\,\right| \ud\xi} \\
&\le& \frac{\|g\|_\infty}{2\pi} \left(\int_{\mathbb R\setminus (-1,1)} \|(i\xi+A)^{-1}x\|^2 \, \ud\xi \right)^{1/2}  \left(\int_{\mathbb R\setminus (-1,1)} \|(-i\xi+A^*)^{-1}y\|^2 \, \ud\xi \right)^{1/2}  \\
&\le& \frac{(C+1)^2\|g\|_\infty}{2\pi} \left(\int_{\mathbb R} \|(1+i\xi+A)^{-1}x\|^2 \, \ud\xi \right)^{1/2}  \left(\int_{\mathbb R} \|(1-i\xi+A^*)^{-1}y\|^2 \, \ud\xi \right)^{1/2} \\
&=& {(C+1)^2\|g\|_{\infty}} \left(\int_0^\infty \| e^{-\tau} T(\tau) x\|^2\, \ud{\tau} \right)^{1/2} \left( \int_0^\infty \|e^{-\tau} T^*(\tau) y\|^2\, \ud{\tau} \right)^{1/2}\\
&\le & {(C+1)^2 \|g\|_\infty K^2} \|x\| \,\|y\|.
\end{eqnarray*}
It then follows that
$$
t \left|\langle F_1(t)x, y \rangle\right|  \le  {(C+1)^2 \|g\|_\infty K^2} \|x\| \,\|y\|  +  \frac{\|g'\|_1 C}{2\pi} \|x\| \,\|y\|
$$
for almost all $t$.  Since $F_1$ is continuous, this estimate holds for all $t$, so \eqref{oneovert} follows.  

Overall we have that
$$
\|T(t)A(1+A)^{-1}\| = \|F_0(t) + F_1(t)\|  \le C \left(\frac {m(r)^{k+1}}{t^k} + \frac{1}{t} + r \right).
$$
For a given $\varepsilon \in (0,1)$, we take  $r = m^{-1}(t^{1-\varepsilon})$, for sufficiently large $t$.  Then we obtain
$$
\|T(t)A(1+A)^{-1}\|  \le  C \left( t^{1-(k+1)\varepsilon} + t^{-1} + m^{-1}(t^{1-\varepsilon}) \right).
$$
We may choose $k$ so that $(k+1)\varepsilon \ge 2$.  Since $m(r) \ge r^{-1}$, we have $m^{-1}(t^{1-\varepsilon}) \ge t^{\varepsilon-1} \ge t^{-1}$.  Then we obtain \eqref{polydec01}.   \end{proof}
\smallskip
  
\begin{example} \label{ex.KT+}
(a) Consider the case when $m(s) = s^{-\alpha}$  where $\alpha\ge1$.  Then \eqref{polydec01} gives 
$$
\|T(t)A(I+A)^{-1}\| = \O \left( t^{-\gamma} \right), \qquad t \to \infty,
$$
for any $\gamma< 1/\alpha$.  We shall see in Theorem \ref{polydec00} that this also holds with $\gamma=1/\alpha$.

(b) Now consider the case when $m(s) = e^{\alpha/s}$ where $\alpha>0$.  Then
$$
m^{-1}(t^{1-\varepsilon}) = \frac{\alpha}{(1-\varepsilon) \log t} \,,
$$
so 
$$
\|T(t)A(I+A)^{-1}\| = \O \left((\log t)^{-1}\right) = \O \left( m^{-1}(t) \right), \qquad t\to\infty.
$$
This rate of decay is sharp in this case (Corollary \ref{lowbound0}).
\end{example}

\section{Singularity at zero: polynomial and regularly varying rates} \label{s.zero2}

In this section we consider the situation of Subsection \ref{ss.decay0} and Theorem \ref{mlog} in cases when $X$ is a Hilbert space and the function $m(1/s)$ is regularly varying.  For simplicity of presentation, we shall write $B(A) = A(I+A)^{-1}$.

Let  $(T(t))_{t \ge 0}$  be  a bounded $C_0$-semigroup with generator $-A$
such that $ \sigma(A) \cap i\R = \{0\}$ and 
\begin{equation*}
  \|(is+ A)^{-1}\|= \begin{cases} \O \left( \dfrac{1}{|s|^\alpha \ell(1/|s|)} \right), &\qquad s \to 0,  \\ \O(1), &\qquad |s| \to \infty, \end{cases}
\end{equation*}
where $\ell$ is slowly varying and monotonic.   
Then the optimal estimate \eqref{genest0} on the decay of $\|T(t)B(A)\|$ under the assumption of dB-symmetry of $\ell$ would be
$$
\|T(t)B(A)\|=\O \left( \frac{1}{(t\ell(t^{1/\alpha}))^{1/\alpha}} \right), \qquad t \to \infty.
$$
We shall show in Theorem \ref{thm.main0} that this estimate does hold if $\ell$ is increasing, i.e.,  $\|(is+A)^{-1}\|$ grows slightly slower than $|s|^{-\alpha}$ as $s\to 0$. Its proof requires a series of steps similar to those used in the case of a singularity at infinity in Section \ref{s.infinity}.

\subsection{Cancelling resolvent growth}

\begin{definition} \label{dfm}
Let $\beta \in (0,1]$, and let $\ell$ be a slowly varying function such that $g: s \mapsto s^{1-\beta} \ell(s)$ is increasing on $\R_+$.  Let $S_g$ be the Stieltjes function associated with $g$ (Example \ref{stfng}), so
$$
S_g(\lambda)= \int_{0+}^{\infty}\frac{ \ud \left(s^{1-\beta}\ell(s)\right)}{s+\lambda} \,, \qquad \lambda>0.
$$
Let
\begin{eqnarray} \label{fg-def}
f_g(\lambda)&:=&S_g(1/\lambda), \qquad \lambda >0, \\
f_m(\lambda) &:=&  \frac{f_g(\lambda)}{1+f_g(\lambda)}, \qquad \lambda >0.\label{fm}
\end{eqnarray}
\end{definition}

Since $f_g$ is a complete Bernstein function, $1/f_g$ is Stieltjes by Theorem \ref{sti-char}, and then $f_m={1}/(1+1/f_g)$ is also a complete Bernstein function. Thus the operator $f_m(A)$ is well-defined, either by Definition \ref{defbernop} or by the extended functional calculus of Theorem \ref{calculus} (see Remark \ref{calcrem}(\ref{calcremciii})). Moreover $f_m(A) = f_g(A)(I+f_g(A))^{-1}$, by the composition rule in Theorem \ref{hirsch}(\ref{hirii}).  In particular $f_m(A)$ is bounded (see Remark \ref{boundedcbf} and \cite[Corollary 12.7]{SSV10}).

Now we make a definition analogous to \eqref{stw2}.

\begin{definition} \label{dvabl}
For $\alpha \ge 1$ and $\beta \in (0,1],$ define
\begin{equation*}\label{0stw2}
V_{\alpha,\beta,\ell} (A):=B(A)^{\alpha -\beta} f_m(A)= B(A)^{\alpha -\beta} f_g(A)(I+f_g(A))^{-1},
\end{equation*}
so that $V_{\alpha,\beta,\ell} (A) \in \mathcal{L}(X)$ by the above.
\end{definition}

The next statement shows that this operator cancels resolvent growth in the case of a singularity at zero. It is a counterpart of Theorem \ref{thm.bounded} dealing with singularities at infinity.  In the proof we again use Karamata's Theorem \ref{karamata}, Proposition \ref{prop.asymp} on domination properties of complete Bernstein functions and complex analysis arguments shown in Lemma \ref{phragmen}.

\begin{theorem} \label{thm.bounded0}
Let  $(T(t))_{t \ge 0}$  be  a bounded $C_0$-semigroup on a Banach space $X$, with generator $-A$ such that
\begin{equation}\label{0001}
 0 \in \sigma(A) \subset \C_+ \cup \{0\} \quad \text{and} \quad \|(is+A)^{-1}\|=\O(1), \quad |s| \to \infty.
\end{equation}
Let $\ell$ be a slowly varying function on $\mathbb R_+$ such that $g:s\mapsto s^{1-\beta}\ell(s)$ is increasing for some $\beta \in (0,1]$.  
If there exists $\alpha > 1$ such that
 \begin{equation}\label{0boundedd1}
  \|(is+ A)^{-1}\|=\O \left( \frac{1}{|s|^\alpha \ell(1/|s|)} \right), \qquad s \to 0,
 \end{equation}
 then
 \begin{equation*}\label{estimwithv}
 \sup_{\lambda \in \mathbb C_+}\|(\lambda+A)^{-1}V_{\alpha, \beta, \ell}(A)\|<\infty.
 \end{equation*}
\end{theorem}

\begin{remark}\label{alpha1}
We are interested only in the case when $0 \in \sigma(A)$, so $\|(is+ A)^{-1}\|\ge |s|^{-1}$.  Thus the assumption that $\alpha > 1$ restricts generality only slightly.  Moreover, if $\alpha=1,$ then the function $\ell$ can  be increasing only if it is bounded. In this case, \eqref{0boundedd1} is equivalent to the same estimate with $\ell\equiv 1$. This situation is settled by Corollary \ref{cor.polybound0} below. Hence if $\alpha=1$ we are interested only in the case when $\ell$ is decreasing and $\lim_{\lambda\to\infty} \ell(\lambda) = 0$, and then necessarily $\beta \in (0,1)$.   The following proof shows that Theorem \ref{thm.bounded0} is also true in that case, provided that $\sum_n \ell(2^n)$ converges.
\end{remark}

\begin{proof}
By the assumption \eqref{0001} and Lemma \ref{phragmen} it suffices to prove that
\begin{equation} \label{estimwithv2}
 \sup \left\{\|(is+A)^{-1}V_{\alpha,\beta,\ell}(A)\| \suchthat s\in\R, 0 < |s| \le 1 \right\} <\infty.
\end{equation}
To this aim, observe first that if $k\in\mathbb N$, then
\begin{equation} \label{0resident3}
(\lambda + A)^{-1} B(A)^k = \sum_{i=0}^{k-1}(- \lambda)^{k-1-i} A^i (I+A)^{-k} + (-\lambda)^k (\lambda + A)^{-1}(I+A)^{-k}.
\end{equation}

Let $\alpha-\beta = m + \gamma$ where $m \in \mathbb N \cup \{0\}$,  and $0 \le \gamma < 1$.  By the assumption \eqref{0boundedd1},
\begin{equation*} \label{0bounded-half}
\| A (is + A)^{-1}\|= \|I - is (is +A)^{-1}\| = \O \left(\frac{1}{|s|^{\alpha-1}\ell(1/|s|)} \right), \qquad s \to 0.
\end{equation*}
By the moment inequality \eqref{momineq3},
\begin{equation*} \label{0bounded-1}
\| A^{\gamma}(is+A)^{-1}\| = \O \left(\frac{1}{|s|^{\alpha-\gamma}\ell(1/|s|)} \right), \qquad s \to 0.
\end{equation*}
Using (\ref{0resident3}) for $k=m$, it follows that
\begin{eqnarray} \label{bounded11}
\| B(A)^{\alpha-\beta} (is + A)^{-1} \|
&=& \|A^\gamma  (is+A)^{-1}B(A)^{m} (I+A)^{-\gamma}\| \\
&=& \O \left(\frac{1}{|s|^{\alpha -\gamma-m}\ell(1/|s|)} \right) \nonumber \\
&=&
\O \left( \frac{1}{|s|^\beta \ell(1/|s|)} \right), \qquad s \to 0. \nonumber \label{0bounded-ses}
\end{eqnarray}

Let $f_g$ and $f_m$ be the complete Bernstein functions defined by \eqref{fg-def} and \eqref{fm}. Then
$$
\lim_{\lambda\to0+} f_g(\lambda)= \lim_{\lambda\to0+} f_m(\lambda) = 0,
$$
and in addition $f_m$ is bounded. Hence $f_m$ has Stieltjes representation of the form $(0,0,\nu)$.  Moreover, by \eqref{stasymp},
\begin{equation}\label{fg}
|f_g(s)|=\O\left(|s|^\beta \ell(1/|s|) \right), \qquad s \to 0.
\end{equation}
By Proposition \ref{prop.asymp} and \eqref{fg},
\begin{equation}\label{bratio}
|f_m(is)| \le C f_m(|s|) = \frac{C f_g(|s|)}{1+f_g(|s|)}=\O\left(|s|^\beta \ell(1/|s|) \right), \qquad s \to 0,
\end{equation}

Now $f_m(A)$ is a bounded operator, and, by \eqref{defbernop},  for every $x \in  {\rm dom}\,(A)$,
\begin{equation*}\label{representberns}
f_m(A)x= \int_{0+}^{\infty} A (\lambda +A)^{-1}x \, \ud{\nu}(\lambda).
\end{equation*}
We shall show that there exists $C>0$ such that
\begin{equation}\label{estimateondom}
\|(is+A)^{-1}B(A)^{\alpha-\beta}f_m(A)x\|\le C \|x\|, \qquad x \in \dom (A),\;0 < |s| \le 1.
\end{equation}
Since ${\rm dom}(A)$ is dense in $X,$ this will imply that  \eqref{estimateondom} holds
for all $x \in X$, thus giving us \eqref{estimwithv2} and completing the proof.

If $x \in {\dom}(A)$ and $s \in \mathbb R\setminus\{0\}$, then
\begin{eqnarray*}
\lefteqn{(is+A)^{-1} f_m(A)x} \\
&=&  \int_{0+}^{\infty} (is + A)^{-1}  A (\lambda+A)^{-1}x \, \ud{\nu}(\lambda)\\
&=&  \int_{0+}^{\infty} \left( \frac {\lambda}{\lambda-is} (\lambda+A)^{-1}x - \frac{is}{\lambda-is} (is + A)^{-1}x\right)\, \ud{\nu}(\lambda)\\
&=&  f_m (-is)  (is + A)^{-1}x +
\int_{0+}^{\infty} \frac{\lambda}{\lambda-is} (\lambda+A)^{-1}x\, \ud{\nu}(\lambda).\\
\end{eqnarray*}
Thus
\begin{eqnarray} \label{3summand}
\lefteqn{}\\
\lefteqn{(is+A)^{-1} f_m(A)B(A)^{\alpha-\beta}x} \nonumber \\
&=& f_m (-is)  (is + A)^{-1} B(A)^{\alpha-\beta}x +
\int_{0+}^{1} \frac{\lambda}{\lambda-is} (\lambda+A)^{-1} B(A)^{\alpha-\beta}x\, \ud{\nu}(\lambda) \nonumber\\
&& \phantom{xxxx}+ \int_{1+}^{\infty} \frac{\lambda}{\lambda-is} (\lambda+A)^{-1} B(A)^{\alpha-\beta}x \, \ud{\nu}(\lambda). \nonumber
\end{eqnarray}
we shall estimate each of the three summands above separately.

To bound the first summand, note that by  \eqref{bounded11} and  \eqref{bratio}, there exists $C>0$ such that
\begin{equation*}\label{secondest}
 |f_m(-is)| \, \| (is+A)^{-1}B(A)^{\alpha-\beta} x\|\le C \|x\|, \qquad x\in X,\; 0<|s| \le 1.
\end{equation*}

To estimate the second summand in \eqref{3summand},
we take into account that $\alpha >\beta$ and consider the following two cases.

If $\alpha -\beta \in (0,1)$ then (\ref{ba}) yields
$$
\|B(A)^{\alpha-\beta}(\lambda+A)^{-1}\| = \|(I+A)^{\beta-\alpha} A^{\alpha-\beta} (\lambda+A)^{-1}\| \le C\lambda^{\alpha-\beta-1},
$$
by Proposition \ref{prop.momineq}.  Then we obtain 
\begin{equation*}
\left \| \int_{0+}^{1} \frac{\lambda}{\lambda-is}B(A)^{\alpha-\beta} (\lambda+A)^{-1} x\, \ud{\nu}(\lambda)\right \|
\le C \int_{0+}^{1} \lambda^{\alpha-\beta-1}\, \ud{\nu}(\lambda) \|x\|.
\end{equation*}
We have to show that the integral here is finite.  For this, let $I_n$ be the interval $(2^{-(n+1)},2^{-n}]$ for $n \in \N \cup \{0\}$.  Then
$$
\int_{I_n} \ud\nu(\lambda) \le \int_{I_n} \frac{2^{-(n-1)}}{2^{-n}+\lambda} \, \ud\nu(\lambda) \le 2 f_m(2^{-n}) \le C 2^{-n\beta} \ell(2^n),
$$
by \eqref{bratio}.  Hence
$$
\int_{I_n} \lambda^{\alpha-\beta-1} \, \ud\nu(\lambda) \le 2^{-n(\alpha-\beta-1)+1} \int_{I_n} \ud\nu(\lambda) \le C 2^{-n(\alpha-1)} \ell(2^n).
$$
For $\alpha>1$, it follows from \eqref{svf} with $\gamma = (\alpha-1)/2$ that $$
\int_{0+}^{1} \lambda^{\alpha-\beta-1} \, \ud\nu(\lambda) = \sum_{n=0}^\infty \int_{I_n} \lambda^{\alpha-\beta-1} \, \ud\nu(\lambda) < \infty.
$$
The same conclusion holds if $\alpha=1$ and $\sum_{n=0}^\infty \ell(2^n)$ converges.

If $\alpha-\beta\ge1$, then
$$
\|B(A)^{\alpha-\beta}(\lambda+A)^{-1}\| \le \|B(A)^{\alpha-\beta-1}\| \, \|(I+A)^{-1} \left(I - \lambda(\lambda+A)^{-1}\right)\| \le C,
$$
and
$$
\left \| \int_{0+}^{1} \frac{\lambda}{\lambda-is}B(A)^{\alpha-\beta} (\lambda+A)^{-1} x\, \ud{\nu}(\lambda)\right \|
\le C \int_{0+}^{1}  \ud{\nu}(\lambda) \|x\|.
$$
The integral here is finite, by \eqref{mmu}.

For the third summand in \eqref{3summand} we have for any $s \in \mathbb R$,
\begin{eqnarray*}
\left \| \int_{1+}^{\infty}  \frac{\lambda}{\lambda-is}B(A)^{\alpha-\beta} (\lambda +A)^{-1} x  \, \ud{\nu}(\lambda)\right \| &\le& K \int_{1+}^{\infty}\frac{ \, \ud{\nu}(\lambda)}{\lambda} \|B(A)^{\alpha-\beta}\| \|x\| \\
&=& C \|x\|,
\end{eqnarray*}
where $K = \sup_{t\ge0} \|T(t)\|$, and $C$ is finite by \eqref{mmu}.

Finally, summarizing all the estimates above we get \eqref{estimateondom}, and thus the proof of the theorem is complete.
\end{proof}

\subsection{Resolvent growth equal to $s^{-\alpha}$}

Here we consider the case when we assume that the resolvent grows (at most) like $s^{-\alpha}$. 
We note first a corollary from the arguments of the proof of Theorem \ref{thm.bounded0} that will be useful in the next subsection and also covers the case $\alpha =1, \ell \equiv 1$ of Theorem  \ref{thm.bounded0} excluded in Remark \ref{alpha1}.

\begin{corollary} \label{cor.polybound0}
Let  $(T(t))_{t \ge 0}$  be a bounded $C_0$-semigroup on a Banach space $X$ with generator $-A$.  Assume that $\sigma(A) \cap i\R = \{0\}$ and
there exists $\alpha \ge 1$ such that
 \begin{equation*}
  \|(is +A)^{-1}\|=
  \begin{cases} \O \left( |s|^{-\alpha} \right), \qquad& s \to 0,  \\
                \O(1), &|s|\to\infty.
 \end{cases}               
 \end{equation*}
 Then
 \begin{equation*}
 \sup_{\lambda \in \mathbb C_+}\|(\lambda + A)^{-1}B(A)^\alpha\|<\infty.
 \end{equation*}
\end{corollary}
\begin{proof}
The proof follows the same lines as the proof of Theorem \ref{thm.bounded0}, but it is much simpler.  It follows immediately from \eqref{0bounded-ses} (now for $\beta=0$ and $\ell=1$) and Lemma \ref{phragmen}.
\end{proof}

The following result, which is analogous to \cite[Theorem 2.4]{BoTo10}, gives the optimal result in the case of exactly polynomial growth of the resolvent for a singularity at zero.

\begin{theorem}\label{polydec00}
Let $(T(t))_{t \ge 0}$ be a bounded $C_0$-semigroup on a Hilbert space $X$ with generator $-A$.  Assume that $\sigma(A) \cap i\mathbb R = \{0\}$, and let $\alpha\ge1$. The following are equivalent:
\begin{enumerate}[\rm(i)]
\item  \label{resolvpolynom1}
$\displaystyle \| (is +A)^{-1}\|
=\begin{cases} \O\, (|s|^{-\alpha}),& \qquad s \to 0,\\
                              \O\, (1), & \qquad |s| \to \infty.
\end{cases}$
\vskip5pt
\item  \label{resolvpolynom2}
$\displaystyle \|T(t)B(A)^{\alpha}\|=\O \, (t^{-1}), \qquad t \to \infty$.
\vskip5pt
\item  \label{resolvpolynom3}
$\displaystyle \|T(t)B(A)\|=\O \, (t^{-1/\alpha}), \qquad t \to \infty$.
\end{enumerate}
\end{theorem}

\begin{proof}
By Corollary \ref{cor.polybound0}, the property (\ref{resolvpolynom1}) implies that
\begin{equation*}
\sup_{\lambda \in \mathbb C_+} \|(\lambda+A)^{-1} B(A)^{\alpha}\|<\infty.
\end{equation*}
Then Theorem \ref{thm.CRbound} yields (\ref{resolvpolynom2}). 

Properties (\ref{resolvpolynom2}) and (\ref{resolvpolynom3}) are equivalent by Lemma \ref{interpoldec}.  Property (\ref{resolvpolynom3}) implies (\ref{resolvpolynom1}) by Theorem \ref{resbound0}.
\end{proof}

\subsection{Resolvent growth slower than $s^{-\alpha}$}

Here we give a counterpart of Theorem \ref{thm.main}.

\begin{theorem}\label{thm.main0}
Let $(T(t))_{t\ge0}$ be a bounded $C_0$-semigroup on a Hilbert space $X$, with generator $-A$.  Assume that $\sigma(A) \cap i\R = \{0\}$ and
\begin{equation}\label{0boundinf}
 \|(is+A)^{-1}\|= \begin{cases} \O \left(\dfrac{1}{|s|^\alpha \ell(1/|s|)} \right), &s \to0, \\ \O(1), &|s| \to \infty, \end{cases}
 \end{equation}
 where $\alpha>1$ and  $\ell$ is increasing and slowly varying.  Then
\begin{equation*} \label{0mainest}
\|T(t)A(I+A)^{-1}\| = \O \left( \frac {1}{(t\ell(t^{1/\alpha}))^{1/\alpha}} \right), \qquad t\to\infty.
\end{equation*}
\end{theorem}

\begin{proof}  We can assume that $T(t)A\ne0$ for each $t>0$.  Since \eqref{0boundinf} implies property \eqref{resolvpolynom1} of Theorem \ref{polydec00}, we obtain from property \eqref{resolvpolynom3} of that theorem that
\begin{equation} \label{0BTest}
\|T(t)B(A)\| =\O \big( t^{-1/\alpha} \big), \qquad t \to \infty.
\end{equation}

We use the notation of Definitions \ref{dfm} and \ref{dvabl}, and apply Theorem \ref{thm.bounded0}, with $\beta=1$ and $g=\ell$.  We conclude that there exists $C>0$ such that
$$
\|(\lambda+A)^{-1} B(A)^{\alpha-1} f_m(A)\| \le C,  \qquad \lambda \in \mathbb{C_+}.
$$
Then, by Theorem \ref{thm.CRbound}, we have
$$
\|T(t) B(A)^{\alpha-1} f_{m}(A)\| \le \frac {C}{t}, \qquad t>0,
$$
Hence by Theorem \ref{interpol2}(\ref{int2b}), with $\gamma=\alpha-1$, and $t_1=t_2=t$,
$$
 f_{m}( \|T(t) B(A)\|) \le  \frac {C \|T(t) B(A) \|}{t \|T(2t) B(A)^{\alpha}\|} \,.
$$
Putting $f_{\alpha,m}(s) = s^{\alpha-1} f_{m}(s)$  we then obtain that
\begin{equation} \label{0fatest2}
 f_{\alpha,m}( \| T(t) B(A) \|) \le \frac {C}{t}  \frac {\|T(t) B(A) \|^\alpha}{\|T(2t) B(A)^{\alpha}\|} \,.
\end{equation}

By Theorem \ref{karamata} with $g=\ell$ and $\rho=\sigma=1$, $S_\ell(t) \sim t^{-1}\ell(t)$ as $t \to \infty$.  The function $f_{m}(s)$ has the same decay at zero as $f_\ell(s)$,  hence $f_{\alpha,m}(s) \sim s^\alpha \ell(1/s)$ as $s \to 0+$.  Let $k(s) = 1/\ell(s^{1/\alpha})$.  Then
$$
f_{\alpha,m}(s) \sim \frac {s^{\alpha}} {k(s^{-\alpha})}, \qquad s\to0+.
$$
By Proposition \ref{regvarinv}(\ref{rvi3}),
$$
f_{\alpha,m}^{-1}(s)  \sim \left(\frac {s} {k^\#(1/s)}\right)^{1/\alpha}, \qquad s\to0+.
$$
If the right-hand side of \eqref{0fatest2} is sufficiently small,
\begin{equation*} \label{0Lest}
\|T(t)B(A)\| \le  \frac { C\|T(t) B(A) \|}{(t\ell(t) \|T(2t) B(A)^{\alpha}\|)^{1/\alpha}} \,,
\end{equation*}
where
$$
L(t) = k^\# \left( t \frac {\|T(2t) B(A) \|}{\|T(t) B(A)\|^\alpha} \right).
$$
Let $\psi(s) = (s k^\#(s))^{1/\alpha}$.   Since $\psi$ is regularly varying with positive index, we can choose $k^\#$ so that $\psi$ is strictly increasing and continuous (see the remarks following Definition \ref{regvar}; in fact, we can choose $k^\#$ to be increasing, since $k$ is decreasing).  Then
\begin{eqnarray}  \label{0psiest}
\|T(t)B(A)\|^{-1} &\ge& c t^{1/\alpha} \, \frac {\|T(2t)B(A)^{\alpha}\|^{1/\alpha}}{\|T(t)B(A)\|} \, L(t)^{1/\alpha} \\
&=& c \, \psi \left(t \frac {\|T(2t)B(A)^{\alpha}\|}{\|T(t)B(A)\|^\alpha} \right). \nonumber
\end{eqnarray}
  If the right-hand side of \eqref{0fatest2} is bounded away from $0$, then \eqref{0psiest} also holds for some $c>0$, since $\|T(t)B(A)\|$ is bounded and $\psi$ is bounded on bounded intervals.  So \eqref{0psiest} holds for all $t>0$, for some $c>0$.

Since $\psi$ is strictly increasing,
$$
t \frac {\|T(2t)B(A)^{\alpha}\|}{\|T(t)B(A)\|^\alpha} \le \psi^{-1} \left( {C}{\|T(t)B(A)\|^{-1}} \right),
$$
for some constant $C>0$.
By Proposition \ref{regvarinv}(\ref{rvi2}),
$$
\psi^{-1} (s) \sim s^\alpha k^{\#\#}(s^\alpha) \sim \frac {s^\alpha}{\ell(s)}, \qquad s \to \infty,
$$
so
\begin{equation*} 
\| T(2t)B(A)^{\alpha} \| \le \frac {C} {t \ell \left( \|T(t)B(A)\|^{-1} \right)} \,.
\end{equation*}
Then (\ref{0BTest}) yields
$$
\|T(2t)B(A)^{\alpha}\|  \le \frac {C}{t\ell(t^{1/\alpha})} \,,
$$
for sufficiently large $t$.  Since $B(A)$ is sectorial, we can apply Lemma \ref{interpoldec} with $B=B(A), \gamma=\alpha$ and $\delta=1$ and we obtain
$$
\|T(t)B(A)\|  \le \frac {C}{(t\ell(t))^{1/\alpha}},
$$
for large $t$, and the proof is finished.
\end{proof}

Now we can formulate a counterpart of Corollary \ref{cor.main}, showing that the optimal estimate \eqref{genest0} holds when $\ell$ is increasing and dB-symmetric.

\begin{corollary}
In addition to the assumptions of Theorem \ref{thm.main0}, assume that $\ell$ is dB-symmetric.  Then
$$
\|T(t)B(A)\| = \O \left( m^{-1}(t) \right), \qquad t \to \infty,
$$
where $m^{-1}$ is any asymptotic inverse of $s^\alpha/\ell(s)$.
\end{corollary}

\begin{proof}
The proof is completely analogous to the proof of Corollary \ref{cor.main}.
\end{proof}

The discussion following the proof of Corollary \ref{cor.main} and addressing the situation when $\ell$ may not be dB-symmetric applies to the current setting as well.  We omit the details.

\begin{remark}
Note the following duality between singularities at zero and at infinity. If $(T(t))_{t \ge 0}$ is a $C_0$-semigroup of contractions on a Hilbert space $X$ with generator $-A$ and the range of $A$ is dense in $X$, then $A$ is injective and the operator $-A^{-1}$, with dense domain $\ran(A)$, also generates a $C_0$-semigroup of contractions on $X$, by a direct application of the Lumer-Phillips theorem.  Thus the case of a singularity at infinity for $A$ corresponds to the case of a singularity at zero for $A^{-1}$.  Conversely, the case of a singularity at zero for $A$ corresponds, in general, to the case of singularities at both zero and at infinity for $A^{-1}$. This case will be studied in Section \ref{sectwosing}.
\end{remark}

\begin{remark} We are unable to give a result corresponding to Theorem \ref{thm.faster} because of the lack of an initial estimate for the rate of decay that is sufficiently close to \eqref{genest0} for our technique to work.  When $\|(is+A)^{-1}\|$ grows slightly faster than $|s|^{-\alpha}$, our best estimate is 
$$
\|T(t)B(A)\| = \O \big(t^{-1/\beta} \big), \qquad t\to\infty,
$$
for each $\beta>\alpha$ (see Example \ref{ex.KT+}(a) and Theorem \ref{polydec00}).  Our technique does not improve this.
\end{remark}

\section{Singularities at both zero and infinity}\label{sectwosing}

In this section we study the rates of decay in the context of Corollary \ref{convtozero} and similar situations, so we are interested in the property
\begin{equation} \label{uniform2}
\lim_{t\to\infty} \|T(t) A(I+A)^{-2}\| = 0,
\end{equation}
for a bounded $C_0$-semigroup $(T(t))_{t \ge0}$ on a Banach space or Hilbert space.  Note that the range of $A(I+A)^{-2}$ is $\ran(A) \cap \dom(A)$ (Proposition \ref{prop.aab}(\ref{domim})).   Thus we consider orbits  starting in $\ran(A)\cap \dom(A)$ or similar spaces.   

We shall assume that \eqref{uniform2} holds, or equivalently that $\sigma(A) \cap i\R \subset \{0\}$ (see Corollary \ref{convtozero}).  
We wish to examine the relation between the rate of decay in \eqref{uniform2} and the rate of growth of $\|(is+A)^{-1}\|$ as $s \to 0$ or as $|s| \to \infty$.  Let $M:[1,\infty)\to \mathbb R_+$ be a continuous increasing function, and  $m:(0,1]\to \mathbb R_+$ be a continuous decreasing function, such that
\begin{equation*} 
\|(is+A)^{-1}\| \le \begin{cases}m(s), &0<|s|<1,\\
M(s), & |s| \ge 1.\end{cases}  
\end{equation*}
We assume that $0 \in \sigma(A)$ and that $\lim_{s\to\infty} M(s) = \infty$ (otherwise we are in the situation of Section \ref{s.infinity}, or Sections \ref{s.zero1} and \ref{s.zero2}, respectively).  By standard theory, $m(s) \ge 1/s$.

For example, $M$ and $m$ can be defined by modified versions of \eqref{defM} and \eqref{defm}:
\begin{eqnarray}
M(s) &=& \sup \left\{\|(ir+A)^{-1}\| \suchthat 1 \le |r| \le s \right\}, \qquad s\ge1, \label{defM2} \\
m(s) &=&\sup \left\{\| (ir+A)^{-1}\|: s< |r| \le 1 \right\}, \qquad 0<s<1. \label{defm2}
\end{eqnarray}

Let $N_2$ be a continuous, decreasing function such that 
$$
\|T(t)A(I+A)^{-2}\| \le N_2(t), \qquad t\ge0.
$$
For example, we can take
$$
N_2(t) = \sup\{ \|T(\tau)A(I+A)^{-2}\| : \tau \ge t\}, \qquad t\ge0.
$$
By Corollary \ref{convtozero}, our assumptions imply that $\lim_{t\to\infty} N_2(t) = 0$.  Let $N_2^{-1}$ denote any right inverse for $N_2$.

Under our assumptions, Mart\'\i nez \cite[Corollary 3.3]{Ma11} has shown that 
\begin{equation} \label{martinez}
\|T(t) A (I+A)^{-2} \| \le C \max \left(m_{\log}^{-1}(c't), \frac {1}{M_{\log}^{-1}(C't)} \right),
\end{equation}
where $M_{\log}$ is defined by \eqref{defMlog} and 
$$
m_{\log}(s) = m(s) \log \left( \frac{1+m(s)}{s} \right), \qquad 0<s\le1.
$$
We have omitted a term $t^{-1}$ which appears in \cite{Ma11}, because $m^{-1}(t) \ge t$ when $0 \in \sigma(A)$.

The upper bound \eqref{martinez} is analogous to the upper bound for $\|T(t)A^{-1}\|$ given in Theorem \ref{chdu} and there is also a lower bound for $\|T(t)A(I+A)^{-2}\|$.   Theorem \ref{split} extends to our present situation, with an almost unchanged proof.   If $\limsup_{t\to\infty} t\|T(t)A(I+A)^{-2}\| = 0$, then $X$ splits as in Theorem \ref{split}(\ref{splitii}).

The following is an analogue of Theorem \ref{resbound0}, with a similar proof.  The term $|s|^{-1}$ in the estimate \eqref{resest0+} is relevant only when splitting occurs and then $\|(is+A)^{-1}\|$ is comparable to $|s|^{-1}$ when $|s|$ is sufficiently small.

\begin{theorem} \label{resbound0+}
Let $(T(t))_{t\ge0}$ be a bounded $C_0$-semigroup on a Banach space $X$, with generator $-A$.  Assume that $\sigma(A) \cap i\R = \{0\}$, and let $N_2$ be as above. Then, for any $c \in (0,1)$,
\begin{equation}\label{resest0+}
\|(is+A)^{-1}\| = \begin{cases} \O \left( N_2^{-1}(c|s|) + |s|^{-1} \right),  \qquad &s\to0,   \\ \O \left( N_2^{-1}\left(c|s|^{-1}\right)\right) &|s|\to\infty. \end{cases}
\end{equation}
\end{theorem}

\begin{proof} Let $K = \sup_{t\ge0} \|T(t)\|$.  Let $s \in \mathbb{R} \setminus \{0\}$, $t>0$ and $x \in \dom(A)$.  We use the formula
\begin{eqnarray*}
ise^{ist} x
 &=&  is e^{ist} \int_0^t e^{-is\tau} T(\tau) (is +A)x \, \ud\tau +is T(t)x \\
 &=& is e^{ist} \int_0^t e^{-is\tau} T(\tau) (is +A)x \, \ud\tau  + T(t)(1+A)^{-1}(is+A)x \\
 && \phantom{X}  - (1-is)T(t)A(I+A)^{-2} ((is+A) + (1-is))x.
\end{eqnarray*}
Since $\|T(t) (I+A)^{-1} \| \le K$, this gives
$$
|s| \, \|x\|  \le K (|s| t + 1)\|(is+A)x\|  + |1-is| N_2(t) \left(\|(is+A)x\| + |1-is| \|x\|\right).
$$
Hence
\begin{equation} \label{Nest0}
\left( |s| - |1-is|^2 N_2(t) \right) \|x\| \le  (K (|s|t + 1) + |1-is|N_2(t)) \|(is+A)x\|.
\end{equation}

Set $t = N_2^{-1}(c|s|)$.  For $|s|$ sufficiently small,
$$
|s| - |1-is|^2 N_2(t) = |s|\left(1 - c (1+s^2)\right) > 0.
$$
For any $K'>K$, \eqref{Nest0} gives
\begin{eqnarray*}
\|(is+A)^{-1}\| &\le& \frac{K(|s|N_2^{-1}(c|s|)+1) +|1-is|c|s|}{|s|(1-c(1+s^2))} \\
&\le& \frac{K'}{1-c} \left( N_2^{-1}(c|s|) + |s|^{-1} \right),
\end{eqnarray*}
for $|s|$ sufficiently small.

Now set $t = N_2^{-1}(c/|s|)$.  For $|s|$ sufficiently large,
$$
|s| - |1-is|^2 N_2(t) = |s|\left(1- c(1+s^{-2})\right) > 0.
$$
So we obtain
\begin{eqnarray*}
\|(is+A)^{-1}\| &\le& \frac{K \left(|s| N_2^{-1}(c|s|^{-1})+1 \right) + |1-is|c|s|^{-1}} {|s|\left(1- c(1+s^{-2})\right)} \\
&=& \O \left(N_2^{-1}(c|s|^{-1}) \right), \qquad |s|\to\infty.
\end{eqnarray*}
\end{proof}

As in Corollary \ref{lowbound0} and \eqref{genest-}, we obtain the following lower bound.

\begin{corollary} \label{lowbound0+}
Let $(T(t))_{t\ge0}$ be a bounded $C_0$-semigroup on a Banach space $X$, with generator $-A$.  Assume that $\sigma(A) \cap i\mathbb{R} = \{0\}$ and that 
\begin{equation*}
\lim_{s\to0} \max \left( \|s(is+A)^{-1}\|,  \|s(-is+A)^{-1}\| \right)=\infty.  
\end{equation*}
Define $M$ and $m$ by \eqref{defM2} and \eqref{defm2} respectively.  Then there exist $c,c',C'>0$ such that
$$
\|T(t)A(I+A)^{-2}\| \ge c \max \left(m^{-1}(c't), \frac {1}{M^{-1}(C't)} \right)
$$
for all sufficiently large $t$.
\end{corollary}

\begin{remark} \label{seifert} (a) For contraction semigroups, the assumption in Corollary \ref{lowbound0+} can be weakened in the same way as in Remark \ref{remlb0}.

(b) If $(T(t))_{t\ge0}$ is bounded and eventually differentiable, then \linebreak[4] $(I+A)^2T(\tau)$ is a bounded operator for some $\tau>0$.  Moreover, 
$$
AT(t+\tau) = (I+A)^2 T(\tau) T(t) A(I+A)^{-2},
$$
so $\lim_{t\to\infty} \|AT(t+\tau)\| = 0$ if $\sigma(A) \cap i\R = \{0\}$.  By Theorem \ref{resbound0}, $\|(is+A)^{-1}\|$ is bounded for $|s|\ge1$.  Then  \eqref{martinez} gives 
$$
\|AT(t)\| = \O\left(m_{\log}^{-1}(c't)\right).
$$
\end{remark}

Now we return to upper bounds, and we consider polynomially growing $m$ and $M$.  In this case we can omit the logarithmic terms in \eqref{martinez} in a similar way to Theorem \ref{polydec00}.

\begin{theorem}\label{poldecayzero}
Let $(T(t))_{t \ge 0}$ be a bounded $C_0$-semigroup on a Hilbert space $X$ with generator $-A$.  Assume that $\sigma(A) \cap i\R = \{0\}$ and that there exist $\alpha\ge1$ and  $\beta >0$ such that
\begin{equation}\label{resolvpolynom+}
\| (is+A)^{-1}\| = \begin{cases} \O \left(|s|^{-\alpha} \right), \qquad &s\to 0, \\
\O \left( |s|^{\beta} \right), &|s|\to\infty. \end{cases}
\end{equation}
Then
\begin{equation}\label{polynomdecayzero+}
\|T(t)A^{\alpha} (I+A)^{-(\alpha+\beta)}\|=\O \, (t^{-1}), \qquad t \to \infty,
\end{equation}
and
\begin{equation} \label{polynomdecay1+}
\|T(t) A (I+A)^{-2}\|= \O \left (t^{-1/\gamma} \right), \qquad t \to \infty,
\end{equation}
where $\gamma=\max(\alpha,\beta)$.

Conversely, if \eqref{polynomdecay1+} holds for some $\gamma>0$, then \eqref{resolvpolynom+} holds for $\alpha=\max(\gamma,1)$ and $\beta=\gamma$.
\end{theorem}

\begin{proof} 
The argument which establishes \eqref{bounded11} in the proof of Theorem \ref{thm.bounded0} (where one puts $\beta=0$ and $\ell\equiv1$) shows that 
\begin{equation*}
\sup \big\{\|(is+A)^{-1}(A (I+A)^{-1})^{\alpha}\|\suchthat 0<|s|\le1 \big\} < \infty.
\end{equation*}
Similarly, the argument that (\ref{boundedd1}) implies \eqref{bounded-ses} in the proof of Theorem \ref{thm.bounded} shows that
\begin{equation*}
\sup \big\{\|(is+A)^{-1} (I+A)^{-\beta} \|\suchthat |s|\ge1 \big\} < \infty.
\end{equation*}
Using product and composition rules similarly to \eqref{powerrules},
\begin{equation*}
\sup \big\{\|(is+A)^{-1}A^\alpha (I+A)^{-(\alpha+\beta)}\|\suchthat s\ne0 \big\} < \infty.
\end{equation*}
Now we can apply Lemma \ref{phragmen} to the function 
$$
f(z) = (z+A)^{-1} A^\alpha (I+A)^{-(\alpha+\beta)},
$$
and we deduce that
\begin{equation*}
\sup \big\{\|(\lambda+A)^{-1}A^\alpha (I+A)^{-(\alpha+\beta)}\|\suchthat \lambda\in\C_+ \big\} < \infty.
\end{equation*}
By Theorem \ref{thm.CRbound},
\begin{equation*}
\|T(t)A^\alpha (I+A)^{-(\alpha+\beta)}\|=\O(t^{-1}), \qquad t \to \infty.
\end{equation*}
Since $A^{\gamma-\alpha}(I+A)^{-(2\gamma-\alpha-\beta)}$ is a bounded operator,
\begin{equation*}
\|T(t)(A (I+A)^{-2})^\gamma\|=\O(t^{-1}), \qquad t \to \infty.
\end{equation*}
Since $A (I+A)^{-2}$ is sectorial, by Lemma \ref{interpoldec}
\begin{equation*}
\|T(t)A (I+A)^{-2}\|=\O(t^{-1/\gamma}), \qquad t \to \infty, 
\end{equation*}
for $\gamma=\max(\alpha,\beta)$.

The converse statement follows from Theorem \ref{resbound0+}.
\end{proof}

In Theorem \ref{poldecayzero}, the case $\gamma<1$ can occur only if the space $X$ splits as described before Theorem \ref{resbound0+}.

\begin{remark}
Proposition \ref{prop.aab} (\ref{domim}) and \eqref{polynomdecayzero+} show that $\|T(t)x\| = \O(t^{-1})$ for all $x \in \ran(A^\alpha) \cap \operatorname{dom}(A^\beta)$.  Moreover, Proposition \ref{prop.aab} (\ref{aabsect}) shows that $A^a(I+A)^{-1}$ is sectorial whenever $0<a<1$, in particular for $a = \alpha/(\alpha+\beta)$.  By Lemma \ref{interpoldec} one obtains that
$$
\|T(t) A^{\alpha\gamma}(I+A)^{-(\alpha\gamma+\beta\gamma)}\| = \O \left( t^{-\gamma} \right), \qquad t \to \infty,
$$
for every $\gamma>0$.  Hence $\|T(t)x\| = O(t^{-\gamma})$ for all $x \in \ran(A^{\alpha\gamma}) \cap \operatorname{dom}(A^{\beta\gamma})$. 
\end{remark}

We do not consider regularly varying rates of resolvent growth in the context of this section.  It is not easy to find a single operator which cancels  resolvent growth at both zero and infinity simultaneously, without losing information about the fine scale of the behaviour.

Finally we make some remarks about quasi-multiplication semigroups.  In many cases (including normal semigroups on Hilbert spaces, and multiplication semigroups on $L^p$-spaces) the space $X$ can be split into a direct sum of two closed invariant subspaces $X_0$ and $X_1$ so that the generator of the semigroup is bounded when restricted to $X_0$ and invertible when restricted to $X_1$.  Then the rate of decay on $X$ is the maximum of the rates on $X_0$ and $X_1$, so upper and lower bounds on $X$ can be deduced from those on $X_0$ and $X_1$.  However knowing only the rate of decay on $X$ is not sufficient to detect whether it is controlled on $X_0$ and $X_1$.  Consequently we do not think it is possible to formulate a succinct result  such as Propositions \ref{normalinf} and \ref{normalinf0} in terms of $\|(is+A)^{-1}\|$ in this case.  If $\|(is+A)^{-1}\|$ dominates $\|(is^{-1}+A)^{-1}\|$ whenever $|s|>1$, then the rate of decay is determined by the behaviour on $X_1$, and one can apply Proposition \ref{normalinf}.  If $\|(is^{-1}+A)^{-1}\|$ dominates $\|(is+A)^{-1}\|$ whenever $|s|>1$, then the behaviour on $X_0$ dominates and Proposition \ref{normalinf0} is applicable.

\end{document}